\documentclass[11pt]{article}
\usepackage{amsmath,amsfonts,amsthm,amssymb, authblk,color,fullpage}
\usepackage{url}
\usepackage{graphicx,amssymb,amsmath,amsthm}
\usepackage{enumerate}
\usepackage{dsfont}
\usepackage{booktabs}
\usepackage{mathrsfs}
\usepackage{verbatim}
\usepackage{epsfig}
\usepackage{lscape}
\usepackage{subfigure}
\usepackage[colorlinks, linkcolor=red, anchorcolor=blue, citecolor=green]{hyperref}
\usepackage{epstopdf}
\usepackage{color}
\usepackage{caption}
\usepackage{algorithm}
\usepackage{algpseudocode}

\usepackage{graphicx}
\usepackage{subfigure}
\usepackage{cleveref}
\usepackage{xcolor}
\graphicspath{{Figures/}{logo/}}
\usepackage{cite}

\newtheorem{definition}{Definition}[section]

\newtheorem{prop}{Proposition}[section]
\newtheorem{theorem}{Theorem}[section]
\newtheorem{lemma}{Lemma}[section]

\newtheorem{remark}{Remark}[section]

\newtheorem{problem}{Problem}[section]

\numberwithin{equation}{section}
\date{}

\makeatletter
\newcounter{phase}[algorithm]
\newlength{\phaserulewidth}
\newcommand{\setphaserulewidth}{\setlength{\phaserulewidth}}

\makeatother

\setphaserulewidth{.7pt}

\newcommand{\Addresses}{{
\bigskip
\footnotesize

Jian-Feng Cai, \textsc{Department of Mathematics, The Hong Kong University of Science and Technology, Clear Water Bay, Kowloon, Hong Kong SAR, China}\par\nopagebreak
\textit{E-mail address}, Jian-Feng Cai: \texttt{jfcai@ust.hk}

\medskip

Zhiqiang Xu, \textsc{State Key Laboratory of Mathematical Sciences, Academy of Mathematics and Systems Science, Chinese Academy of Sciences, Beijing 100190, China;    School of Mathematical Sciences, University of Chinese Academy of Sciences, Beijing 100049, China. }
\par\nopagebreak
\textit{E-mail address}, Zhiqiang Xu: \texttt{xuzq@lsec.cc.ac.cn}

\medskip

Zili Xu, \textsc{School of Mathematical Sciences, East China Normal University, Shanghai 200241, China;
Shanghai Key Laboratory of PMMP, East China Normal University, Shanghai 200241, China; and
Key Laboratory of MEA, Ministry of Education, East China Normal University, Shanghai 200241, China}\par\nopagebreak
\textit{E-mail address}, Zili Xu: \texttt{zlxu@math.ecnu.edu.cn}
}}

\begin{document}
\baselineskip 14pt
\bibliographystyle{plain}

\title{
Generalized Interlacing Families: New Error Bounds for CUR Matrix Decompositions
 }

\author{Jian-Feng Cai, Zhiqiang Xu, and Zili Xu}

\maketitle

\begin{abstract}

This paper introduces the concept of generalized interlacing families of polynomials, which extends the classical interlacing polynomial method to handle polynomials of varying degrees. We establish a fundamental property for these families, proving the existence of a polynomial with a desired degree whose smallest root is greater than or equal to the smallest root of the expected polynomial. Applying this framework to the generalized CUR matrix approximation problem, we derive a theoretical upper bound on the spectral norm of a residual matrix, expressed in terms of the largest root of the expected polynomial. We then explore two important special cases: the classical CUR matrix  decompositions and the row subset selection problem. For classical CUR matrix decompositions, we derive an explicit upper bound for the largest root of the expected polynomial. This yields a tighter spectral norm error bound for the residual matrix compared to many existing results. Furthermore, we present a deterministic polynomial-time algorithm for solving the classical CUR problem under certain matrix conditions. For the row subset selection problem, we establish the first known spectral norm error bound. This paper extends the applicability of interlacing families and deepens the theoretical foundations of CUR matrix decompositions and related approximation problems.

\end{abstract}

{
  \hypersetup{linkcolor=black}
  \tableofcontents
}


\section{Introduction}

The method of interlacing polynomials was originally introduced by Marcus, Spielman, and Srivastava \cite{inter1,inter2,inter3} in their resolution to the Kadison-Singer problem \cite{inter2}. 
Since its inception, this technique has found powerful applications across various areas of mathematics, including the construction of Ramanujan graphs \cite{inter1,inter4,GM21}, Bourgain-Tzafriri's restricted invertibility principle \cite{inter3,ravi1}, matrix discrepancy problem \cite{KLS20,XZZ22,Bow2023}, and asymmetric traveling salesman problem \cite{AG14}.

The key component in this method is to show that a given set of real-rooted degree-$d$ polynomials with positive leading coefficients forms an interlacing family. 
If these polynomials form an interlacing family, then for any $j\in\{1,2,\ldots,d\}$, there exists  a polynomial from the family whose $j$-th largest root is bounded below by that of the expected polynomial of this family.
This framework offers a powerful new approach for proving the existence and construction of combinatorial or algebraic objects whose associated expected polynomials have well-bounded roots.

This paper introduces a novel concept called {\em generalized interlacing families}. 
Unlike classical interlacing families, which are typically restricted to polynomials of the same degree, our generalized framework allows polynomials to have different degrees.
This relaxation of the degree constraint thus possesses a broader applicability, which enables its application to the CUR matrix decompositions.

Matrix approximation is a critical step in processing high-dimensional data. A standard technique is singular value decomposition (SVD), which produces low-rank approximations by using the top singular vectors. However, SVD cannot leverage the sparsity of the input matrix, and its components often lack interpretability. The CUR matrix decomposition addresses this by selecting actual rows and columns from the input matrix, producing more interpretable and structurally meaningful approximations.
In this paper, we focus on the CUR matrix decomposition problem. We apply the framework of generalized interlacing families to derive new upper bounds on the approximation error, highlighting the strength and flexibility of this new approach.

\subsection{CUR matrix decompositions}

We first introduce some notations.
For an integer $n$ we write $[n]:=\{1,\ldots,n\}$.
A subset $S\subset[n]$ is called a $k$-subset if its cardinality is $k$. 
We use $\binom{[n]}{k}$ to denote the set of all $k$-subsets of $[n]$, i.e.,
\begin{equation*}
\binom{[n]}{k}:=\{ S\subset[n]: |S|=k\}.	
\end{equation*}
For a matrix $ \boldsymbol{A}\in\mathbb{R}^{n\times d}$, we use $ \boldsymbol{A}_{S,W}$ to denote the submatrix of $ \boldsymbol{A}$ consisting of rows indexed in the set $S$ and columns indexed in the set $W$. 
For simplicity, if $S=[n]$ then we write $ \boldsymbol{A}_{S,W}$ as $ \boldsymbol{A}_{:,W}$. 
Similarly, if $W=[d]$ then we write $ \boldsymbol{A}_{S,W}$ as $ \boldsymbol{A}_{S,:}$. 
Denote by $ \boldsymbol{A}^{\dagger}\in \mathbb{R}^{d\times n}$  the Moore-Penrose pseudoinverse of $ \boldsymbol{A}\in \mathbb{R}^{n\times d}$, and we denote by $\boldsymbol{A}(i,j)$ the $(i,j)$-th entry of $\boldsymbol{A}$.

The main idea of the classical CUR matrix decomposition is to approximate a target matrix $\boldsymbol{A}\in\mathbb{R}^{n\times d}$ by a low rank matrix 
$\boldsymbol{C} \cdot \boldsymbol{U}\cdot  \boldsymbol{R}$,
where $\boldsymbol{C}:=\boldsymbol{A}_{:,W}$ is a column submatrix, $\boldsymbol{R}:=\boldsymbol{A}_{S,:}$ is a row submatrix, and $\boldsymbol{U}\in\mathbb{R}^{|W|\times |S|}$ is an appropriately chosen matrix. 
This idea can be traced back to Penrose in 1956 \cite{Penrose}, and has since been developed extensively over the last few decades \cite{CXX1,CXX2,Gol65,BG65,GZT95,GTZ97,GTZ972,GT01,DM05,DKM06,DMM08,MD09,GM16}.

Two common choices for the middle matrix $\boldsymbol{U}$ are $\boldsymbol{U}:=(\boldsymbol{A}_{:,W})^{\dagger} \boldsymbol{A} (\boldsymbol{A}_{S,:})^{\dagger}$ \cite{MD09,SE16} and $\boldsymbol{U}:=(\boldsymbol{A}_{S,W})^{\dagger}$ \cite{GTZ97,GTZ972}.
In either case, if $\mathrm{rank}(\boldsymbol{U})=\mathrm{rank}(\boldsymbol{A})$, then the CUR decomposition exactly reconstructs the target matrix, i.e., $\boldsymbol{A}=\boldsymbol{A}_{:,W}\cdot \boldsymbol{U}\cdot  \boldsymbol{A}_{S,:}$ \cite{KH20,CC10,AHKS19,WZZ16}.
In practical applications, the goal is to identify small subsets $S\subset[n]$ and $W\subset[d]$ that minimize the approximation error.
Empirical results indicate that choosing $\boldsymbol{U}=(\boldsymbol{A}_{:,W})^{\dagger} \boldsymbol{A} (\boldsymbol{A}_{S,:})^{\dagger}$ often results in a small approximation error. 
However, computing $(\boldsymbol{A}_{:,W})^{\dagger} \boldsymbol{A} (\boldsymbol{A}_{S,:})^{\dagger}$ requires matrix-vector products, which can be prohibitively expensive.
Therefore, in this paper we primarily focus on the case $\boldsymbol{U}=(\boldsymbol{A}_{S,W})^{\dagger}$. 
The classical CUR problem can be formulated as follows.

\begin{problem}{\rm (Classical CUR problem)}\label{pr1}
Given a matrix $ \boldsymbol{A}\in\mathbb{R}^{n\times d}$ and two positive integers $k,l\leq \mathrm{rank}(\boldsymbol{A})$, find a $k$-subset $S\subset[n]$ and an $l$-subset $W\subset[d]$ such that the following residual
\begin{equation*}
\Vert \boldsymbol{A}- \boldsymbol{A}_{:,W}( \boldsymbol{A}_{S,W})^{\dagger}  \boldsymbol{A}_{S,:}\Vert_{2}
\end{equation*}
is minimized over all possible $S\in \binom{[n]}{k}$ and $W\in \binom{[d]}{l}$.

\end{problem}

In this paper, we also propose the generalized CUR decomposition. 
It allows us to select columns and rows from arbitrary source matrices, offering a more flexible approach to matrix approximation.

\begin{problem}{\rm (Generalized CUR problem)}\label{pr2}
Given a target matrix $ \boldsymbol{A}\in\mathbb{R}^{n\times d}$, three arbitrary source matrices 
\begin{equation*}
\boldsymbol{C}\in\mathbb{R}^{n\times d_{{C}}},
\quad\boldsymbol{U}\in\mathbb{R}^{n_{{R}}\times d_C},
\quad 	\boldsymbol{R}\in\mathbb{R}^{n_{{R}}\times d},
\end{equation*}
and two positive integers $k,l$, find a $k$-subset $S\subset[n_R]$ and an $l$-subset $W\subset[d_C]$ such that the following residual
\begin{equation*}
\Vert \boldsymbol{A}- \boldsymbol{C}_{:,W}( \boldsymbol{U}_{S,W})^{\dagger}  \boldsymbol{R}_{S,:}\Vert_{2}
\end{equation*}
is minimized over all possible $S\in \binom{[n_R]}{k}$ and $W\in \binom{[d_C]}{l}$.

\end{problem}

The generalized CUR problem encompasses a wide range of problems in matrix approximation theory and signal processing.
In the following, we highlight several interesting problems that fall under this framework.
\begin{enumerate}[(i)]
\item
If $\boldsymbol{C}=\boldsymbol{U}=\boldsymbol{R}=\boldsymbol{A}$, then Problem \ref{pr2} simplifies to Problem \ref{pr1}.
If $\boldsymbol{U}=\boldsymbol{C}$, $\boldsymbol{R}=\boldsymbol{A}$ and $S=[n]$, then Problem \ref{pr2} simplifies to the generalized column subset selection (GCSS) problem, which aims to minimize $\Vert \boldsymbol{A}- \boldsymbol{C}_{:,W}( \boldsymbol{C}_{:,W})^{\dagger}  \boldsymbol{A}\Vert_{2}$
over all possible $k$-subsets $W\subset[d_C]$. 
If we further let $\boldsymbol{C}=\boldsymbol{A}$, then GCSS problem becomes the classical column subset selection problem.	
The GCSS problem covers many other problems, such as sparse coding and distributed column subset selection \cite{FEGK13,FGK13,LBRN06,OF97,OMG22,CXX2}.

\item
If $\boldsymbol{A}=\sum_{i=1}^{m}\boldsymbol{a}_i\boldsymbol{a}_i^T\in\mathbb{R}^{n\times n}$ is semidefinite positive, $\boldsymbol{C}=\boldsymbol{R}^T=[\boldsymbol{a}_1,\ldots,\boldsymbol{a}_{m}]\in\mathbb{R}^{n\times m}$, $\boldsymbol{U}=\frac{|S|}{m}\cdot \boldsymbol{I}_{m}$ and $W=S$, then Problem \ref{pr2}  reduces to approximating a semidefinite positive matrix $\boldsymbol{A}$ by minimizing 
$$\bigg\Vert \sum_{i=1}^{m}\boldsymbol{a}_i\boldsymbol{a}_i^T-\frac{m}{|S|} \sum_{i\in S}\boldsymbol{a}_{i}  \boldsymbol{a}_{i}^T\bigg\Vert_{2}$$
over all possible $k$-subsets $S\subset[m]$. 
This problem is closely related to the study of spectral sparsifiers of graphs \cite{Sri13} and the approximation of John's decomposition  \cite{Vershynin,You3}.

\item 
If $\boldsymbol{U}=\boldsymbol{C}\in\mathbb{R}^{n\times k}$, $\boldsymbol{R}=\boldsymbol{A}\in\mathbb{R}^{n\times d}$ and $W=[k]$, then Problem \ref{pr2} simplifies to the following row subset selection problem:
\begin{problem}{\rm (Row subset selection problem)}\label{pr3}
Given a positive integer $k$, a target matrix $ \boldsymbol{A}\in\mathbb{R}^{n\times d}$ and a source matrix $\boldsymbol{C}\in\mathbb{R}^{n\times k}$, find a $k$-subset $S\subset[n]$ such that the following residual
\begin{equation*}
\Vert \boldsymbol{A}- \boldsymbol{C}( \boldsymbol{C}_{S,:})^{\dagger}  \boldsymbol{A}_{S,:}\Vert_{2}
\end{equation*}
is minimized over all possible $S\in \binom{[n]}{k}$.

\end{problem}
This problem has been recently studied in the context of  linear regression \cite{DW17}, where $\boldsymbol{A}=\boldsymbol{a}\in\mathbb{R}^{n}$ is taken as a vector.
The authors in \cite[Theorem 5]{DW17} showed that if $\mathrm{rank}(\boldsymbol{C}_{S,:})=k$ for each $k$-subset $S\subset[n]$, then the expectation of $\Vert \boldsymbol{a}- \boldsymbol{C}( \boldsymbol{C}_{S,:})^{\dagger}  \boldsymbol{a}_{S,:}\Vert_{2}^2$ under volume sampling over all possible $k$ rows of $\boldsymbol{C}$ is equal to $(k+1)\cdot \Vert \boldsymbol{a}- \boldsymbol{C} \boldsymbol{C}^{\dagger}  \boldsymbol{a}_{}\Vert_{2}^2$. 
This finding led to a new volume sampling algorithm for linear regression.
The Frobenius norm version of Problem \ref{pr3} was subsequently investigated in \cite{CK24,Epperly1},
where an efficient algorithm called adaptive randomized pivoting was proposed.
We refer the reader to these works for further details.

\end{enumerate}
\subsection{Our contributions}

In this paper we introduce a new concept called generalized interlacing families (see Definition \ref{Generalized interlacing family}). 
It extends the classical interlacing family introduced by 
Marcus, Spielman and Srivastava in \cite{inter1,inter2,inter3}
to the case where the polynomials in the family can have different degrees.
Analogous to results in the classical setting, we demonstrate the existence of a polynomial within a generalized interlacing family that shares the same degree as the expected polynomial and whose smallest root is bounded below by that of the expected polynomial (see Theorem \ref{generalized interlacing family}).

We apply this generalized framework to the generalized CUR problem (Problem \ref{pr2}). Our primary focus is on a special case where $|S|=|W|=k\leq \mathrm{rank}(\boldsymbol{U})$, which ensures that the middle matrix $\boldsymbol{U}_{S,W}\in\mathbb{R}^{k\times k}$ is always square. We then prove that the characteristic polynomials associated with this problem form a generalized interlacing family. This result enables us to derive an upper bound on the spectral norm of a residual matrix by leveraging the smallest root of the expected polynomial.
This upper bound is, to our knowledge, the first known theoretical bound for a residual matrix in generalized CUR decompositions.

We begin by introducing some notations and definitions.

\begin{definition}\label{def-p-SW}
Let $\boldsymbol{A}\in\mathbb{R}^{n\times d}$, $\boldsymbol{C}\in\mathbb{R}^{n\times d_{{C}}}$,  $\boldsymbol{U}\in\mathbb{R}^{n_{{R}}\times d_C}$ and $\boldsymbol{R}\in\mathbb{R}^{n_{{R}}\times d}$. 
Let $k$ be an integer satisfying $0\leq k\leq \mathrm{rank}(\boldsymbol{U})$. 
For any $k$-subset $S\subset[n_R]$ and for any $k$-subset $W\subset[d_C]$, 
we define
\begin{equation*}
p_{S,W}(x;\boldsymbol{A},\boldsymbol{C},\boldsymbol{U},\boldsymbol{R})
:=(-1)^{d-k}\cdot 
\det\left[\begin{matrix}
\boldsymbol{I}_n  & \boldsymbol{0} & \boldsymbol{C}_{:,W} & \boldsymbol{A}  \\
\boldsymbol{0}& \boldsymbol{0}_{} & \boldsymbol{U}_{S,W}  & \boldsymbol{R}_{S,:}  \\
(\boldsymbol{C}_{:,W})^{T}  &  (\boldsymbol{U}_{S,W})^T &   \boldsymbol{0} &  \boldsymbol{0} \\
\boldsymbol{A}^{T}  & (\boldsymbol{R}_{S,:})^{T} & \boldsymbol{0} & -x\cdot  \boldsymbol{I}_d
\end{matrix}\right].
\end{equation*}
We also define the expected polynomial
\begin{equation*}
\begin{aligned}
P_{k}(x;\boldsymbol{A},\boldsymbol{C},\boldsymbol{U},\boldsymbol{R})
=
\sum_{S\in\binom{[n_R]}{k},W\in\binom{[d_C]}{k}}
p_{S,W}(x;\boldsymbol{A},\boldsymbol{C},\boldsymbol{U},\boldsymbol{R}).
\end{aligned}
\end{equation*}

\end{definition}

\begin{remark}
As will be shown in Proposition \ref{P-exp-base-gcur1}, $p_{S,W}(x;\boldsymbol{A},\boldsymbol{C},\boldsymbol{U},\boldsymbol{R})$ is either identically zero or a polynomial of degree $d-k+\mathrm{rank}(\boldsymbol{U}_{S,W})$ with a positive leading coefficient. 
In particular, if $\boldsymbol{U}_{S,W}\in\mathbb{R}^{k\times k}$ is invertible, 
i.e., $\mathrm{rank}(\boldsymbol{U}_{S,W})=k$, then 
$p_{S,W}(x;\boldsymbol{A},\boldsymbol{C},\boldsymbol{U},\boldsymbol{R})$ can be rewritten as
\begin{equation*}
p_{S,W}(x;\boldsymbol{A},\boldsymbol{C},\boldsymbol{U},\boldsymbol{R})	
=\det[\boldsymbol{U}_{S,W}]^2\cdot 
\det[x\cdot \boldsymbol{I}_d+
(\boldsymbol{A}- \boldsymbol{C}_{:,W}( \boldsymbol{U}_{S,W})^{-1}  \boldsymbol{R}_{S,:})^T
(\boldsymbol{A}- \boldsymbol{C}_{:,W}( \boldsymbol{U}_{S,W})^{-1}  \boldsymbol{R}_{S,:})].
\end{equation*}
Hence, for any pair $(S,W)\in \binom{[n_R]}{k}\times \binom{[d_C]}{k}$ with $\mathrm{rank}(\boldsymbol{U}_{S,W})=k$, we have
\begin{equation*}
\mathrm{minroot}\ p_{S,W}(x;\boldsymbol{A},\boldsymbol{C},\boldsymbol{U},\boldsymbol{R})
=-\|\boldsymbol{A}- \boldsymbol{C}_{:,W}( \boldsymbol{U}_{S,W})^{-1}  \boldsymbol{R}_{S,:}\|_2^2.
\end{equation*}

\end{remark}

\subsubsection{Generalized CUR matrix decompositions}

As will be shown in Lemma   \ref{forms GIF}, the polynomial set
\begin{equation*}
\mathcal{P}_k:=
\Big\{
p_{S,W}(x;\boldsymbol{A},\boldsymbol{C},\boldsymbol{U},\boldsymbol{R})
\Big\}_{S\in\binom{[n_R]}{k}, W\in\binom{[d_C]}{k}}
\end{equation*}
 forms a generalized interlacing family 
 (see Definition \ref{Generalized interlacing family}).
Consequently, 
based on the property of generalized interlacing families (see Theorem \ref{generalized interlacing family}),
we have the following theorem.
Theorem \ref{gcur-th2} states that the smallest root of the expected polynomial for the set $\mathcal{P}_k$, namely $P_k(x;\boldsymbol{A},\boldsymbol{C},\boldsymbol{U},\boldsymbol{R})$, provides a lower bound for the smallest root of one degree-$d$ polynomial belonging to the set $\mathcal{P}_k$.
This implies an upper bound for a residual matrix of generalized CUR matrix decomposition.

\begin{theorem}
\label{gcur-th2}
Let $\boldsymbol{A}\in\mathbb{R}^{n\times d}$, $\boldsymbol{C}\in\mathbb{R}^{n\times d_{{C}}}$,  $\boldsymbol{U}\in\mathbb{R}^{n_{{R}}\times d_C}$ and $\boldsymbol{R}\in\mathbb{R}^{n_{{R}}\times d}$  be given matrices. 
Let $k$ be an integer satisfying $1\leq k\leq \mathrm{rank}(\boldsymbol{U})$.
Then there is a $k$-subset $\widehat{S}\subset[n_R]$ and a $k$-subset $\widehat{W}\subset[d_C]$ such that $\boldsymbol{U}_{\widehat{S},\widehat{W}}\in\mathbb{R}^{k\times k}$ is invertible and
\begin{equation*}
\mathrm{minroot}\ p_{\widehat{S},\widehat{W}}(x;\boldsymbol{A},\boldsymbol{C},\boldsymbol{U},\boldsymbol{R})
\geq \mathrm{minroot}\ P_{k}(x;\boldsymbol{A},\boldsymbol{C},\boldsymbol{U},\boldsymbol{R}).
\end{equation*}
Consequently, we have
\begin{equation}\label{meq71}
\begin{aligned}
&\Vert \boldsymbol{A}- \boldsymbol{C}_{:,\widehat{W}}( \boldsymbol{U}_{\widehat{S},\widehat{W}})^{-1}  \boldsymbol{R}_{\widehat{S},:}\Vert_{2}^2
\leq \mathrm{maxroot}\ P_{k}(-x;\boldsymbol{A},\boldsymbol{C},\boldsymbol{U},\boldsymbol{R}).
\end{aligned}
\end{equation}

\end{theorem}

\begin{remark}
We will show that $P_{k}(-x;\boldsymbol{A},\boldsymbol{C},\boldsymbol{U},\boldsymbol{R})$  has only real roots when $1\leq k\leq \mathrm{rank}(\boldsymbol{U})$.
In particular, this polynomial admits simplifications in the settings of the classical CUR matrix problem and the row subset selection problem (see equation \eqref{eq:2025:xu94} and \eqref{eq:2025:xu94-newcase}, respectively). 
We will use these simplified forms to derive error bounds for these two problems. 
In the general case, the formula \eqref{eq:P-exp-base-qpsw} in Proposition \ref{P-exp-base-gcur2} enables numerical estimation of the largest root of $P_{k}(-x;\boldsymbol{A},\boldsymbol{C},\boldsymbol{U},\boldsymbol{R})$.
Developing a theoretical estimation for this largest root remains an important direction for future research.	
\end{remark}

\begin{remark}
The concept of generalized interlacing families allows the polynomials in the family to have different degrees, which broadens the applicability of the method of interlacing polynomials.
For example, when attempting to apply the method of interlacing polynomials to Problem \ref{pr2}, a natural approach would be to consider the following set of degree-$d$ polynomials:
\begin{equation*}
\mathcal{Q}_k^{}:=\Big\{
\det[x\cdot \boldsymbol{I}_d+
(\boldsymbol{A}- \boldsymbol{C}_{:,W}( \boldsymbol{U}_{S,W})^{\dagger}  \boldsymbol{R}_{S,:})^T
(\boldsymbol{A}- \boldsymbol{C}_{:,W}( \boldsymbol{U}_{S,W})^{\dagger}  \boldsymbol{R}_{S,:})]
\Big\}_{S\in\binom{[n_R]}{k}, W\in\binom{[d_C]}{k}}.
\end{equation*}
However, it is not generally known or  guaranteed that this polynomial set forms a classical interlacing family when $\boldsymbol{U}$ contains singular $k \times k$ submatrices. 
This uncertainty prevents the direct application of classical interlacing methods.
To overcome this issus, we adopt an alternative approach by working with the polynomial set $\mathcal{P}_k^{}$ instead.
Specifically, for each pair $(S,W)$ with $\mathrm{rank}(\boldsymbol{U}_{S,W})<k$, we replace the original polynomial from $\mathcal{Q}_k^{}$
(i.e. $\det[x\cdot \boldsymbol{I}_d+
(\boldsymbol{A}- \boldsymbol{C}_{:,W}( \boldsymbol{U}_{S,W})^{\dagger}  \boldsymbol{R}_{S,:})^T
(\boldsymbol{A}- \boldsymbol{C}_{:,W}( \boldsymbol{U}_{S,W})^{\dagger}  \boldsymbol{R}_{S,:})]$) with $p_{S,W}(x;\boldsymbol{A},\boldsymbol{C},\boldsymbol{U},\boldsymbol{R})$.
Although the polynomials in $\mathcal{P}_k^{}$ have different degrees,
 their roots exhibit a desired interlacing property. Motivated by this observation, we introduce the definition of {\em generalized interlacing families},
which allows us to use the expected polynomial $P_{k}(x;\boldsymbol{A},\boldsymbol{C},\boldsymbol{U},\boldsymbol{R})$ to derive an error bound for Problem \ref{pr2}.
\end{remark}

\subsubsection{Classical CUR matrix decompositions}

We next focus on the classical CUR matrix decompositions (Problem \ref{pr1}), i.e., $\boldsymbol{C}=\boldsymbol{U}=\boldsymbol{R}=\boldsymbol{A}$.
In this case, we will demonstrate in Proposition \ref{P-exp-base-gcur-x} that 
 $P_{k}(-x;\boldsymbol{A},\boldsymbol{A},\boldsymbol{A},\boldsymbol{A})$ can be expressed in terms of the flip operator $\mathcal{R}_{x,d}^+$ and the Laguerre derivative operator $\partial_x\cdot x\cdot \partial_x$. Specifically, we have 
\begin{equation}\label{eq:2025:xu94}
P_{k}(-x;\boldsymbol{A},\boldsymbol{A},\boldsymbol{A},\boldsymbol{A})
=\frac{(-1)^{d-k}}{(k!)^2}\cdot  \mathcal{R}_{x,d}^+\cdot (\partial_x\cdot x\cdot \partial_x)^k \cdot \mathcal{R}_{x,d}^+\   \det[x\cdot \boldsymbol{I}_d-\boldsymbol{A}^T\boldsymbol{A}].
\end{equation}
Here, $\partial_x$ denotes the derivative operator $\frac{\partial}{\partial_x}$, and $\mathcal{R}_{x,d}^+$ is the flip operator defined as $\mathcal{R}_{x,d}^+\ p(x) :=x^d\cdot p(1/x)$ for any polynomial $p(x)$ of degree at most $d$. 
This simplified formula \eqref{eq:2025:xu94} facilitates the estimation of the largest root of $P_{k}(-x;\boldsymbol{A},\boldsymbol{A},\boldsymbol{A},\boldsymbol{A})$ by employing the polynomial multiplicative convolution and barrier function method introduced in \cite{inter3,inter5}.

Building upon the properties of generalized interlacing families and expression \eqref{eq:2025:xu94}, we establish in the following theorem an upper bound on the approximation error
 $\Vert \boldsymbol{A}- \boldsymbol{A}_{:,\widehat{W}}( \boldsymbol{A}_{\widehat{S},\widehat{W}})^{-1}  \boldsymbol{A}_{\widehat{S},:}\Vert_{2}^2$,  where the index sets $\widehat{S}$ and $\widehat{W}$ are selected via the method of interlacing polynomials. A detailed comparison between our bound and existing results can be found in Section~\ref{Comparison with previous work}.

\begin{theorem}
\label{th1}
Assume that $\boldsymbol{A}\in\mathbb{R}^{n\times d}$ is a rank-$t$ matrix with $\|\boldsymbol{A}\|_2\leq 1$, and let $k$ be an integer satisfying $2\leq k\leq  t-1$.
Let $\lambda_i$ be the $i$-th largest eigenvalue of $\boldsymbol{A}^T\boldsymbol{A}$.
Assume that $\lambda_1>\lambda_k$. 
Then 
there is a $k$-subset $\widehat{S}\subset[n]$ and a $k$-subset $\widehat{W}\subset[d]$  such that $\boldsymbol{A}_{\widehat{S},\widehat{W}}$ is invertible and
\begin{equation}
\label{xueq11}
\begin{aligned}
\Vert \boldsymbol{A}- \boldsymbol{A}_{:,\widehat{W}}( \boldsymbol{A}_{\widehat{S},\widehat{W}})^{-1}  \boldsymbol{A}_{\widehat{S},:}\Vert_{2}^2
\leq 	
\frac{(t-k+1)t}{(1-\sqrt{\alpha_k}	)^2} \cdot  \|\boldsymbol{A}-\boldsymbol{A}_k\|_2^2.
\end{aligned}
\end{equation}
Here, $\boldsymbol{A}_k\in\mathbb{R}^{n\times d}$ is the best rank-$k$ approximation to $\boldsymbol{A}$, and $\alpha_{k}\in(0,1)$ is defined as
\begin{equation*}
\alpha_{k}=\frac{\frac{1}{\lambda_{k+1}}-\frac{1}{k}\sum_{i=1}^{k}\frac{1}{\lambda_i}}{\frac{1}{\lambda_{k+1}}-1}.
\end{equation*}

\end{theorem}

By exploiting the rank-one update formula for the residual matrix,
we next develop a deterministic polynomial-time algorithm for the classical CUR matrix problem, under the assumption that the matrix $\boldsymbol{A}$ has no singular square submatrix of size at most $k$.
The following theorem states that the output of Algorithm \ref{alg1} can attain the error bound \eqref{meq71} in Theorem \ref{gcur-th2} up to a  computational error.

\begin{theorem}
\label{th3}
Assume that $\boldsymbol{A}\in\mathbb{R}^{n\times d}$.
Let $k$ be an integer satisfying $1\leq k\leq \mathrm{rank}(\boldsymbol{A})$.
Assume that $\boldsymbol{A}_{S,W}$ is invertible for any $S\subset[n]$ and $W\subset[d]$ with $|S|=|W|\leq k$.
{For any $\epsilon>0$, }
Algorithm \ref{alg1} can output a $k$-subset $\widehat{S}=\{i_1,\ldots,i_k\}\subset[n]$ and a $k$-subset $\widehat{W}=\{j_1,\ldots,j_k\}\subset[d]$ such that 
\begin{equation}\label{eq:2025:xu32}
\Vert \boldsymbol{A}- \boldsymbol{A}_{:,\widehat{W}}( \boldsymbol{A}_{\widehat{S},\widehat{W}})^{-1}  \boldsymbol{A}_{\widehat{S},:}\Vert_{2}^2
\leq 2k\varepsilon+ \mathrm{maxroot}\     P_{k}(-x;\boldsymbol{A},\boldsymbol{A},\boldsymbol{A},\boldsymbol{A}).
\end{equation}
The time complexity of Algorithm~\ref{alg1} is $O(kn^2d^2 + knd^{w+1} \log(d \vee \tfrac{1}{\varepsilon}))$, where $a \vee b := \max\{a, b\}$ for real numbers $a$ and $b$, and $w \in (2, 2.373)$ denotes the matrix multiplication exponent.

\end{theorem}

\begin{remark}
For any matrix $\boldsymbol{A} \in \mathbb{R}^{n \times d}$, we can construct a perturbed matrix 
${\boldsymbol{B}}=\boldsymbol{A}+t\cdot \boldsymbol{E}$, where $\boldsymbol{E}$ is a totally positive matrix, i.e., the determinant of each square submatrix is positive.
When $t>0$ is sufficiently small,
the perturbed matrix ${\boldsymbol{B}}$ contains no singular square submatrix of size at most $k$.
We then apply Theorem \ref{th3} to ${\boldsymbol{B}}$ to obtain two $k$-subsets $\widehat{S}\subset[n]$ and $\widehat{W}\subset[d]$.
Numerical experiments suggest that if $\boldsymbol{A}_{\widehat{S}, \widehat{W}}$ is invertible, 
then the approximation error $\Vert \boldsymbol{A}- \boldsymbol{A}_{:,\widehat{W}}( \boldsymbol{A}_{\widehat{S},\widehat{W}})^{-1}  \boldsymbol{A}_{\widehat{S},:}\Vert_{2}^2$ is comparable to $\Vert {\boldsymbol{B}}- {\boldsymbol{B}}_{:,\widehat{W}}( {\boldsymbol{B}}_{\widehat{S},\widehat{W}})^{-1}  {\boldsymbol{B}}_{\widehat{S},:}\Vert_{2}^2$, and can therefore also be effectively bounded.
 A rigorous theoretical understanding of this observation is left for future work.
\end{remark}

\subsubsection{Row subset selection problem}

We next focus on the row subset selection problem (Problem \ref{pr3}), i.e., $\boldsymbol{U}=\boldsymbol{C}\in\mathbb{R}^{n\times k}$, $\boldsymbol{R}=\boldsymbol{A}\in\mathbb{R}^{n\times d}$, where $k$ is the sampling size.
We consider a special case when $\mathrm{rank}(\boldsymbol{C})=k$.
In this case, we will demonstrate in Proposition \ref{P-exp-base-gcur-x-newcase} that 
 $P_{k}(-x;\boldsymbol{A},\boldsymbol{C},\boldsymbol{C},\boldsymbol{A})$ can be expressed as 
\begin{equation}\label{eq:2025:xu94-newcase}
P_{k}(-x;\boldsymbol{A},\boldsymbol{C},\boldsymbol{C},\boldsymbol{A})
=\frac{(-1)^{d}\det\boldsymbol{C}^T\boldsymbol{C}}{k!}\cdot  \mathcal{R}_{x,d}^+\cdot  \partial_x^k\cdot x^k \cdot \mathcal{R}_{x,d}^+\   \det[x\cdot \boldsymbol{I}_d-\boldsymbol{A}^T(\boldsymbol{I}_n-\boldsymbol{C}\boldsymbol{C}^{\dagger})\boldsymbol{A}].
\end{equation}
Based on this simplified formula,
we establish the approximation error bound for the row subset selection problem (Problem \ref{pr3}).

\begin{theorem}
\label{th1-newcase}
Assume that $\boldsymbol{A}\in\mathbb{R}^{n\times d}$ and $\boldsymbol{C}\in\mathbb{R}^{n\times k}$ satisfying $\mathrm{rank}(\boldsymbol{C})=k$.
Let $r=\mathrm{rank}(\boldsymbol{A}-\boldsymbol{C}\boldsymbol{C}^{\dagger}\boldsymbol{A})$.
Then there is a $k$-subset $\widehat{S}\subset[n]$ such that $\boldsymbol{C}_{\widehat{S},:}\in\mathbb{R}^{k\times k}$ is invertible and
\begin{equation}
\label{xueq11-newcase}
\begin{aligned}
\Vert \boldsymbol{A}- \boldsymbol{C}( \boldsymbol{C}_{\widehat{S},:})^{-1}  \boldsymbol{A}_{\widehat{S},:}\Vert_{2}^2
\leq 	
(1+ kr) \cdot  \|\boldsymbol{A}-\boldsymbol{C}\boldsymbol{C}^{\dagger}\boldsymbol{A}\|_2^2.
\end{aligned}
\end{equation}

\end{theorem}

\begin{remark}
When $\boldsymbol{A}=\boldsymbol{a}\in\mathbb{R}^{n}$ is taken as a vector, we have $d=r=1$, so Theorem \ref{th1-newcase} recovers the result in \cite[Theorem 5]{DW17}. When $\boldsymbol{A}$ is a general matrix, Theorem \ref{th1-newcase} provides the first known spectral norm error bound for Problem  \ref{pr3}.
\end{remark}

We now summarize the main contributions of this paper. 
\begin{enumerate}[\rm (i)]
	\item First, we introduce the concept of generalized interlacing families, which extends the interlacing family framework developed by Marcus, Spielman, and Srivastava \cite{inter1,inter2,inter3}.  This generalization broadens the scope of applications for interlacing families.
	\item Second, we employ the framework of generalized interlacing families to derive an upper bound for a residual matrix in generalized CUR matrix decompositions. While this upper bound is expressed in terms of the largest root of the  polynomial $P_{k}(-x;\boldsymbol{A},\boldsymbol{C},\boldsymbol{U},\boldsymbol{R})$, it represents, to our knowledge, the first known theoretical upper bound for a residual matrix in generalized CUR decompositions.
	\item Finally, we focus on two special cases: the classical CUR matrix problem ($\boldsymbol{C}=\boldsymbol{U}=\boldsymbol{R}=\boldsymbol{A}\in\mathbb{R}^{n\times d}$) and the row subset selection problem ($\boldsymbol{U}=\boldsymbol{C}\in\mathbb{R}^{n\times k}$, $\boldsymbol{R}=\boldsymbol{A}\in\mathbb{R}^{n\times d}$). 
For the classical CUR matrix problem, we derive an explicit upper bound for the largest root of the polynomial $P_{k}(-x;\boldsymbol{A},\boldsymbol{A},\boldsymbol{A},\boldsymbol{A})$. This result leads directly to an explicit upper bound for a residual matrix in classical CUR decompositions, which we demonstrate to be tighter than existing bounds in numerous instances.
For the row subset selection problem, we similarly bound the largest root of $P_{k}(-x;\boldsymbol{A},\boldsymbol{C},\boldsymbol{C},\boldsymbol{A})$, yielding the first known spectral norm error bound for this problem when $\boldsymbol{A}$ is a general matrix.

\end{enumerate}

\subsection{Comparison with previous work}
\subsubsection{Classical CUR matrix problem}
\label{Comparison with previous work}

The CUR matrix decomposition has gained much attention in both numerical linear algebra \cite{Gol65,BG65,GE96,GTZ97,GTZ972,ADMMWG15} and and theoretical computer science  \cite{DM05,DKM06,DMM08,MD09,GM16,BW14,WZ13,WZZ16}.
The works of Goreinov, Zamarashkin, and Tyrtyshnikov \cite{GZT95,GTZ97,GTZ972,GT01} initiated much of the modern development for the classical CUR decomposition problem (Problem \ref{pr1}). 
In some formulations, the standard pseudoinverse $( \boldsymbol{A}_{S,W})^{\dagger}$ of the middle matrix is substituted with its $\tau$-pseudoinverse \cite{GTZ97,GTZ972}. 
Specifically, for a real number $\tau>0$ and a matrix $\boldsymbol{M}\in\mathbb{R}^{n\times n}$ with the singular value decomposition $\boldsymbol{M}=\boldsymbol{U}\mathrm{diag}(\sigma_1,\ldots,\sigma_n)\boldsymbol{V}^T$, the $\tau$-pseudoinverse $\boldsymbol{M}_{\tau}^{\dagger}$ of $\boldsymbol{M}$ is defined by $\boldsymbol{V}\mathrm{diag}(\widehat{\sigma}_1,\ldots,\widehat{\sigma}_n)\boldsymbol{U}^T$, where 
\begin{equation*}
\widehat{\sigma}_i=
\begin{cases}
\frac{1}{{\sigma}_i}&\text{if ${\sigma}_i\geq \tau$,}\\
0&\text{if ${\sigma}_i< \tau$.}
\end{cases}
\end{equation*}
Assume that $\boldsymbol{A}\in\mathbb{R}^{n\times d}$ and $\boldsymbol{F}\in\mathbb{R}^{n\times d}$ satisfy $\mathrm{rank}(\boldsymbol{A}-\boldsymbol{F})\leq k$ and $\|\boldsymbol{F}\|_2\leq \epsilon$. 
Then the author in \cite{GZT95} showed that 
there is a $k$-subset $\widehat{S}\subset[n]$ and a $k$-subset $\widehat{W}\subset[d]$ such that
\begin{equation}\label{eqxu7}
\Vert \boldsymbol{A}- \boldsymbol{A}_{:,\widehat{W}}  ( \boldsymbol{A}_{\widehat{S},\widehat{W}})^{\dagger}_{\epsilon} \boldsymbol{A}_{\widehat{S},:}\Vert_{2}^2
\leq 	
\|\boldsymbol{F}\|_2^2 \cdot \big(1+T^2_{k,n} \big) \cdot  \big(1+ (\sqrt{T_{k,n}}+\sqrt{T_{k,d}} )^2\big)^2,
\end{equation}
where $T_{k,n}$ is defined as
\begin{equation*}
T_{k,n}:=\frac{1}{\min\limits_{\boldsymbol{U}\in\mathbb{R}^{n\times k}, \boldsymbol{U}^T\boldsymbol{U}=\boldsymbol{I}_k} \max\limits_{S\subset[n],|S|=k} \sigma_{\min}(\boldsymbol{U}_{S,:})}
\end{equation*}
for any positive integers $k,n$ with $k\leq n$.
Later, the authors in \cite{Mik14,OZ18} showed that there is a $k$-subset $\widehat{S}\subset[n]$ and a $k$-subset $\widehat{W}\subset[d]$ such that
\begin{equation}\label{eqxu1}
\Vert \boldsymbol{A}- \boldsymbol{A}_{:,\widehat{W}}( \boldsymbol{A}_{\widehat{S},\widehat{W}})^{\dagger}  \boldsymbol{A}_{\widehat{S},:}\Vert_{2}^2
\leq 	
\|\boldsymbol{F}\|_2^2  \cdot  \big(1+T^2_{k,n} \big)\cdot \big(1+T_{k,d} \big)^2.	
\end{equation}
It has been shown in \cite{GTZ972, HongPan} that for any $1\leq k\leq n$, we have
\begin{equation}\label{eq:Tkn}
T_{k,n}\leq \sqrt{1+k(n-k)}.
\end{equation} 
In \cite{Xu25}, the author employs the method of interlacing polynomials to present a simple proof of (\ref{eq:Tkn}) with a slight improvement.
By setting $\boldsymbol{F}=\boldsymbol{A}-\boldsymbol{A}_k$, both error bounds \eqref{eqxu7} and \eqref{eqxu1} are
\begin{equation}\label{xueq11-1}
O(k^2(n-k)(d-k))	\cdot \|\boldsymbol{A}-\boldsymbol{A}_k\|_2^2.	
\end{equation}
Here, $\boldsymbol{A}_k\in {\mathbb R}^{n\times d}$ is the best rank-$k$ approximation to $\boldsymbol{A}$.
Recently, the authors in \cite{OZ18} showed that there exists a $k$-subset $\widehat{S}\subset[n]$ and a $k$-subset $\widehat{W}\subset[d]$ such that 
\begin{equation}\label{eqxu8}
\Vert \boldsymbol{A}- \boldsymbol{A}_{:,\widehat{W}}( \boldsymbol{A}_{\widehat{S},\widehat{W}})^{-1}  \boldsymbol{A}_{\widehat{S},:}\Vert_{F}^2
\leq 	
(k+1)^2\cdot  \|\boldsymbol{A}-\boldsymbol{A}_k\|_F^2.	
\end{equation}
Later, Cortinovis and Kressner \cite{CK20} proposed a deterministic polynomial-time algorithm to output $\widehat{S}$ and $\widehat{W}$ such that \eqref{eqxu8} is satisfied.
Note that 
\begin{equation*}
\|\boldsymbol{A}-\boldsymbol{A}_k\|_F^2=\sum_{j=k+1}^{t}\sigma^2_{j}\leq (t-k) \cdot 	\sigma^2_{k+1}=(t-k) \cdot\|\boldsymbol{A}-\boldsymbol{A}_k\|_2^2,
\end{equation*}
where $t=\mathrm{rank}(\boldsymbol{A})$ and $\sigma_{j} $ denotes the $j$-th largest singular value of $\boldsymbol{A}$.
Therefore, the error bound in \eqref{eqxu8} implies the following spectral norm error bound:
\begin{equation}\label{eqxu9}
\Vert \boldsymbol{A}- \boldsymbol{A}_{:,\widehat{W}}( \boldsymbol{A}_{\widehat{S},\widehat{W}})^{-1}  \boldsymbol{A}_{\widehat{S},:}\Vert_{2}^2
\leq 	
(k+1)^2\cdot (t-k)\cdot  \|\boldsymbol{A}-\boldsymbol{A}_k\|_2^2.	
\end{equation}
This bound improves upon the earlier results in \eqref{xueq11-1}. 
Comparing \eqref{eqxu9} with our error bound \eqref{xueq11}, we see that our method yields a tighter spectral norm error when the condition $t\leq (k+1)^2( 1-\sqrt{\alpha_k}	)^2 $ holds.

\subsubsection{Matrix approximation using the method of interlacing families}
\label{related work-CXX}

 Recently, we applied the method of interlacing polynomials to the classical column selection problem \cite{CXX1}.  
Recall that the aim of the classical column selection problem is to approximate a given matrix $\boldsymbol{A}\in\mathbb{R}^{n\times d}$ by minimizing $\Vert \boldsymbol{A}- \boldsymbol{A}_{:,W}( \boldsymbol{A}_{:,W})^{\dagger}  \boldsymbol{A}\Vert_{2}$
over all possible $k$-subsets $W\subset[d]$.
We showed that there is a $k$-subset $\widehat{W}\subset[d]$ such that
\begin{equation*}
\begin{aligned}
\|\boldsymbol{A}- \boldsymbol{A}_{:,\widehat{W}}( \boldsymbol{A}_{:,\widehat{W}})^{\dagger}  \boldsymbol{A}\|_2^2
&=\mathrm{maxroot}\ \det\big[x\cdot \boldsymbol{I}_d-
\big(\boldsymbol{A}- \boldsymbol{A}_{:,\widehat{W}}( \boldsymbol{A}_{:,\widehat{W}})^{\dagger}  \boldsymbol{A}\big)^T
\big(\boldsymbol{A}- \boldsymbol{A}_{:,\widehat{W}}( \boldsymbol{A}_{:,\widehat{W}})^{\dagger}  \boldsymbol{A}\big)\big]
\\&
\leq \mathrm{maxroot}\ P_k(x;\boldsymbol{A}),
\end{aligned}
\end{equation*}
where the expected polynomial $P_k(x;\boldsymbol{A})$ is defined as
\begin{equation*}
P_k(x;\boldsymbol{A}):=\sum_{W\in\binom{[d]}{k}}\det\big[(\boldsymbol{A}_{:,W})^T\boldsymbol{A}_{:,W}\big]\cdot
\det\big[x\cdot \boldsymbol{I}_d-
\big(\boldsymbol{A}- \boldsymbol{A}_{:,W}( \boldsymbol{A}_{:,W})^{\dagger}  \boldsymbol{A}\big)^T
\big(\boldsymbol{A}- \boldsymbol{A}_{:,W}( \boldsymbol{A}_{:,W})^{\dagger}  \boldsymbol{A}\big)\big].	
\end{equation*}
Furthermore, we showed in \cite[Proposition 3.6]{CXX1} that $P_k(x;\boldsymbol{A})$ can be expressed as
\begin{equation}\label{meq33}
P_k(x;\boldsymbol{A})=\frac{(-1)^{k}}{k!}\cdot  \mathcal{R}_{x,d}^+\cdot \partial_x^k\cdot  \mathcal{R}_{x,d}^+\   \det[x\cdot \boldsymbol{I}_d-\boldsymbol{A}^T\boldsymbol{A}].	
\end{equation}
Comparing \eqref{meq33} with \eqref{eq:2025:xu94}, we observe that the polynomial $P_{k}(-x;\boldsymbol{A},\boldsymbol{A},\boldsymbol{A},\boldsymbol{A})$ defined in \eqref{eq:2025:xu94} can be viewed as a variant of $P_k(x;\boldsymbol{A})$, where the $k$-th order differential operator $\partial_x^k$ (appearing between the two flip operators $\mathcal{R}_{x,d}^+$ in \eqref{meq33}) is replaced by the $k$-th order Laguerre derivative operator  $(\partial_x\cdot x\cdot \partial_x)^k$.

\section{Preliminaries}

\subsection{Notations}

For a complex number $a$, we denote its imaginary part by $\mathbf{Im}(a)$.
We denote by $\boldsymbol{I}_n$ the identity matrix of size $n$, and by $\boldsymbol{0}_{n\times d}$ the zero matrix of size $n\times d$.
When the dimensions are clear from context, we may omit the subscript and simply write $\boldsymbol{0}$ to denote a zero matrix of appropriate size.
For a collection of real numbers ${a}_1,\ldots,{a}_d$, we use $\mathrm{diag}({a}_1,\ldots,{a}_d)$ to denote the $d\times d$ diagonal matrix whose $i$-th diagonal entry is $a_i$. 

Let $\boldsymbol{A} \in \mathbb{R}^{n \times d}$ be a matrix. 
We denote by $\boldsymbol{A}_{S,W}$ the submatrix of $\boldsymbol{A}$ formed by selecting the rows indexed by the set $S$ and the columns indexed by the set $W$. 
For convenience, when $S = [n]$, we write $\boldsymbol{A}_{S,W}$ as $\boldsymbol{A}_{:,W}$.
Similarly, when $W = [d]$, we write $\boldsymbol{A}_{S,W}$ as $\boldsymbol{A}_{S,:}$. 
We use $\boldsymbol{A}^\dagger \in \mathbb{R}^{d \times n}$ to denote the Moore-Penrose pseudoinverse of $\boldsymbol{A}$, and let $\boldsymbol{A}(i,j)$ denote the $(i,j)$-th entry of $\boldsymbol{A}$.
For convenience, we set $\det[\boldsymbol{A}_{\emptyset,\emptyset}]=1$ for any matrix $\boldsymbol{A}$.

We denote by $\mathbb{R}[z_1,\ldots,z_n]$ the set of multivariate polynomials in the variables $z_1,\ldots,z_n$ with real coefficients. 
The partial derivative with respect to $z_i$ is denoted by $\partial_{z_i}:=\partial\slash\partial_{z_i}$. For each subset $S\subset[n]$, we define $\partial_{\boldsymbol{Z}^S}:=\prod_{i\in S}\partial_{z_i}$. 

For a univariate polynomial $p(x)\in\mathbb{R}[x]$, we use $\mathrm{deg}(p)$ to denote the degree of $p(x)$. If all the roots of $p(x)$ are real, we use $\lambda_i(p)$	to denote the $i$-th largest root of $p(x)$.
Moreover, we use $\lambda_{\max}^{\varepsilon}(p)\in \mathbb{R}$ to denote an $\varepsilon$-approximation to the largest root of  $p(x)$, i.e.,
\begin{equation*}
| \lambda_{\max}^{\varepsilon}(p)-\mathrm{maxroot}\ p|\leq \varepsilon.
\end{equation*}

Let $p(x)=\sum_{i=0}^{d}a_i\cdot x^i\in\mathbb{R}[x]$ be a univariate polynomial with degree at most $d$. We define the flip operator $\mathcal{R}_{x,d}^+:\mathbb{R}[x]\to \mathbb{R}[x]$ as
\begin{equation*}
\mathcal{R}_{x,d}^+\ p(x) :=x^d\cdot p(1/x)=\sum_{i=0}^{d}a_i\cdot x^{d-i}.
\end{equation*}
Now, let $f(z_1,\ldots,z_d)\in\mathbb{R}[z_1,\ldots,z_d]$ be a multivariate polynomial such that its partial degree with respect to each variable $z_i$ is at most $r_i$ for $i=1,\dots,d$. For each $i\in[d]$, we define the operator $\mathcal{R}_{z_i,r_i}^-:\mathbb{R}[z_1,\ldots,z_d]\to \mathbb{R}[z_1,\ldots,z_d]$ as
\begin{equation*}
\mathcal{R}_{z_i,r_i}^-\ f(z_1,\ldots,z_d) :=z_i^{r_i}\cdot f(z_1,\ldots,-z_i^{-1},\ldots,z_d).
\end{equation*}

\subsection{Linear algebra}

Let $\boldsymbol{A}\in\mathbb{R}^{d\times d}$ and
let $\boldsymbol{Z}=\mathrm{diag}(z_1,\ldots,z_d)$, where $z_1,\ldots,z_d$ are variables.
Then the determinant of $\boldsymbol{Z}+\boldsymbol{A}$ can be expanded as  
\begin{equation}\label{meq11}
\det[\boldsymbol{Z}+\boldsymbol{A}]=\sum_{S\subset[d]}  \det[\boldsymbol{A}_{S,S}]\cdot \prod_{i\notin S}z_i, 	
\end{equation}
(see \cite[Page 524]{ravi2}).
For any subset $S\subset[d]$, by differentiating both sides of \eqref{meq11} with respect to $z_j$ for each $j\notin S$, and then setting $z_1=\cdots=z_d=0$, we obtain the following identity 
\begin{equation}\label{meq18}
\det[\boldsymbol{A}_{S,S}]
=\Big(\prod_{i\notin S}\partial_{z_i}\Big) \det[\boldsymbol{Z}+\boldsymbol{A}]\ \bigg|_{z_i=0,\forall i\in[d]}.
\end{equation}

The following is the Schur determinantal formula.

\begin{lemma}\label{Linear algebra-lemma-schur}{\rm\cite[Page 24]{HJ12}}
Let $\boldsymbol{A} \in \mathbb{R}^{d \times d}$, $\boldsymbol{B} \in \mathbb{R}^{d \times n}$, $\boldsymbol{C} \in \mathbb{R}^{n \times d}$, and $\boldsymbol{D} \in \mathbb{R}^{n \times n}$ be real matrices. If $\boldsymbol{D}$ is invertible, then 
\begin{equation*}
\det\left[\begin{matrix}
\boldsymbol{A} & \boldsymbol{B}  \\
\boldsymbol{C}   & \boldsymbol{D} \\
\end{matrix}\right]
= \det(\boldsymbol{D}) \cdot \det(\boldsymbol{A}-\boldsymbol{B}\boldsymbol{D}^{-1}\boldsymbol{C}).
\end{equation*}
\end{lemma}

The following lemma presents the rank-one update formula for the residual matrix in the classical CUR problem.
{\color{black}
\begin{lemma}\label{Linear algebra-lemma-2}

Let $\boldsymbol{A}\in\mathbb{R}^{n\times d}$. Assume that $\boldsymbol{A}_{S,W}$ is invertible, where $S\subset[n]$ and $W\subset[d]$ satisfy $|S|=|W|$.
Let $\boldsymbol{B}:=\boldsymbol{A} -\boldsymbol{A}_{:,W}(\boldsymbol{A}_{S,W})^{-1}\boldsymbol{A}_{S,:}\in\mathbb{R}^{n\times d}$. 

\begin{itemize}
\item[\rm{(i)}] If $i\in S$ or $j\in  W$, then we have $\boldsymbol{B}(i,j)=0$. If  $i\in [n]\backslash S$ and $j\in [d]\backslash W$, then we have $\det[\boldsymbol{A}_{S\cup\{i\},W\cup\{j\}}]
= \det[\boldsymbol{A}_{S,W}]\cdot \boldsymbol{B}(i,j)$.
\item[\rm{(ii)}] Let $i\in [n]\backslash S$ and $j\in [d]\backslash W$.  Assume that $\boldsymbol{B}(i,j)\neq 0$. Then  $\boldsymbol{A}_{S\cup \{i\},W\cup \{j\}}$ is invertible and
\begin{equation}\label{meq61}
\begin{aligned}
\boldsymbol{A}-\boldsymbol{A}_{:,W\cup\{j\}}(\boldsymbol{A}_{S\cup \{i\},W\cup \{j\}})^{-1}\boldsymbol{A}_{S\cup \{i\},:}
=\boldsymbol{B} -\frac{1}{\boldsymbol{B}(i,j)}\cdot \boldsymbol{B}_{:,\{j\}}\cdot \boldsymbol{B}_{\{i\},:}.
\end{aligned}
\end{equation}	
\end{itemize}

\end{lemma}

\begin{proof}
See Appendix \ref{proof-Linear algebra-lemma-2}.	
\end{proof}

\subsection{The differential operator $\mathcal{R}_{x,d}^+\cdot (\partial_x\cdot x\cdot \partial_x )^k\cdot \mathcal{R}_{x,d}^+$}

\begin{lemma}\label{Linear algebra-lemma-6}
Assume that $p(x)=\sum_{i=0}^dc_i\cdot x^i$ is a polynomial of degree at most $d$, where each $c_i\in\mathbb{R}$.
Then for any integer $0\leq k\leq d$, we have
\begin{equation*}
\begin{aligned}
\mathcal{R}_{x,d}^+\cdot (\partial_x\cdot x\cdot \partial_x)^k \cdot \mathcal{R}_{x,d}^+ \ p(x)
&= x^k\cdot   \sum_{i=0}^{d-k} \Big(\frac{(d-i)!}{(d-i-k)!}\Big )^2\cdot c_{i}\cdot x^{i}.\\
\end{aligned}
\end{equation*}
In particular, 
if $p(x)=x^d$, then $\mathcal{R}_{x,d}^+\cdot \partial_x^k\cdot x^k\cdot \partial_x^k \cdot \mathcal{R}_{x,d}^+ \ p(x)	\equiv 0$ for each $1\leq k\leq d$.
\end{lemma}

\begin{proof}
Note that  the Laguerre derivative operator $\partial_x\cdot x\cdot \partial_x$ satisfies 
\begin{equation}\label{meq72}
(\partial_x\cdot x\cdot \partial_x)^l f(x)=\partial_x^l\cdot x^l\cdot \partial_x^l\ f(x)	
\end{equation}
for any positive integer $l$ and for any polynomial $f(x)$ \cite{Vis94}.
Hence, we have
\begin{equation}\label{meq70}
\mathcal{R}_{x,d}^+\cdot (\partial_x\cdot x\cdot \partial_x)^k \cdot \mathcal{R}_{x,d}^+ \ p(x)=
\mathcal{R}_{x,d}^+\cdot \partial_x^k\cdot x^k\cdot \partial_x^k \cdot \mathcal{R}_{x,d}^+ \ p(x).	
\end{equation}
Then a direct calculation gives
\begin{equation*}
\begin{aligned}
\mathcal{R}_{x,d}^+\cdot (\partial_x\cdot x\cdot \partial_x)^k \cdot \mathcal{R}_{x,d}^+ \ p(x) 
&=  \mathcal{R}_{x,d}^+\cdot \partial_x^k\cdot x^k\cdot \partial_x^k \cdot  \sum_{i=0}^dc_{d-i}\cdot x^i
= \mathcal{R}_{x,d}^+\cdot \partial_x^k \cdot  \sum_{i=k}^d \frac{i!}{(i-k)!}\cdot c_{d-i}\cdot x^i\\
&=  \mathcal{R}_{x,d}^+    \sum_{i=k}^d \Big(\frac{i!}{(i-k)!}\Big)^2\cdot c_{d-i}\cdot x^{i-k}
=  x^k\cdot   \sum_{i=0}^{d-k} \Big(\frac{(d-i)!}{(d-i-k)!}\Big )^2\cdot c_{i}\cdot x^{i}.\\
\end{aligned}	
\end{equation*}	
In particular, if $p(x)=x^d$, then $\mathcal{R}_{x,d}^+\ p(x)=1$ and $\partial_x \cdot\mathcal{R}_{x,d}^+\ p(x)\equiv 0$, so we have $\mathcal{R}_{x,d}^+\cdot  (\partial_x\cdot x\cdot \partial_x)^k \cdot \mathcal{R}_{x,d}^+ \ p(x)	\equiv 0$ for each $1\leq k\leq d$.

\end{proof}

\subsection{Real stability}

Recall that a univariate real polynomial  is said to be {\em real-rooted} if all of its roots are real. 
Here we introduce the concept of real stable polynomials, which extends the notion of real-rootedness from univariate to multivariate polynomials.

\begin{definition}\label{defstable}
A multivariate polynomial $p\in\mathbb{R}[z_1,\ldots, z_n]$ is called \emph{real stable} if it does not vanish for all $(z_1,\ldots,z_n)\in\mathbb{C}^n$ with $\mathbf{Im}(z_i)>0, i=1,\ldots,n$.
\end{definition}

A univariate polynomial is real stable if and only if it has only real roots. 
The following two lemmas will be instrumental in establishing the real stability of the polynomials considered in this paper.

\begin{lemma}\label{real stable}
{\rm {\cite[Proposition 2.4]{Branden}}} 
Let $\boldsymbol{B}\in\mathbb{C}^{d\times d}$ be a Hermitian matrix, and let $\boldsymbol{A}_1,\ldots,\boldsymbol{A}_n\in\mathbb{C}^{d\times d}$ be positive semidefinite Hermitian matrices. 
Then the polynomial
\begin{equation*}
\det[\boldsymbol{A}_1 z_1+\cdots+\boldsymbol{A}_n z_n+\boldsymbol{B}]\in\mathbb{R}[z_1,\ldots,z_n]
\end{equation*}
is either identically zero or real stable in $z_1,\ldots,z_n$.
\end{lemma}

\begin{lemma}\label{real stable2}
Let $r_1,\ldots,r_n$ be nonnegative integers, and let $p \in \mathbb{R}[z_1,\ldots, z_n]$ be a real stable polynomial such that the degree of  $z_i$ is at most $r_i$ for $i=1,\ldots,n$. Then the following polynomials are also real stable if they are not identically zero:
\begin{enumerate}[{\rm (i)}]
\item $p(z_1, z_2 ,\ldots, z_n)|_{z_1=a}\in\mathbb{R}[z_2,z_3,\ldots,z_n], \text{ for any } a\in\mathbb{R}$;
\item $p(z_1, z_2 ,\ldots, z_n)|_{z_1=z_2=x}\in\mathbb{R}[x,z_3,\ldots,z_n]$;
\item $(\sum_{i=1}^{n}a_i\cdot \partial_{z_i})\cdot p(z_1,\ldots,z_n),\ \text{ for any } a_1\geq 0,\ldots, a_n\geq 0$;
\item $\mathcal{R}_{z_i,r_i}^- p(z_1,\ldots,z_n),\quad\text{for any } i\in[n]$.
\end{enumerate}
\end{lemma}
\begin{proof}
Statements (i), (ii) and (iii) follow from \cite[Lemma 2.4 (d)]{wag}, Definition \ref{defstable} and \cite[Theorem 1.3]{Branden01}, respectively.
We next prove (iv).
Recall that $\mathcal{R}_{z_i,r_i}^-p=z_i^{r_i}\cdot p(z_1,\ldots,-z_i^{-1},\ldots,z_d)$. The real stability of $\mathcal{R}_{z_i,r_i}^-p$ then follows  from the fact that $\mathbf{Im}(z_i)>0$ if and only if $\mathbf{Im}(-z_i^{-1})>0$.
\end{proof}

\section{Generalized interlacing families}

In this section, we introduce the concept of generalized interlacing families. We first recall the definition of classical interlacing families.

\subsection{Interlacing families}
We start with introducing the concept of common interlacing for polynomials with positive leading coefficients. While our specific formulation may differ slightly from the definition in \cite[Page 355]{CS07}, it is fundamentally consistent and captures the same mathematical essence.
Recall that for a degree-$d$ real-rooted polynomial $p(x)$, its $i$-th largest root is denoted by $\lambda_i(p)$.

\begin{definition}{\rm \cite[Page 355]{CS07}}
\label{def2.1}
Let $f(x)$ and $g(x)$ be two real-rooted polynomials with positive leading coefficients, where $f(x)$ has degree $d$.
We say that $f(x)$ interlaces $g(x)$ if $\mathrm{deg}(g)\in\{d,d-1\}$, and 
the roots of $f(x)$ and $g(x)$ satisfy the following inequalities:
\begin{enumerate}[\rm (i)]
	\item If $\mathrm{deg}(g)=d$, then
	\begin{equation*}
\lambda_1(f)\geq \lambda_1(g)\geq \lambda_2(f)\geq \lambda_2(g)\geq \cdots 
\geq \lambda_d(f)\geq \lambda_d(g).
\end{equation*}
\item If $\mathrm{deg}(g)=d-1$, then
\begin{equation*}
\lambda_1(f)\geq \lambda_1(g)\geq \lambda_2(f)\geq \lambda_2(g)\geq \cdots \geq \lambda_{d-1}(f) \geq \lambda_{d-1}(g)
\geq \lambda_d(f).
\end{equation*}
\end{enumerate}
We say that real-rooted polynomials $p_1(x),\ldots,p_m(x)$ with positive leading coefficients have a common interlacing if there exists a polynomial $f(x)$ such that $f(x)$ interlaces $p_i(x)$ for each $i\in[m]$.
\end{definition}

Note that if a collection of real-rooted polynomials has a common interlacing, then they either have the same degree or have degrees differing by one. 
The following lemma has been discovered a number of times, which provides an equivalent characterization of common interlacing \cite{CS07,Ded92,Fell}.

\begin{lemma}{\rm \cite[Theorem 3.6]{CS07}}
\label{fell-interlacing}
Assume that $p_1(x),\ldots,p_m(x)$ are real-rooted polynomials with positive leading coefficients.
Then $p_1(x),\ldots,p_m(x)$ have a common interlacing if and only if $c_1p_1(x)+\cdots+c_mp_m(x)$ is real-rooted for all $c_i\geq 0$ with $\sum_{i=1}^{m}c_i=1$.
\end{lemma}

Marcus, Spielman, and Srivastava showed that the existence of a common interlacing allows us to  relate the roots of the expected polynomial to those of specific polynomials within the collection.

\begin{lemma}\label{lemma2.6}{\rm\cite[Lemma 2.11]{inter3}}
Assume that $p_1(x),\ldots,p_m(x)$ are degree-$d$ real-rooted polynomials with positive leading coefficients. 
Assume that $p_1(x),\ldots,p_m(x)$ have a common interlacing. 
Then for every integer $k\in[d]$ and for every nonnegative $c_1,\ldots,c_m$ with $\sum_{i=1}^mc_i=1$, there exist two integers $j_1,j_2\in[m]$ such that
\begin{equation*}
\lambda_k( p_{j_1})\geq \lambda_k\Big( \sum\limits_{i=1}^{m}c_ip_i\Big)
\geq \lambda_k( p_{j_2}).
\end{equation*}

\end{lemma}

Marcus, Spielman, and Srivastava next introduced the concept of interlacing family, which can be viewed as a generalization of common interlacing.

\begin{definition}{\rm\cite[Definition 2.5]{inter3}}
An interlacing family consists of a finite rooted tree $T$, together with a labeling of each node $v\in T$ by a real-rooted degree-$d$ polynomial $f_v(x)\in  \mathbb{R}[x]$ with positive leading coefficient. 
The labeling satisfies the following two properties: 
\begin{enumerate}
\item [\rm (i)] For every non-leaf node $v$, the polynomial $f_v(x)$ is a convex combination of the polynomials associated with its children.

\item [\rm (ii)] For any non-leaf node $v$, the polynomials corresponding to the children of $v$ have a common interlacing.

\end{enumerate}
For simplicity, we call a set of polynomials an interlacing family if they are the labels of the leaves of such a tree.

\end{definition}

By applying Lemma \ref{lemma2.6} iteratively across the layers of the rooted tree associated with an interlacing family,  the following theorem is obtained in \cite{inter3}.

\begin{theorem}\label{classical interlacing family}
{\rm\cite[Theorem 2.7]{inter3}}
Let $\mathcal{F}=\{f_i(x)\}_{i=1}^m$ be an interlacing family of degree-$d$ polynomials with positive leading coefficients. 
The root of the family is labeled by $f_{\emptyset}(x)$.
Then for any $k\in[d]$, there exist two integers $j_1,j_2\in[m]$ such that
\begin{equation*}
\lambda_k( f_{j_1})\geq \lambda_k( f_{\emptyset})
\geq \lambda_k( f_{j_2}).
\end{equation*}

\end{theorem}

\subsection{Generalized interlacing families}

The definition of the classical interlacing family has the restriction that all polynomials in the family have the same degree.
This will prevent its application, as a set of polynomials of interest may not satisfy this condition.
To navigate this obstacle, we introduce the concept of the generalized  interlacing family, which allows the polynomials in the family to have different degrees.

\begin{definition}\label{Generalized interlacing family}
A generalized  interlacing family consists of a finite rooted tree $T$ and a labeling of the nodes $v\in T$ by real-rooted polynomials $f_v(x)\in  \mathbb{R}[x]$ with positive leading coefficients.
The labeling satisfies the following three properties: 
\begin{enumerate}
\item [\rm (i)] For every non-leaf node $v$, the polynomial $f_v(x)$ is a convex combination of the polynomials associated with its children. 
\item[\rm (ii)] 
The degree of the  polynomial labelled to the root node, denoted by $f_{\emptyset}(x)$, equals the maximal degree of the polynomials labelled to the leaf nodes of the tree.

\item [\rm (iii)] 

For any non-leaf node $v$, if $\mathrm{deg}(f_v)=\mathrm{deg}(f_{\emptyset})$, then  the nonzero polynomials corresponding to the children of $v$ have degrees $\mathrm{deg}(f_{\emptyset})$ or $\mathrm{deg}(f_{\emptyset})-1$, and they have a common interlacing.

\end{enumerate}
We say that a set of polynomials forms a {\em  generalized interlacing family} if they are the polynomials associated with the leaf nodes of such a tree $T$.

\end{definition}

As a consequence of Definition \ref{Generalized interlacing family}, the degree of a polynomial associated with any non-root node of $T$ can be any integer between $0$ and $\mathrm{deg}(f_{\emptyset})$.
Note that a classical interlacing family is a special case of a generalized  interlacing family where the polynomials labelling to the nodes of the tree have the same degree.

Lemma \ref{lemma2.6} applies to the polynomials of the same degree. We now establish a generalized version that relaxes this degree restriction.

\begin{lemma}\label{lemma2.67}
Assume that $p_1(x),\ldots,p_m(x)$ are real-rooted polynomials with positive leading coefficients.
Assume that there is a nonempty subset $S\subset[m]$ such that $p_i(x)$ has degree $d$ when $i\in S$ and has degree $d-1$ when $i\notin S$.
Assume that $p_1(x),\ldots,p_m(x)$ have a common interlacing. 
Then for any nonnegative $c_1,\ldots,c_m$ with $\sum_{i\in S}^{}c_i>0$, there exists an integer $j\in S$ such that
\begin{equation}\label{eqxu_2025:sep:2-1}
\mathrm{minroot}\ p_j(x)
\geq \mathrm{minroot}\ 
\sum\limits_{i=1}^{m}c_ip_i(x).
\end{equation}

\end{lemma}

\begin{proof}
By assumption, there is a degree-$d$ polynomial $h(x)$ that interlaces  ${p}_i(x)$  for each $i\in[m]$.
By Definition \ref{def2.1}
we have 
\begin{equation}\label{claim-xu2}
\lambda_{d-1}(h)\geq  
\mathrm{minroot}\ p_l(x)\geq 
\lambda_{d}(h)\geq 
\mathrm{minroot}\  p_i(x),
\quad \text{ for all } i\in S, \text{ for all } l\notin S.
\end{equation}
Since $p_i(x), i\in S$ have the same degree and have a common interlacing,
by Lemma \ref{lemma2.6} there exists an integer $j\in S$ such that
\begin{equation}\label{eqxu_2025:sep:1-1}
\mathrm{minroot}\ p_j(x)\geq r:=
 \mathrm{minroot}\ \sum\limits_{i\in S}^{}c_ip_i(x).
\end{equation}
Combining with \eqref{claim-xu2}, we have
\begin{equation*}
\mathrm{minroot}\ p_l(x)\geq r,\quad \text{ for all }	l\notin S.
\end{equation*} 
For each $l\notin S$, note that $p_l(x)$ is a degree-$(d-1)$ polynomial with positive leading coefficient, so we have $(-1)^d\cdot p_l(r)\leq 0$. 
Hence, 
\begin{equation*}
(-1)^d \cdot\sum\limits_{i=1}^{m}c_ip_i(r)
=(-1)^d\cdot \sum\limits_{l\notin S}^{}c_lp_l(r)\leq 0,
\end{equation*}
which implies
\begin{equation}\label{eq:rzuixiao}
r\geq \mathrm{minroot}\ \sum\limits_{i=1}^{m}c_ip_i(x).
\end{equation}
Here, we use that $\sum_{i=1}^{m}c_ip_i(x)$ is a degree $d$ polynomial with positive leading coefficient.
Combining (\ref{eq:rzuixiao}) with \eqref{eqxu_2025:sep:1-1}, we arrive at \eqref{eqxu_2025:sep:2-1}.

\end{proof}

By applying Lemma \ref{lemma2.67} iteratively across the layers of the rooted tree associated with a generalized interlacing family, we obtain the following theorem.

\begin{theorem}\label{generalized interlacing family}
Let $\mathcal{F}=\{f_i(x)\}_{i=1}^m$ be a generalized interlacing family of polynomials with positive leading coefficients. 
The root of the family is labeled by $f_{\emptyset}(x)$.
There exists an integer $j\in[m]$ such that $\mathrm{deg}(f_j)=\mathrm{deg}(f_{\emptyset})$ and
\begin{equation}\label{eqxu_2025:sep:2-1-9}
\mathrm{minroot}\ f_{j}(x)\geq 
\mathrm{minroot}\ f_{\emptyset}(x).
\end{equation}

\end{theorem}

\begin{proof}
As a generalized interlacing family, $\mathcal{F}=\{f_i(x)\}_{i=1}^m$ constitutes the polynomial set corresponding to the leaf nodes of an associated tree $T$.
We prove the theorem by induction on the height of the tree.
The case of height 1 is established  by Lemma \ref{lemma2.67}.
 For the inductive hypothesis, assume the theorem holds for all trees of height less than $k$. We now consider a tree of height $k$.
Applying Lemma \ref{lemma2.67}, we deduce that for the root node, there exists a child $v$ and an associated polynomial $f_v(x)$ satisfying 
 $\mathrm{deg}(f_v)=\mathrm{deg}(f_{\emptyset})$ and 
\begin{equation*}
\mathrm{minroot}\ f_{v}(x)\geq 
\mathrm{minroot}\ f_{\emptyset}(x).
\end{equation*}
If $v$ is a leaf node, then we are done. If $v$ is not a leaf node, then we consider the subtree $T'$ rooted at $v$. 
Note that the polynomials labelled to the leaf nodes of $T'$ also form a generalized interlacing family. 
Also note that the leaf nodes of $T'$ is a subset of that of $T$.
Since $T'$ has height less than $k$, by assumption there is an $j\in[m]$, such that 
$\mathrm{deg}(f_j)=\mathrm{deg}(f_v)=\mathrm{deg}(f_{\emptyset})$ and 
\begin{equation*}
\mathrm{minroot}\ f_{j}(x)\geq \mathrm{minroot}\ f_{v}(x)\geq 
\mathrm{minroot}\ f_{\emptyset}(x).
\end{equation*}
Hence, we obtain the desired result.

\end{proof}

\begin{remark}
A key distinction between Theorem \ref{classical interlacing family} and Theorem \ref{generalized interlacing family} lies in the fact that the latter exclusively addresses the smallest root. This naturally raises the question of whether the conclusion of Theorem \ref{generalized interlacing family} can be extended to other roots. However, under the constraint $\deg(f_j) = \deg(f_{\emptyset})$, the analogous result to equation \eqref{eqxu_2025:sep:2-1-9} does not generally hold for the other roots.
To illustrate this, we examine the behavior of the largest root.
Consider $f_1(x)=x(x-2)(x-4)$ and $f_2(x)=(x-1)(x-3)$, which have a common interlacing and hence form a generalized interlacing family. The largest root of $f_1(x)+f_2(x)$ is approximately $3.7$, which is smaller than the largest root of $f_1(x)$, which is $4$. 
This example demonstrates that the inequality 
$
\mathrm{maxroot}\ f_1(x) \leq \mathrm{maxroot}\ f_1(x)+f_2(x)
$
 does not hold in general. Note that we have $\deg(f_1)=\deg(f_1+f_2)\neq \deg(f_2)$.
Conversely, if we choose $f_3(x)=(x-1)(x-5)$, the largest root of $f_1(x)+f_3(x)$ becomes approximately $4.25$, which exceeds the largest root of $f_1(x)$. This, in turn, shows that the inequality $\mathrm{maxroot}\ f_1(x) \geq \mathrm{maxroot}\ f_1(x)+f_3(x)$  is not always satisfied.
\end{remark}
\begin{remark}
The classical interlacing family requires all polynomials in the family to have the same degree, which may not be satisfied in practice.
For instance, a given set of degree-$d$ polynomials $f_1(x), \ldots, f_m(x)$ may not form an interlacing family under this definition, thereby precluding the use of classical interlacing polynomial method to estimate the maximum of their smallest roots.
The generalized interlacing family addresses this limitation by introducing auxiliary polynomials $h_1(x), \ldots, h_n(x)$ of lower degrees. 
If the combined collection $\{f_1(x), \ldots, f_m(x), h_1(x), \ldots, h_n(x)\}$ forms a generalized interlacing family,
then the smallest root of the expected polynomial $\sum_{i=1}^{m}f_i(x)+\sum_{j=1}^{n}h_j(x)$ can be used to estimate the maximum of the smallest roots among the original polynomials $f_1(x), \ldots, f_m(x)$.
\end{remark}

\section{Properties of the expected polynomials}

In this section, we present some basic properties of the polynomial $P_{k}(x;\boldsymbol{A},\boldsymbol{C},\boldsymbol{U},\boldsymbol{R})$ defined in Definition \ref{def-p-SW}.
These properties will be instrumental in our subsequent proofs.

\begin{definition}\label{def-H}
Let $\boldsymbol{A}\in\mathbb{R}^{n\times d}$, $\boldsymbol{C}\in\mathbb{R}^{n\times d_{{C}}}$,  $\boldsymbol{U}\in\mathbb{R}^{n_{{R}}\times d_C}$ and $\boldsymbol{R}\in\mathbb{R}^{n_{{R}}\times d}$. 
Let $\boldsymbol{Y}=\mathrm{diag}(y_1,\ldots,y_{d_C})$ and 
$\boldsymbol{Z}=\mathrm{diag}(z_1,\ldots,z_{n_R})$, where $y_i$ and $z_i$ are variables. 
We define the multivariate polynomial
\begin{equation*}
Q(x,\boldsymbol{Y},\boldsymbol{Z};\boldsymbol{A},\boldsymbol{C},\boldsymbol{U},\boldsymbol{R})
:=\det\left[\begin{matrix}
\boldsymbol{I}_n  & \boldsymbol{0} & \boldsymbol{C} & \boldsymbol{A}  \\
\boldsymbol{0}& \boldsymbol{Z} & \boldsymbol{U}  & \boldsymbol{R}  \\
\boldsymbol{C}^{T}  &  \boldsymbol{U}^T & \boldsymbol{Y} &  \boldsymbol{0} \\
\boldsymbol{A}^{T}  & \boldsymbol{R}^{T} & \boldsymbol{0} & x\cdot  \boldsymbol{I}_d
\end{matrix}\right]
\in\mathbb{R}[x,y_1,\ldots,y_{d_C},z_1,\ldots,z_{n_R}].
\end{equation*}
\end{definition}

In the following proposition, we show that 
the polynomial $p_{S,W}(x;\boldsymbol{A},\boldsymbol{C},\boldsymbol{U},\boldsymbol{R})$
defined in Definition \ref{def-p-SW} 
can be expressed in terms of the multivariate polynomial $Q(x,\boldsymbol{Y},\boldsymbol{Z};\boldsymbol{A},\boldsymbol{C},\boldsymbol{U},\boldsymbol{R})$.

\begin{prop}\label{P-exp-base-gcur1}
Let $\boldsymbol{A}\in\mathbb{R}^{n\times d}$, $\boldsymbol{C}\in\mathbb{R}^{n\times d_{{C}}}$,  $\boldsymbol{U}\in\mathbb{R}^{n_{{R}}\times d_C}$ and $\boldsymbol{R}\in\mathbb{R}^{n_{{R}}\times d}$. 
Let $\boldsymbol{Y}=\mathrm{diag}(y_1,\ldots,y_{d_C})$ and 
$\boldsymbol{Z}=\mathrm{diag}(z_1,\ldots,z_{n_R})$, where $y_i$ and $z_i$ are variables. 
Let $k$ be an integer satisfying $0\leq k\leq \mathrm{rank}(\boldsymbol{U})$. 
Let $S\in \binom{[n_R]}{k}$ and $W\in \binom{[d_C]}{k}$.

\begin{enumerate}
\item[\rm (i)] We have
\begin{equation*}
\begin{aligned}
p_{S,W}(-x;\boldsymbol{A},\boldsymbol{C},\boldsymbol{U},\boldsymbol{R})
=(-1)^{d-k}\cdot  \partial_{\boldsymbol{Y}^{W^C}} \partial_{\boldsymbol{Z}^{S^C}}\ 
Q(x,\boldsymbol{Y},\boldsymbol{Z};\boldsymbol{A},\boldsymbol{C},\boldsymbol{U},\boldsymbol{R})  \ \Bigg\vert_{\substack{y_i=0,\forall i\in[d_C]\\z_j=0,\forall j\in[n_R]}},
\end{aligned}
\end{equation*}
where $\partial_{\boldsymbol{Y}^{W^C}}=\prod_{i\in [d_C]\backslash W}\partial_{y_i}$, 
$\partial_{\boldsymbol{Z}^{S^C}}=\prod_{i\in [n_R]\backslash S}\partial_{z_i}$. 

\item[\rm (ii)]  If $\mathrm{rank}(\boldsymbol{U}_{S,W})=k$, then
\begin{equation*}
p_{S,W}(x;\boldsymbol{A},\boldsymbol{C},\boldsymbol{U},\boldsymbol{R})	
=\det[\boldsymbol{U}_{S,W}]^2\cdot 
\det[x\cdot \boldsymbol{I}_d+
(\boldsymbol{A}- \boldsymbol{C}_{:,W}( \boldsymbol{U}_{S,W})^{-1}  \boldsymbol{R}_{S,:})^T
(\boldsymbol{A}- \boldsymbol{C}_{:,W}( \boldsymbol{U}_{S,W})^{-1}  \boldsymbol{R}_{S,:})].
\end{equation*}

\item[\rm (iii)] If either  $
\begin{bmatrix}
\boldsymbol{C}_{:,W}\\
\boldsymbol{U}_{S,W}	
\end{bmatrix}$ or $\begin{bmatrix}
\boldsymbol{U}_{S,W} &	\boldsymbol{R}_{S,:}
\end{bmatrix}$ has rank less than $k$, then $p_{S,W}(x;\boldsymbol{A},\boldsymbol{C},\boldsymbol{U},\boldsymbol{R})$ is identically zero.
Otherwise, $p_{S,W}(x;\boldsymbol{A},\boldsymbol{C},\boldsymbol{U},\boldsymbol{R})$ is a degree $d-k+\mathrm{rank}(\boldsymbol{U}_{S,W})$ polynomial with  positive leading coefficient. 

\end{enumerate}
\end{prop}

\begin{proof}
See Appendix \ref{proof-P-exp-base-gcur1}.	
\end{proof}

\begin{prop}\label{P-exp-base-gcur2}
Let $\boldsymbol{A}\in\mathbb{R}^{n\times d}$, $\boldsymbol{C}\in\mathbb{R}^{n\times d_{{C}}}$,  $\boldsymbol{U}\in\mathbb{R}^{n_{{R}}\times d_C}$ and $\boldsymbol{R}\in\mathbb{R}^{n_{{R}}\times d}$. 
Let $k$ be an integer satisfying $0\leq k\leq \mathrm{rank}(\boldsymbol{U})$. 
Then we have 
\begin{equation}
\label{eq:P-exp-base-qpsw}
\begin{aligned}
P_{k}(x;\boldsymbol{A},\boldsymbol{C},\boldsymbol{U},\boldsymbol{R})
=\frac{(-1)^{d-k}}{ (d_C-k)!\cdot(n_R-k)!}\cdot \partial_y^{d_C-k} \partial_z^{n_R-k}\ 
\det\left[\begin{matrix}
\boldsymbol{I}_n  & \boldsymbol{0} & \boldsymbol{C} & \boldsymbol{A}  \\
\boldsymbol{0}& z\cdot \boldsymbol{I}_{n_R} & \boldsymbol{U}  & \boldsymbol{R}  \\
\boldsymbol{C}^{T}  &  \boldsymbol{U}^T & y\cdot \boldsymbol{I}_{d_C} &  \boldsymbol{0} \\
\boldsymbol{A}^{T}  & \boldsymbol{R}^{T} & \boldsymbol{0} & -x\cdot  \boldsymbol{I}_d
\end{matrix}\right] 
\ \Bigg\vert_{y=z=0},
\end{aligned}
\end{equation}
where $P_{k}(x;\boldsymbol{A},\boldsymbol{C},\boldsymbol{U},\boldsymbol{R})$ is defined in Definition \ref{def-p-SW}.
Moreover, $P_{k}(x;\boldsymbol{A},\boldsymbol{C},\boldsymbol{U},\boldsymbol{R})$ is a degree-$d$ real-rooted polynomial with a  positive leading coefficient. 
\end{prop}

\begin{proof}
Note that for any multiaffine polynomial $f(x_1,\ldots,x_m)$ we have
\begin{equation}\label{meq31}
\sum_{\substack{S\subset[m], |S|=k}}\partial_{\boldsymbol{X}^{S^C}}	\ f(x_1,\ldots,x_m)\ \bigg|_{x_i=x,\forall i}=\frac{1}{(m-k)!}\cdot \partial_{x}^{m-k}	\ f(x,\ldots,x),
\end{equation}
where $\partial_{\boldsymbol{X}^{S^C}}=\prod_{i\in [m]\backslash S}\partial_{x_i}$.
Therefore, combining \eqref{meq31} with Proposition \ref{P-exp-base-gcur1} (i), 
we have
\begin{equation*}
\begin{aligned}
P_{k}(-x;\boldsymbol{A},\boldsymbol{C},\boldsymbol{U},\boldsymbol{R})
&=(-1)^{d-k}
\sum_{S\in\binom{[n_R]}{k},W\in\binom{[d_C]}{k}}
\partial_{\boldsymbol{Y}^{W^C}} \partial_{\boldsymbol{Z}^{S^C}}\ 
Q(x,\boldsymbol{Y},\boldsymbol{Z};\boldsymbol{A},\boldsymbol{C},\boldsymbol{U},\boldsymbol{R})  \ \Bigg\vert_{\substack{y_i=0,\forall i\in[d_C]\\z_j=0,\forall j\in[n_R]}}\\
&=\frac{(-1)^{d-k}}{ (d_C-k)!\cdot(n_R-k)!}\cdot \partial_y^{d_C-k} \partial_z^{n_R-k}\ 
\det\left[\begin{matrix}
\boldsymbol{I}_n  & \boldsymbol{0} & \boldsymbol{C} & \boldsymbol{A}  \\
\boldsymbol{0}& z\cdot \boldsymbol{I}_{n_R} & \boldsymbol{U}  & \boldsymbol{R}  \\
\boldsymbol{C}^{T}  &  \boldsymbol{U}^T & y\cdot \boldsymbol{I}_{d_C} &  \boldsymbol{0} \\
\boldsymbol{A}^{T}  & \boldsymbol{R}^{T} & \boldsymbol{0} & x\cdot  \boldsymbol{I}_d
\end{matrix}\right] 
\ \Bigg\vert_{y=z=0},\\
\end{aligned}	
\end{equation*}
which immediately implies \eqref{eq:P-exp-base-qpsw}.

Combining Lemma \ref{real stable} with Lemma \ref{real stable2}, we see that $P_{k}(-x;\boldsymbol{A},\boldsymbol{C},\boldsymbol{U},\boldsymbol{R})$ is real stable and hence real-rooted.
Therefore, $P_{k}(x;\boldsymbol{A},\boldsymbol{C},\boldsymbol{U},\boldsymbol{R})$ is also real-rooted.
Proposition \ref{P-exp-base-gcur1} shows that $p_{S,W}(x;\boldsymbol{A},\boldsymbol{C},\boldsymbol{U},\boldsymbol{R})$ is either a degree-$d$ polynomial with positive leading coefficient, or a polynomial of degree less than $d$. 
Since $k\leq \mathrm{rank}(\boldsymbol{U})$, there exists at least one polynomial $p_{S_0,W_0}(x;\boldsymbol{A},\boldsymbol{C},\boldsymbol{U},\boldsymbol{R})$ that has degree $d$.
Therefore, $P_{k}(x;\boldsymbol{A},\boldsymbol{C},\boldsymbol{U},\boldsymbol{R})$ is a degree-$d$ real-rooted polynomial with positive leading coefficient.
This completes the proof.

\end{proof}

In the special case when $\boldsymbol{C}=\boldsymbol{U}=\boldsymbol{R}=\boldsymbol{A}$, we show that $P_{k}(-x;\boldsymbol{A},\boldsymbol{C},\boldsymbol{U},\boldsymbol{R})$ can be expressed in terms of the flip operator $\mathcal{R}_{x,d}^+$ and the Laguerre derivative operator $\partial_x\cdot x\cdot \partial_x$.

\begin{prop}\label{P-exp-base-gcur-x}
Let $\boldsymbol{A}\in\mathbb{R}^{n\times d}$. 
Let $k$ be an integer satisfying $0\leq k\leq \mathrm{rank}(\boldsymbol{A})$. 
Then we have
\begin{equation}
\label{eq:P-exp-base-p1}
P_{k}(-x;\boldsymbol{A},\boldsymbol{A},\boldsymbol{A},\boldsymbol{A})
=\frac{(-1)^{d-k}}{(k!)^2}\cdot  \mathcal{R}_{x,d}^+\cdot  (\partial_x\cdot x\cdot \partial_x)^k \cdot \mathcal{R}_{x,d}^+\   \det[x\cdot \boldsymbol{I}_d-\boldsymbol{A}^T\boldsymbol{A}].
\end{equation}

\end{prop}

\begin{proof}
See Appendix \ref{proof-P-exp-base-gcur-x}.	
\end{proof}

We next simplify the formula of $P_{k}(-x;\boldsymbol{A},\boldsymbol{C},\boldsymbol{U},\boldsymbol{R})$ when $\boldsymbol{U}=\boldsymbol{C}\in\mathbb{R}^{n\times k}$ and $\boldsymbol{R}=\boldsymbol{A}\in\mathbb{R}^{n\times d}$.
\begin{prop}\label{P-exp-base-gcur-x-newcase}
Let $\boldsymbol{A}\in\mathbb{R}^{n\times d}$ and $\boldsymbol{C}\in\mathbb{R}^{n\times k}$ with $\mathrm{rank}(\boldsymbol{C})=k$. 
Then we have
\begin{equation}
\label{eq:P-exp-base-p1-newcase}
P_{k}(-x;\boldsymbol{A},\boldsymbol{C},\boldsymbol{C},\boldsymbol{A})
=\frac{(-1)^{d}\det\boldsymbol{C}^T\boldsymbol{C}}{k!}\cdot  \mathcal{R}_{x,d}^+\cdot  \partial_x^k\cdot x^k \cdot \mathcal{R}_{x,d}^+\   \det[x\cdot \boldsymbol{I}_d-\boldsymbol{A}^T(\boldsymbol{I}_n-\boldsymbol{C}\boldsymbol{C}^{\dagger})\boldsymbol{A}].
\end{equation}

\end{prop}

\begin{proof}
See Appendix \ref{proof-P-exp-base-gcur-x-newcase}.	
\end{proof}

\section{Generalized CUR matrix decompositions: Proof of Theorem \ref{gcur-th2}}

The aim of this section is to prove Theorem \ref{gcur-th2}. We begin by defining $f_{\widetilde{S},\widetilde{W}}(x;\boldsymbol{A},\boldsymbol{C},\boldsymbol{U},\boldsymbol{R})$, which serves as a labeling of the nodes $v\in T$, where $T$ is the tree corresponding to a generalized interlacing family.
For convenience, we assume throughout this section that the sampling size $k$ is implicitly encoded within the definition of the polynomial $f_{\widetilde{S},\widetilde{W}}(x;\boldsymbol{A},\boldsymbol{C},\boldsymbol{U},\boldsymbol{R})$.

\begin{definition}
\label{def-f-SW}
Let $\boldsymbol{A}\in\mathbb{R}^{n\times d}$, $\boldsymbol{C}\in\mathbb{R}^{n\times d_{{C}}}$,  $\boldsymbol{U}\in\mathbb{R}^{n_{{R}}\times d_C}$ and $\boldsymbol{R}\in\mathbb{R}^{n_{{R}}\times d}$. 
Let $k$ be an integer satisfying $0\leq k\leq \mathrm{rank}(\boldsymbol{U})$.
For any two subsets $\widetilde{S}\subset[n_R]$ and $\widetilde{W}\subset[d_C]$ with $|\widetilde{S}|\leq k$ and $|\widetilde{W}|\leq k$, we define
\begin{equation}\label{summation1}
\begin{aligned}
f_{\widetilde{S},\widetilde{W}}(x;\boldsymbol{A},\boldsymbol{C},\boldsymbol{U},\boldsymbol{R})
=
\sum_{\substack{S\in\binom{[n_R]}{k}\\ \widetilde{S}\subset S}}
\sum_{\substack{W\in\binom{[d_C]}{k}\\ \widetilde{W}\subset W}}
p_{S,W}(x;\boldsymbol{A},\boldsymbol{C},\boldsymbol{U},\boldsymbol{R}),
\end{aligned}	
\end{equation}
where $p_{S,W}(x;\boldsymbol{A},\boldsymbol{C},\boldsymbol{U},\boldsymbol{R})$ is defined in Definition \ref{def-p-SW}.
\end{definition}

We first present several basic properties of $f_{\widetilde{S},\widetilde{W}}(x;\boldsymbol{A},\boldsymbol{C},\boldsymbol{U},\boldsymbol{R})$.

\begin{lemma}\label{prop-xujia1}
Let $\boldsymbol{A}\in\mathbb{R}^{n\times d}$, $\boldsymbol{C}\in\mathbb{R}^{n\times d_{{C}}}$,  $\boldsymbol{U}\in\mathbb{R}^{n_{{R}}\times d_C}$ and $\boldsymbol{R}\in\mathbb{R}^{n_{{R}}\times d}$. 
Let $k$ be an integer satisfying $0\leq k\leq \mathrm{rank}(\boldsymbol{U})$.
Let $\widetilde{S}\subset[n_R]$ and $\widetilde{W}\subset[d_C]$ be two subsets satisfying $|\widetilde{S}|\leq k$ and $|\widetilde{W}|\leq k$. Then we have the following results.
\begin{enumerate}[{\rm (i)}]
	\item The polynomial $f_{\widetilde{S},\widetilde{W}}(x;\boldsymbol{A},\boldsymbol{C},\boldsymbol{U},\boldsymbol{R})$ is identically zero if 
$\mathcal{T}=\emptyset$, where
\begin{equation*}
\mathcal{T}_{}:=
\left\{(S,W)\in \binom{[n_R]}{k}\times \binom{[d_C]}{k}:
\widetilde{S}\subset S, \widetilde{W}\subset W,
\mathrm{rank}
\left(
\begin{matrix}
\boldsymbol{C}_{:,W}\\
\boldsymbol{U}_{S,W}	
\end{matrix}
\right)
=\mathrm{rank}
\left(
\begin{matrix}
\boldsymbol{U}_{S,W} &	\boldsymbol{R}_{S,:}
\end{matrix}
\right)=k
\right\}.	
\end{equation*}
If $\mathcal{T}\neq\emptyset$, then
$f_{\widetilde{S},\widetilde{W}}(x;\boldsymbol{A},\boldsymbol{C},\boldsymbol{U},\boldsymbol{R})$ is a degree-$(d-k+t)$ polynomial with positive leading coefficient, where $t:=\max_{(S,W)\in\mathcal{T}} \mathrm{rank}(\boldsymbol{U}_{S,W})$.
In particular, $f_{\widetilde{S},\widetilde{W}}(x;\boldsymbol{A},\boldsymbol{C},\boldsymbol{U},\boldsymbol{R})$ has degree $d$ if and only if there is a nonsingular submatrix $\boldsymbol{U}_{S,W}\in\mathbb{R}^{k\times k}$ such that $\widetilde{S}\subset S$ and $\widetilde{W}\subset W$.

\item If $|\widetilde{S}|< k$ then
\begin{equation}\label{combination1-1}
\begin{aligned}
\sum_{i\notin\widetilde{S}}^{} f_{\widetilde{S}\cup \{{i}\},\widetilde{W}}(x;\boldsymbol{A},\boldsymbol{C},\boldsymbol{U},\boldsymbol{R})
&=(k-|\widetilde{S}|)\cdot f_{\widetilde{S},\widetilde{W}}(x;\boldsymbol{A},\boldsymbol{C},\boldsymbol{U},\boldsymbol{R}).
\end{aligned}	
\end{equation}
If $|\widetilde{W}|< k$ then
\begin{equation}\label{combination1-2}
\begin{aligned}
\sum_{i\notin\widetilde{W}}^{} f_{\widetilde{S},\widetilde{W}\cup \{{i}\}}(x;\boldsymbol{A},\boldsymbol{C},\boldsymbol{U},\boldsymbol{R})
=(k-|\widetilde{W}|)\cdot f_{\widetilde{S},\widetilde{W}}(x;\boldsymbol{A},\boldsymbol{C},\boldsymbol{U},\boldsymbol{R}).
\end{aligned}	
\end{equation}

\end{enumerate}

\end{lemma}

\begin{proof}
(i) The degree of $f_{\widetilde{S},\widetilde{W}}(x;\boldsymbol{A},\boldsymbol{C},\boldsymbol{U},\boldsymbol{R})$ is determined by the maximum degree among the polynomials $p_{S,W}(x;\boldsymbol{A},\boldsymbol{C},\boldsymbol{U},\boldsymbol{R})$ on the right-hand side of \eqref{summation1}.	
According to Proposition \ref{P-exp-base-gcur1} (iii), if $\mathcal{T}=\emptyset$, then  $f_{\widetilde{S},\widetilde{W}}(x;\boldsymbol{A},\boldsymbol{C},\boldsymbol{U},\boldsymbol{R})$ is identically zero. Otherwise,   $f_{\widetilde{S},\widetilde{W}}(x;\boldsymbol{A},\boldsymbol{C},\boldsymbol{U},\boldsymbol{R})$ is a degree-$(d-k+t)$ polynomial with positive leading coefficient.
Combining with Proposition \ref{P-exp-base-gcur1} (ii), we see that $f_{\widetilde{S},\widetilde{W}}(x;\boldsymbol{A},\boldsymbol{C},\boldsymbol{U},\boldsymbol{R})$ has degree $d$ precisely when there is a nonsingular submatrix $\boldsymbol{U}_{S,W}\in\mathbb{R}^{k\times k}$ such that $\widetilde{S}\subset S$ and $\widetilde{W}\subset W$.

(ii) We first prove \eqref{combination1-1}. A direct calculation shows that
\begin{equation*}
\begin{aligned}
\sum_{i\notin\widetilde{S}}^{} f_{\widetilde{S}\cup \{{i}\},\widetilde{W}}(x;\boldsymbol{A},\boldsymbol{C},\boldsymbol{U},\boldsymbol{R})
&=\sum_{i\notin\widetilde{S}}^{} 
\sum_{\substack{S\in\binom{[n_R]}{k}\\ \widetilde{S}\cup\{i\}\subset S}}
\sum_{\substack{W\in\binom{[d_C]}{k}\\ \widetilde{W}\subset W}}
p_{S,W}(x;\boldsymbol{A},\boldsymbol{C},\boldsymbol{U},\boldsymbol{R})\\
&=(k-|\widetilde{S}|)\cdot 
\sum_{\substack{S\in\binom{[n_R]}{k}\\ \widetilde{S}\subset S}}
\sum_{\substack{W\in\binom{[d_C]}{k}\\ \widetilde{W}\subset W}}
p_{S,W}(x;\boldsymbol{A},\boldsymbol{C},\boldsymbol{U},\boldsymbol{R})\\
&=(k-|\widetilde{S}|)\cdot f_{\widetilde{S},\widetilde{W}}(x;\boldsymbol{A},\boldsymbol{C},\boldsymbol{U},\boldsymbol{R}).
\end{aligned}	
\end{equation*}
We can prove \eqref{combination1-2} using a similar calculation. This completes the proof. 

\end{proof}

We next present several alternative expressions for $f_{\widetilde{S},\widetilde{W}}(x;\boldsymbol{A},\boldsymbol{C},\boldsymbol{U},\boldsymbol{R})$.

\begin{lemma}\label{lemma2-f-SW}
\begin{enumerate}
\item[\rm (i)] 
Let $\boldsymbol{A}\in\mathbb{R}^{n\times d}$, $\boldsymbol{C}\in\mathbb{R}^{n\times d_{{C}}}$,  $\boldsymbol{U}\in\mathbb{R}^{n_{{R}}\times d_C}$ and $\boldsymbol{R}\in\mathbb{R}^{n_{{R}}\times d}$. 
Let $k$ be an integer satisfying $0\leq k\leq \mathrm{rank}(\boldsymbol{U})$.
Let $\widetilde{S}\subset[n_R]$ and $\widetilde{W}\subset[d_C]$ be two subsets of size at most $k$.
Then $f_{\widetilde{S},\widetilde{W}}(x;\boldsymbol{A},\boldsymbol{C},\boldsymbol{U},\boldsymbol{R})$ is either identically zero or real-rooted. Moreover, we have 
\begin{equation}\label{meq30}
\begin{aligned}
f_{\widetilde{S},\widetilde{W}}(-x;\boldsymbol{A},\boldsymbol{C},\boldsymbol{U},\boldsymbol{R})
&=\frac{(-1)^{d-k}}{(k-|\widetilde{W}|)!\cdot (k-|\widetilde{S}|)!}
\partial_y^{k-|\widetilde{W}|}
\partial_z^{k-|\widetilde{S}|}
g(x,y,z)
\ \Bigg\vert_{y=z=0},\\ 
\end{aligned}	
\end{equation}
where
\begin{equation}\label{meq32}
\begin{aligned}
g(x,y,z)&:=	\partial_{\boldsymbol{Y}^{\widetilde{W}}}
\partial_{\boldsymbol{Z}^{\widetilde{S}}} 
\Big(\prod_{i=1}^{d_C}\mathcal{R}_{y_i,1}^-\Big)
\Big(\prod_{i=1}^{n_R}\mathcal{R}_{z_i,1}^-\Big)
Q(x,\boldsymbol{Y},\boldsymbol{Z};\boldsymbol{A},\boldsymbol{C},\boldsymbol{U},\boldsymbol{R})\ \bigg|_{\substack{y_i=y,\forall i\in[d_C]\\ z_i=z,\forall i\in[n_R]}},\\
\end{aligned}
\end{equation}

\item[\rm (ii)] 
Let $\boldsymbol{A}\in\mathbb{R}^{n\times d}$.
Let $l$ and $k$ be two integers satisfying $0\leq l\leq  k\leq \mathrm{rank}(\boldsymbol{A})$.
Let $\widetilde{S}\subset[n_R]$ and $\widetilde{W}\subset[d_C]$ be two $l$-subsets such that
 $\boldsymbol{A}_{\widetilde{S},\widetilde{W}}\in\mathbb{R}^{l\times l}$ is invertible.
Then 
\begin{equation}
\label{meq35}
\begin{aligned}
f_{\widetilde{S},\widetilde{W}}(-x;\boldsymbol{A},\boldsymbol{A},\boldsymbol{A},\boldsymbol{A})
=\frac{(-1)^{d-k+l}\det[\boldsymbol{A}_{\widetilde{S},\widetilde{W}}]^2  }{((k-l)!)^2}\cdot  
\mathcal{R}_{x,d}^+\cdot (\partial_x\cdot x \cdot \partial_x)^{k-l} \cdot \mathcal{R}_{x,d}^+\   \det[x\cdot \boldsymbol{I}_d-\boldsymbol{B}^T\boldsymbol{B}],
\end{aligned}	
\end{equation}
where $\boldsymbol{B}=\boldsymbol{A}-\boldsymbol{A}_{:,\widetilde{W}}(\boldsymbol{A}_{\widetilde{S},\widetilde{W}})^{-1}\boldsymbol{A}_{\widetilde{S},:}\in\mathbb{R}^{n\times d}$.
Consequently, we have
\begin{equation}\label{meq350}
\mathrm{minroot}\ f_{\widetilde{S},\widetilde{W}}(x;\boldsymbol{A},\boldsymbol{A},\boldsymbol{A},\boldsymbol{A})
=-	\mathrm{maxroot}\ \mathcal{R}_{x,d}^+\cdot (\partial_x\cdot x \cdot \partial_x)^{k-l} \cdot \mathcal{R}_{x,d}^+\   \det[x\cdot \boldsymbol{I}_d-\boldsymbol{B}^T\boldsymbol{B}].
\end{equation}

\end{enumerate}

\end{lemma}

\begin{proof}
See Appendix \ref{proof-lemma2-f-SW}.	
\end{proof}

\begin{lemma}\label{lemma-f-SW}
Let $\boldsymbol{A}\in\mathbb{R}^{n\times d}$, $\boldsymbol{C}\in\mathbb{R}^{n\times d_{{C}}}$,  $\boldsymbol{U}\in\mathbb{R}^{n_{{R}}\times d_C}$ and $\boldsymbol{R}\in\mathbb{R}^{n_{{R}}\times d}$. 
Let $k$ be an integer satisfying $0\leq k\leq \mathrm{rank}(\boldsymbol{U})$.
Let $\widetilde{S}\subset[n_R]$ and $\widetilde{W}\subset[d_C]$ be two subsets of size at most $k$.
Assume that $f_{\widetilde{S},\widetilde{W}}(x;\boldsymbol{A},\boldsymbol{C},\boldsymbol{U},\boldsymbol{R})$ is not identically zero.

\begin{itemize}
\item[\rm (i)] 
If $|\widetilde{S}|<k$,
then the nonzero polynomials in the set $\{f_{\widetilde{S}\cup \{i\},\widetilde{W}}(x;\boldsymbol{A},\boldsymbol{C},\boldsymbol{U},\boldsymbol{R}) : i \in [n_R] \setminus \widetilde{S}\}$ have a common interlacing.

\item[\rm (ii)] 
If $|\widetilde{W}|<k$, then the nonzero polynomials in the set $\{f_{\widetilde{S},\widetilde{W}\cup \{i\}}(x;\boldsymbol{A},\boldsymbol{C},\boldsymbol{U},\boldsymbol{R}) : i \in [d_C] \setminus \widetilde{W}\}$ have a common interlacing.

\end{itemize}

\end{lemma}

\begin{proof}
(i) 
Define the set $\mathcal{A}:=\{i\in [n_R]\backslash\widetilde{S}: f_{\widetilde{S}\cup \{i\},\widetilde{W}}(x;\boldsymbol{A},\boldsymbol{C},\boldsymbol{U},\boldsymbol{R})\not\equiv0 \}$.
By Lemma \ref{lemma2-f-SW}, we have
\begin{equation*}
\begin{aligned}
&\sum_{i\in \mathcal{A} }	 c_i \cdot f_{\widetilde{S}\cup \{i\},\widetilde{W}}(x;\boldsymbol{A},\boldsymbol{C},\boldsymbol{U},\boldsymbol{R})
&=\frac{(-1)^{d-k}}{(k-|\widetilde{W}|)!(k-|\widetilde{S}|-1)!}
\partial_y^{k-|\widetilde{W}|}
\partial_z^{k-|\widetilde{S}|-1}
h(x,y,z)
\ \Bigg\vert_{y=z=0},\\ 
\end{aligned}	
\end{equation*}
where
\begin{equation*}
h(x,y,z):=	\partial_{\boldsymbol{Y}^{\widetilde{W}}}
\partial_{\boldsymbol{Z}^{\widetilde{S}}} 
\Big(\sum_{i\in  \mathcal{A}}c_i\partial_{\boldsymbol{z}_i}\Big)\cdot 
\Big(\prod_{i=1}^{d_C}\mathcal{R}_{y_i,1}^-\Big)
\Big(\prod_{i=1}^{n_R}\mathcal{R}_{z_i,1}^-\Big)
Q(x,\boldsymbol{Y},\boldsymbol{Z};\boldsymbol{A},\boldsymbol{C},\boldsymbol{U},\boldsymbol{R}) \ \bigg|_{\substack{y_i=y,\forall i\in[d_C]\\ z_i=z,\forall i\in[n_R]}}.
\end{equation*}
By Lemma \ref{real stable} and Lemma \ref{real stable2}, we see that $h(x,y,z)$ is identically zero or real stable for any $c_i\geq 0$. 
This implies that $\sum_{i\in \mathcal{A} }	 c_i \cdot f_{\widetilde{S}\cup \{i\},\widetilde{W}}(x;\boldsymbol{A},\boldsymbol{C},\boldsymbol{U},\boldsymbol{R})$ is real-rooted for any $c_i\geq 0$ with $\sum_{i\in \mathcal{A} }	 c_i=1$. 
Note that $f_{\widetilde{S}\cup \{i\},\widetilde{W}}(x;\boldsymbol{A},\boldsymbol{C},\boldsymbol{U},\boldsymbol{R})$ has positive leading coefficient for each $i\in\mathcal{A}$.
By Lemma \ref{fell-interlacing}, we see that the polynomials $f_{\widetilde{S}\cup \{i\},\widetilde{W}}(x;\boldsymbol{A},\boldsymbol{C},\boldsymbol{U},\boldsymbol{R})$, $i\in \mathcal{A}$, have a common interlacing.
This completes the proof of (i).

(ii) The proof is similar with that of  (i). For brevity, we omit the proof.

\end{proof}

\begin{lemma}\label{forms GIF}
Let $\boldsymbol{A}\in\mathbb{R}^{n\times d}$, $\boldsymbol{C}\in\mathbb{R}^{n\times d_{{C}}}$,  $\boldsymbol{U}\in\mathbb{R}^{n_{{R}}\times d_C}$ and $\boldsymbol{R}\in\mathbb{R}^{n_{{R}}\times d}$. 
Let $k$ be an integer satisfying $0\leq k\leq \mathrm{rank}(\boldsymbol{U})$.
The polynomial set
\begin{equation}\label{pk-set}	
\mathcal{P}_k=
\Big\{
p_{S,W}(x;\boldsymbol{A},\boldsymbol{C},\boldsymbol{U},\boldsymbol{R})
\Big\}_{S\in\binom{[n_R]}{k}, W\in\binom{[d_C]}{k}}
\end{equation}
forms a generalized interlacing family.	Here, $p_{S,W}(x;\boldsymbol{A},\boldsymbol{C},\boldsymbol{U},\boldsymbol{R})$ is defined in  Definition \ref{def-p-SW}.
\end{lemma}

\begin{proof}

We construct a tree $T$ with levels indexed from $0$ to $2k$. Each node in the tree is a pair $(\widetilde{S},\widetilde{W})$, where $\widetilde{S} \subseteq [n_R]$ and $\widetilde{W} \subseteq [d_C]$ are subsets, each of size at most $k$.
The root node (level $0$) is $(\emptyset,\emptyset)$.
   Nodes at even levels $i \in \{0, 2, \ldots, 2k\}$ have the form $(\widetilde{S},\widetilde{W})$, where $\widetilde{S} \subseteq [n_R]$ and $\widetilde{W} \subseteq [d_C]$ satisfy $|\widetilde{S}|=|\widetilde{W}|=\frac{i}{2}$.
   Each such node $(\widetilde{S},\widetilde{W})$ has children of the form $(\widetilde{S}\cup\{j\},\widetilde{W})$, where $j \in [n_R]\setminus \widetilde{S}$. This transition generates children for the next (odd) level by adding an element to the first subset.
   Consequently, nodes at odd levels $i \in \{1, 3, \ldots, 2k-1\}$ have the form $(\widetilde{S},\widetilde{W})$, where $\widetilde{S} \subseteq [n_R]$ satisfies $|\widetilde{S}|=\frac{i+1}{2}$ and $\widetilde{W} \subseteq [d_C]$ satisfies $|\widetilde{W}|=\frac{i-1}{2}$.
   Each such node $(\widetilde{S},\widetilde{W})$ has children of the form $(\widetilde{S},\widetilde{W}\cup\{j\})$, where $j \in [d_C]\setminus \widetilde{W}$. This transition generates children for the next (even) level by adding an element to the second subset.

Every node $(\widetilde{S},\widetilde{W})$ is labeled with a polynomial $f_{\widetilde{S},\widetilde{W}}(x;\boldsymbol{A},\boldsymbol{C},\boldsymbol{U},\boldsymbol{R})$.
   Specifically, the root node $(\emptyset,\emptyset)$ is labeled with the polynomial $f_{\emptyset,\emptyset}(x;\boldsymbol{A},\boldsymbol{C},\boldsymbol{U},\boldsymbol{R})=P_k(x;\boldsymbol{A},\boldsymbol{C},\boldsymbol{U},\boldsymbol{R})$.
   The leaf nodes, which are at level $2k$ and thus have the form $(S,W)$ where $S\in\binom{[n_R]}{k}$ and $W\in\binom{[d_C]}{k}$, are labeled with $f_{S,W}(x;\boldsymbol{A},\boldsymbol{C},\boldsymbol{U},\boldsymbol{R})=p_{S,W}(x;\boldsymbol{A},\boldsymbol{C},\boldsymbol{U},\boldsymbol{R})$.

We now prove that the tree $T$ satisfies (i), (ii) and (iii)  in Definition \ref{Generalized interlacing family}.
By Lemma \ref{prop-xujia1} we see that the condition (i) is satisfied.
Since $k\leq \mathrm{rank}(\boldsymbol{U})$,  there exists a nonsingular submatrix $\boldsymbol{U}_{S_0,W_0}\in\mathbb{R}^{k\times k}$.
Therefore, by Lemma \ref{prop-xujia1}, we have
 \[
\mathrm{deg}(f_{\emptyset,\emptyset}(x;\boldsymbol{A},\boldsymbol{C},\boldsymbol{U},\boldsymbol{R}))=\mathrm{deg}(f_{S_0,W_0}(x;\boldsymbol{A},\boldsymbol{C},\boldsymbol{U},\boldsymbol{R}))=d.
\] 
Hence, the condition (ii) is satisfied.
Lemma \ref{lemma-f-SW} (i) and (ii) show that for any non-leaf node $(\widetilde{S},\widetilde{W})$ of the tree, if $f_{\widetilde{S},\widetilde{W}}(x;\boldsymbol{A},\boldsymbol{C},\boldsymbol{U},\boldsymbol{R})$ is not identically zero, then the nonzero polynomials corresponding to the children of $(\widetilde{S},\widetilde{W})$ have a common interlacing.
In particular, if $f_{\widetilde{S},\widetilde{W}}(x;\boldsymbol{A},\boldsymbol{C},\boldsymbol{U},\boldsymbol{R})$ has degree $d$, then the nonzero polynomials corresponding to the children of $(\widetilde{S},\widetilde{W})$ have a common interlacing.
This means that condition (iii) also holds.
Therefore, the polynomial set $\mathcal{P}_k$ forms a generalized interlacing family.

\end{proof}

Now we can prove Theorem \ref{gcur-th2}.

\begin{proof}[Proof of Theorem \ref{gcur-th2}]
Lemma \ref{forms GIF} shows that the polynomial set $\mathcal{P}_k$ in \eqref{pk-set} forms a generalized interlacing family.
By Theorem \ref{generalized interlacing family}, there is a pair $(\widehat{S},\widehat{W})\in \binom{[n_R]}{k}\times \binom{[d_C]}{k}$ such that $p_{\widehat{S},\widehat{W}}(x;\boldsymbol{A},\boldsymbol{C},\boldsymbol{U},\boldsymbol{R})$ has degree $d$ and
\begin{equation*}
\mathrm{minroot}\ p_{\widehat{S},\widehat{W}}(x;\boldsymbol{A},\boldsymbol{C},\boldsymbol{U},\boldsymbol{R})
\geq \mathrm{minroot}\ P_{k}(x;\boldsymbol{A},\boldsymbol{C},\boldsymbol{U},\boldsymbol{R}),
\end{equation*}
which implies 
\begin{equation}\label{meq19}
\mathrm{maxroot}\ p_{\widehat{S},\widehat{W}}(-x;\boldsymbol{A},\boldsymbol{C},\boldsymbol{U},\boldsymbol{R})
\leq \mathrm{maxroot}\ P_{k}(-x;\boldsymbol{A},\boldsymbol{C},\boldsymbol{U},\boldsymbol{R}).
\end{equation}
From Proposition \ref{P-exp-base-gcur1}, we have $\mathrm{rank}(\boldsymbol{U}_{\widehat{S},\widehat{W}})=k$ by (iii), and by (ii), the expression for $p_{\widehat{S},\widehat{W}}(x;\boldsymbol{A},\boldsymbol{C},\boldsymbol{U},\boldsymbol{R})$ is given by:
\begin{equation*}
p_{\widehat{S},\widehat{W}}(x;\boldsymbol{A},\boldsymbol{C},\boldsymbol{U},\boldsymbol{R})	
=\det[\boldsymbol{U}_{\widehat{S},\widehat{W}}]^2\cdot 
\det[x\cdot \boldsymbol{I}_d+
(\boldsymbol{A}- \boldsymbol{C}_{:,\widehat{W}}( \boldsymbol{U}_{\widehat{S},\widehat{W}})^{-1}  \boldsymbol{R}_{\widehat{S},:})^T
(\boldsymbol{A}- \boldsymbol{C}_{:,\widehat{W}}( \boldsymbol{U}_{\widehat{S},\widehat{W}})^{-1}  \boldsymbol{R}_{\widehat{S},:})].
\end{equation*}
Hence, we can rewrite \eqref{meq19} as
\begin{equation*}
\begin{aligned}
&\Vert \boldsymbol{A}- \boldsymbol{C}_{:,\widehat{W}}( \boldsymbol{U}_{\widehat{S},\widehat{W}})^{-1}  \boldsymbol{R}_{\widehat{S},:}\Vert_{2}^2
\leq \mathrm{maxroot}\ P_{k}(-x;\boldsymbol{A},\boldsymbol{C},\boldsymbol{U},\boldsymbol{R}).
\end{aligned}
\end{equation*}
This completes the proof.

\end{proof}

\begin{remark}\label{remark-tree}
Based on Theorem \ref{generalized interlacing family} and the tree structure of a generalized interlacing family, the subsets $\widehat{S}$ and $\widehat{W}$ in Theorem \ref{gcur-th2} can be iteratively constructed.
The detailed selection process is summarized in Algorithm \ref{alg-spr-CUR}. 
Developing a polynomial-time algorithm for computing each $f_{\widetilde{S}\cup\{i\},\widetilde{W}}(x;\boldsymbol{A},\boldsymbol{C},\boldsymbol{U},\boldsymbol{R})$ and $f_{\widetilde{S},\widetilde{W}\cup\{j\}}(x;\boldsymbol{A},\boldsymbol{C},\boldsymbol{U},\boldsymbol{R})$ remains an important direction for future research.
Nevertheless, for the classical CUR matrix decompositions, we develop a polynomial-time algorithm under certain matrix conditions (see Algorithm \ref{alg1} in Section \ref{section-alg} for more details). 
\end{remark}

\begin{algorithm}[!t]
\caption{Deterministic algorithm for the generalized CUR matrix problem.}\label{alg-spr-CUR}
\begin{algorithmic}[1]
\Require $\boldsymbol{A}\in\mathbb{R}^{n\times d}$, $\boldsymbol{C}\in\mathbb{R}^{n\times d_{{C}}}$,  $\boldsymbol{U}\in\mathbb{R}^{n_{{R}}\times d_C}$ and $\boldsymbol{R}\in\mathbb{R}^{n_{{R}}\times d}$;  sampling parameter $k$.
\Ensure A $k$-subset $\widehat{S}\subset[n_R]$ and a $k$-subset $\widehat{W}\subset[d_C]$ satisfying the error bound \eqref{meq71}.
\State Set $\widetilde{S}=\widetilde{W}=\emptyset$.

\For{$l=1,\ldots,k$}

\State Find $i_{l}=\mathop{\mathrm{argmax}}_{i\in \mathcal{I}} \mathrm{minroot}\  f_{\widetilde{S}\cup\{i\},\widetilde{W}}(x;\boldsymbol{A},\boldsymbol{C},\boldsymbol{U},\boldsymbol{R})$, where $\mathcal{I}:=\{i\in [n_R]\backslash \widetilde{S}: \mathrm{deg}(f_{\widetilde{S}\cup\{i\},\widetilde{W}}(x;\boldsymbol{A},\boldsymbol{C},\boldsymbol{U},\boldsymbol{R}))=d\}$. 

\State 
Update $\widetilde{S}\gets \widetilde{S}\cup \{i_{l}\}$.

\State Find $j_{l}=\mathop{\mathrm{argmax}}_{j\in \mathcal{J}} \mathrm{minroot}\  f_{\widetilde{S},\widetilde{W}\cup\{j\}}(x;\boldsymbol{A},\boldsymbol{C},\boldsymbol{U},\boldsymbol{R})$, where $\mathcal{J}:=\{j\in [d_C]\backslash \widetilde{W}: \mathrm{deg}(f_{\widetilde{S},\widetilde{W}\cup\{j\}}(x;\boldsymbol{A},\boldsymbol{C},\boldsymbol{U},\boldsymbol{R}))=d\}$.

\State 
Update $\widetilde{W}\gets \widetilde{W}\cup \{j_{l}\}$.
\EndFor 
\\
\Return $\widehat{S}=\{i_1,i_2,\ldots,i_{k}\}\subset[n_R]$ and $\widehat{W}=\{j_1,j_2,\ldots,j_{k}\}\subset[d_C]$.
\end{algorithmic}
\end{algorithm}

\section{Classical CUR matrix decompositions: Proof of Theorem \ref{th1} and Theorem \ref{th3}}

\subsection{Proof of Theorem \ref{th1}}

The aim of this subsection is to prove Theorem \ref{th1}. 
We start with presenting an upper bound on the largest root of $\mathcal{R}_{x,d}^+\cdot (\partial_x\cdot x\cdot \partial_x)^k\cdot \mathcal{R}_{x,d}^+\ \det[x\cdot \boldsymbol{I}_d-\boldsymbol{A}^T\boldsymbol{A}]$.
The proof of Theorem \ref{th1-maxroot-thm} is postponed to Section \ref{proof of thm51}.

\begin{theorem}\label{th1-maxroot-thm}
Assume that $\boldsymbol{A}\in\mathbb{R}^{n\times d}$ is a rank-$t$ matrix with $\|\boldsymbol{A}\|_2\leq 1$.
Let $\lambda_i$ be the $i$-th largest eigenvalue of $\boldsymbol{A}^T\boldsymbol{A}$.
Let $k$ be an integer satisfying $2\leq k\leq t-1$.
Assume that $\lambda_1>\lambda_k$.
Then
\begin{equation*}
\begin{aligned}
\mathrm{maxroot}\ 	\mathcal{R}_{x,d}^+\cdot (\partial_x\cdot x\cdot \partial_x)^k\cdot \mathcal{R}_{x,d}^+\ \det[x\cdot \boldsymbol{I}_d-\boldsymbol{A}^T\boldsymbol{A}]\,\,\leq \,\,
\frac{(t-k+1)t}{(1-\sqrt{\alpha_k}	)^2} \cdot \lambda_{k+1},
\end{aligned}
\end{equation*}
where
\begin{equation*}
\alpha_{k}=\frac{\frac{1}{\lambda_{k+1}}-\frac{1}{k}\sum_{i=1}^{k}\frac{1}{\lambda_i}}{\frac{1}{\lambda_{k+1}}-1}\in(0,1).
\end{equation*}

\end{theorem}

Now we can give a proof of Theorem \ref{th1}.

\begin{proof}[Proof of Theorem \ref{th1}]
By Theorem \ref{gcur-th2}, there is a $k$-subset $\widehat{S}\subset[n]$ and a $k$-subset $\widehat{W}\subset[d]$ such that $\boldsymbol{A}_{\widehat{S},\widehat{W}}$ is invertible and
\begin{equation*}
\begin{aligned}
\Vert \boldsymbol{A}- \boldsymbol{A}_{:,\widehat{W}}( \boldsymbol{A}_{\widehat{S},\widehat{W}})^{-1}  \boldsymbol{A}_{\widehat{S},:}\Vert_{2}^2
&\leq \mathrm{maxroot}\ P_{k}(-x;\boldsymbol{A},\boldsymbol{A},\boldsymbol{A},\boldsymbol{A})\\
&=\mathrm{maxroot}\ 	\mathcal{R}_{x,d}^+\cdot (\partial_x\cdot x\cdot \partial_x)^k\cdot \mathcal{R}_{x,d}^+\ \det[x\cdot \boldsymbol{I}_d-\boldsymbol{A}^T\boldsymbol{A}],
\end{aligned}
\end{equation*} 
where the last equation follows from Proposition \ref{P-exp-base-gcur-x}.
Combining with Theorem \ref{th1-maxroot-thm}, we  obtain \eqref{xueq11}. This completes the proof.

\end{proof}

\subsubsection{Proof of Theorem \ref{th1-maxroot-thm}}
\label{proof of thm51}

To prove Theorem \ref{th1-maxroot-thm}, we need the following two lemmas. 
The proofs of Lemma \ref{th1-maxroot-lemma} and Lemma \ref{th1-maxroot-lemma2} are presented in Section \ref{proof of th1-maxroot-lemma} and Section \ref{proof of th1-maxroot-lemma2}, respectively. 

\begin{lemma}\label{th1-maxroot-lemma}
Assume that $\boldsymbol{A}\in\mathbb{R}^{n\times d}$ is a rank-$t$ matrix.
Then for any integer $1\leq k<t$ we have
\begin{equation}
\label{maxroot}
\begin{aligned}
&\mathrm{maxroot}\ 	\mathcal{R}_{x,d}^+\cdot (\partial_x\cdot x\cdot \partial_x)^k\cdot \mathcal{R}_{x,d}^+\ \det[x\cdot \boldsymbol{I}_d-\boldsymbol{A}^T\boldsymbol{A}]\\
&\leq (1+(t-k)k)\cdot  	\mathrm{maxroot}\ \mathcal{R}_{x,d}^+\cdot \partial_x^k\cdot \mathcal{R}_{x,d}^+ \ \det[x\cdot \boldsymbol{I}_d-\boldsymbol{A}^T\boldsymbol{A}].
\end{aligned}
\end{equation}

\end{lemma}

\begin{lemma}
\label{th1-maxroot-lemma2}
Assume that $p(x)=x^{d-t}\prod_{i=1}^{t}(x-\lambda_i)$ is a real-rooted polynomial of degree $d$, where $t\leq d$ is a positive integer and $\lambda_1,\ldots,\lambda_t$ are positive  numbers satisfying $1\geq \lambda_1\geq\cdots\geq \lambda_t>0$. 
Let $k$ be an integer satisfying $2\leq k\leq t-1$.
Assume that $\lambda_1>\lambda_k$.
Then
\begin{equation}\label{thm3.1:eq1}
\text{\rm maxroot}\ \mathcal{R}_{x,d}^+\cdot \partial_x^k\cdot \mathcal{R}_{x,d}^+\ p(x)
\leq 
\frac{t}{k}\cdot 
\frac{1}{(1-\sqrt{\alpha_k}	)^2} \cdot \lambda_{k+1},
\end{equation}
where
\begin{equation*}
\alpha_{k}=\frac{\frac{1}{\lambda_{k+1}}-\frac{1}{k}\sum_{i=1}^{k}\frac{1}{\lambda_i}}{\frac{1}{\lambda_{k+1}}-1}\in(0,1).
\end{equation*}

\end{lemma}

Now we can give a proof of Theorem \ref{th1-maxroot-thm}.

\begin{proof}[Proof of Theorem \ref{th1-maxroot-thm}]
Note that $\det[x\cdot \boldsymbol{I}_d-\boldsymbol{A}^T\boldsymbol{A}]=x^{d-t}\prod_{i=1}^{t}(x-\lambda_i)$, and hence we have
\begin{equation*}
\begin{aligned}
&\mathrm{maxroot}\ 	\mathcal{R}_{x,d}^+\cdot (\partial_x\cdot x\cdot \partial_x)^k\cdot \mathcal{R}_{x,d}^+\ \det[x\cdot \boldsymbol{I}_d-\boldsymbol{A}^T\boldsymbol{A}]\\
&\overset{(a)}\leq (1+(t-k)k)\cdot  	\mathrm{maxroot}\ \mathcal{R}_{x,d}^+\cdot \partial_x^k\cdot \mathcal{R}_{x,d}^+ \ \det[x\cdot \boldsymbol{I}_d-\boldsymbol{A}^T\boldsymbol{A}]\\
&\overset{(b)}\leq 
\frac{\frac{t}{k}+(t-k)t}{(1-\sqrt{\alpha_k}	)^2} \cdot \lambda_{k+1}
\leq 
\frac{(t-k+1)t}{(1-\sqrt{\alpha_k}	)^2} \cdot \lambda_{k+1},
\end{aligned}
\end{equation*}
where $(a)$ follows from Lemma \ref{th1-maxroot-lemma} and $(b)$ follows from Lemma \ref{th1-maxroot-lemma2}. This completes the proof.

\end{proof}

\subsubsection{Proof of Lemma \ref{th1-maxroot-lemma}}\label{proof of th1-maxroot-lemma}

To prove Lemma \ref{th1-maxroot-lemma}, we employ the multiplicative convolution of polynomials, which was originally studied by Szeg\"o \cite{Sze22} (see also \cite{inter5}).

\begin{definition}
{\rm \cite[Definition 1.4]{inter5}}
For two univariate polynomials $p(x)=\sum_{i=0}^{d}(-1)^ia_i\cdot x^{d-i}$ and $q(x)=\sum_{i=0}^{d}(-1)^ib_i\cdot x^{d-i}$ of degree $d$, where $a_i, b_i\in {\mathbb R}$, the $d$-th symmetric multiplicative convolution of $p(x)$ and $q(x)$ is defined as
\begin{equation}\label{def-multi}
p(x)\boxtimes_d q(x)=\sum\limits_{i=0}^{d}(-1)^i\frac{a_ib_i}{\binom{d}{i}}\cdot x^{d-i}.
\end{equation}

\end{definition}

\begin{remark}\label{eq:2025:xu91}
Let $p(x)$ and $q(x)$ be two polynomials of degree $d$. It is easy to check that $p(x)\boxtimes_d q(x)=q(x)\boxtimes_d p(x)$ and $\mathcal{R}_{x,d}^+(p)\boxtimes_d \mathcal{R}_{x,d}^+(q)
=(-1)^d\cdot \mathcal{R}_{x,d}^+( p\boxtimes_d q)$.
\end{remark}

\begin{lemma}{\rm \cite[Theorem 1.5]{inter5}}\label{inter5multi}
If $p(x)$ and $q(x)$ are degree-$d$ polynomials with only nonnegative real roots, then $p(x)\boxtimes_d q(x)$ also has only nonnegative real roots. Moreover, we have
\begin{equation*}
\text{\rm maxroot}\ p(x)\boxtimes_d q(x)\leq \text{\rm maxroot}\ p(x)\cdot \text{\rm maxroot}\ q(x).
\end{equation*}	

\end{lemma}

\begin{lemma}\label{lemma3}
Let $r\geq 1$ be an integer.
For any polynomial $q(x)$ of degree $r$, we have
\begin{equation*}
\partial_x^{k} (x^{k}q(x))= q_{}(x)\boxtimes_{r} \partial_x^{k} (x^{k}(x-1)^{r}).
\end{equation*}	
\end{lemma}

\begin{proof}
Write $q(x)=\sum\limits_{j=0}^{r}(-1)^j\cdot b_j\cdot x^{r-j}$.
By Leibniz rule, we have
\begin{equation}\label{align1}
\begin{aligned}
\partial_x^{k} (x^{k}q(x))
&=\sum\limits_{i=0}^{\min\{k,r\}} \binom{k}{i} \cdot \bigg(\partial_x^{k-i}\ x^{k}\bigg)\cdot \bigg(\partial_x^{i}\ q(x)\bigg)\\
&=\sum\limits_{i=0}^{\min\{k,r\}} \binom{k}{i}\cdot \bigg( \frac{k!}{i!}\cdot x^{i}\bigg)\cdot 
\bigg(\sum\limits_{j=0}^{r-i}(-1)^j\cdot b_j\cdot \frac{(r-j)!}{(r-j-i)!}\cdot x^{r-j-i}\bigg)\\
&=\sum\limits_{j=0}^{r}(-1)^j\cdot b_j\cdot a_j\cdot x^{r-j},
\end{aligned}
\end{equation}
where 
\begin{equation*}
a_j:=\sum\limits_{i=0}^{\min\{k,r-j\}}\binom{k}{i}\cdot \frac{k!}{i!}\cdot \frac{(r-j)!}{(r-j-i)!},\quad j=0,1,\ldots,r.
\end{equation*}
Substituting $q(x)=(x-1)^r$ into equation \eqref{align1}, we obtain
\begin{equation}\label{align2}
\partial_x^{k} (x^{k}(x-1)^r)
=\sum\limits_{j=0}^{r}(-1)^j\cdot \binom{r}{j}\cdot a_j\cdot x^{r-j}.
\end{equation}
Then, combining \eqref{align2} with the definition of the symmetric multiplicative convolution in \eqref{def-multi} we arrive at
\begin{equation*}
\begin{aligned}
q(x) \boxtimes_r  \partial_x^{k} (x^{k}(x-1)^r)
=\sum\limits_{j=0}^{r}(-1)^j\cdot \frac{\binom{r}{j}\cdot a_j\cdot b_j}{\binom{r}{j}}\cdot x^{r-j}
=\sum\limits_{j=0}^{r}(-1)^j\cdot a_j\cdot b_j\cdot x^{r-j}
=\partial_x^{k} (x^{k}q(x)),	
\end{aligned}
\end{equation*}
where the last equation follows from \eqref{align1}.

\end{proof}

\begin{lemma}\label{lemma4}
Let $r\geq 1$ be an integer.
Then  we have
\begin{equation*}
\mathrm{minroot}\ \partial_x^{k} (x^{k}(x-1)^{r})
\geq \frac{1}{1+kr}.
\end{equation*}	
\end{lemma}
\begin{proof}
Note that
\begin{equation*}
\begin{aligned}
\partial_x^{k} (x^{k}(x-1)^{r})
&=\partial_x^{k}\ \sum_{i=0}^{r}\binom{r}{i}\cdot (-1)^{r-i}\cdot x^{i+k}=\sum_{i=0}^{r}\binom{r}{i}\cdot (-1)^{r-i}\cdot \frac{(i+k)!}{i!}\cdot  x^{i}.	
\end{aligned}	
\end{equation*}
Let $\lambda_i$ be the $i$-th largest root of $\partial_x^{k} (x^{k}(x-1)^{r})$ for each  $i\in\{1,2,\ldots,r\}$. 
Then we have
\begin{equation*}
r(k+1)\,\,\overset{(a)}=\,\,\sum_{i=1}^{r}\frac{1}{\lambda_i}\,\,\overset{(b)}\geq\,\, \frac{1}{\lambda_{r}}+r-1,	
\end{equation*} 
where ($a$) follows from Vieta's formula, and ($b$) follows from the fact that each $\lambda_i\in[0,1]$.
Rearranging the terms, we arrive at 
\begin{equation*}
\mathrm{minroot}\ \partial_x^{k} (x^{k}(x-1)^{r})=\lambda_{r}\geq \frac{1}{1+kr}.	
\end{equation*}

\end{proof}

\begin{lemma}\label{lemma4-new}
Let $r\geq 1$ be an integer.	
For any degree-$r$ polynomial $q(x)$ which has only positive real roots, we have	
\begin{equation*}
\mathrm{minroot}\  \partial_x^{k} (x^{k}q(x))
\geq \frac{1}{1+kr}\cdot \mathrm{minroot}\ q(x).
\end{equation*}	
\end{lemma}

\begin{proof}
Let $h(x)=\partial_x^{k} (x^{k}(x-1)^{r})$.
By Lemma \ref{lemma3} and Remark \ref{eq:2025:xu91}, we have
\begin{equation*}
(-1)^r\cdot \mathcal{R}_{x,r}^+\cdot \partial_x^{k} (x^{k}q(x))
=\mathcal{R}_{x,r}^+ (q)\boxtimes_{r} \mathcal{R}_{x,r}^+(h).
\end{equation*}	
Note that $q(x)$, $h(x)$, $\mathcal{R}_{x,r}^+ (q)$ and $\mathcal{R}_{x,r}^+ (h)$ are degree-$r$ polynomials with only positive real roots.
Then by Lemma \ref{inter5multi} we have
\begin{equation*}
\text{\rm maxroot}\ \mathcal{R}_{x,r}^+\cdot \partial_x^{k} (x^{k}q(x))
\leq \text{\rm maxroot}\ \mathcal{R}_{x,r}^+ (q)\cdot \text{\rm maxroot}\ \mathcal{R}_{x,r}^+ (h),
\end{equation*}
implying that
\begin{equation*}
\text{\rm minroot}\ \partial_x^{k} (x^{k}q(x))
\geq \text{\rm minroot}\ q(x)\cdot \text{\rm minroot}\ h(x)
\geq \frac{1}{1+kr}\cdot \mathrm{minroot}\ q(x).
\end{equation*}
Here, the last inequality follows from Lemma \ref{lemma4}.
We arrive at our conclusion.

\end{proof}

Now we can give a proof of Lemma \ref{th1-maxroot-lemma}.

\begin{proof}[Proof of Lemma \ref{th1-maxroot-lemma}]
Let $p(x)= \det[x\cdot \boldsymbol{I}_d-\boldsymbol{A}^T\boldsymbol{A}]$.
Note that $\mathcal{R}_{x,d}^+ \ p(x)$ and $(\partial_x\cdot x\cdot \partial_x)^k\cdot \mathcal{R}_{x,d}^+ \ p(x)$ are real-rooted polynomials with only positive roots.
Also note that by \eqref{meq70} we have $(\partial_x\cdot x\cdot \partial_x)^k\cdot \mathcal{R}_{x,d}^+ \ p(x)=\partial_x^k\cdot x^k\cdot \partial_x^k\cdot \mathcal{R}_{x,d}^+ \ p(x)$.
Therefore, we have
\begin{equation}\label{2026-xu8}
\mathrm{maxroot}\ \mathcal{R}_{x,d}^+\cdot (\partial_x\cdot x\cdot \partial_x)^k\cdot \mathcal{R}_{x,d}^+ \ p(x)
=\frac{1}{\mathrm{minroot}\  \partial_x^k\cdot x^k\cdot \partial_x^k\cdot \mathcal{R}_{x,d}^+ \ p(x)}.	
\end{equation}
Note that $\partial_x^k\cdot \mathcal{R}_{x,d}^+ \ p(x)$ is a degree-$(t-k)$ polynomial with only positive real roots.
Hence, by Lemma \ref{lemma4-new} we have
\begin{equation}\label{2026-xu7}
\begin{aligned}
\mathrm{minroot}\  \partial_x^k\cdot x^k\cdot \partial_x^k\cdot \mathcal{R}_{x,d}^+ \ p(x)
&\geq \frac{1}{1+(t-k)k}\cdot \mathrm{minroot}\  \partial_x^k\cdot \mathcal{R}_{x,d}^+ \ p(x)\\
&=\frac{1}{1+(t-k)k}\cdot\frac{1}{\mathrm{maxroot}\ \mathcal{R}_{x,d}^+\cdot  \partial_x^k\cdot \mathcal{R}_{x,d}^+ \ p(x)}.	
\end{aligned}
\end{equation}
Substituting \eqref{2026-xu7} into \eqref{2026-xu8}, we arrive at \eqref{maxroot}.
\end{proof}

\subsubsection{Proof of Lemma \ref{th1-maxroot-lemma2}}\label{proof of th1-maxroot-lemma2}

To prove Lemma \ref{th1-maxroot-lemma2}, we need the following lemma.

\begin{lemma}\label{minroot-new}
Suppose that $f(x)=\prod_{i=1}^{t}(x-\beta_i)$ is a real-rooted degree-$t$ polynomial, where $\beta_1\leq \cdots \leq \beta_t$. 
Let $k$ be a positive integer satisfying $2\leq k\leq t-1$.
Assume that $\beta_{k}>\beta_1$.
Then
\begin{equation}\label{eq:2025:93}
\mathrm{minroot}\ \partial_x^k f(x)
\geq 
(1-c_k)\cdot \beta_1+
c_k\cdot \beta_{k+1},
\end{equation}
where
\begin{equation*}
\alpha_{k}=\frac{\beta_{k+1}-\frac{1}{k}\sum_{i=1}^{k}\beta_i}{\beta_{k+1}-\beta_1}\in(0,1)
\quad\text{and}\quad
c_k=\frac{k}{t}\cdot 
\bigg(\sqrt{1-\frac{k}{t}\cdot \alpha_{k}}-\sqrt{\alpha_{k}-\frac{k}{t}\cdot \alpha_{k}}	\bigg)^2\in[0,1].	
\end{equation*}
In particular, if $\beta_1\geq 0$ then we have
\begin{equation}\label{eq:2025:92}
\mathrm{minroot}\ \partial_x^k f(x)
\geq \frac{k}{t}\cdot 
(1-\sqrt{\alpha_{k}}	)^2 \cdot \beta_{k+1}.
\end{equation}	
\end{lemma}
\begin{proof}
See Appendix \ref{proof-minroot-new}.	
\end{proof}

\begin{remark}
By Rolle's theorem, we see  that $\partial_x f(x)$ inerlaces  $f(x)$. In particular, the smallest root of $\partial_x^k f(x)$ is between $\beta_{1}$ and $\beta_{2}$.
Repeating this argument, we can obtain 
\begin{equation}\label{eqxu10}
\beta_{1}\leq \mathrm{minroot}\ \partial_x^k f(x)\leq \beta_{k+1},\ \forall 1\leq k\leq t-1.
\end{equation}
Lemma \ref{minroot-new} improves the lower bound in \eqref{eqxu10} by showing that the smallest root of $\partial_x^k f(x)$ is no smaller than a convex combination of $\beta_1$ and $\beta_{k+1}$.

\end{remark}

Now we can give a proof of Lemma \ref{th1-maxroot-lemma2}.
\begin{proof}[Proof of Lemma \ref{th1-maxroot-lemma2}]
Set $\beta_i:=\frac{1}{\lambda_i}$ for each $1\leq i\leq t$. Given that $\det[x\cdot \boldsymbol{I}_d-\boldsymbol{A}^T\boldsymbol{A}]=x^{d-t}\prod_{i=1}^{t}(x-\lambda_i)=p(x)$, a direct computation reveals: 
\begin{equation*}
\begin{aligned}
\mathcal{R}_{x,d}^+\cdot \partial_x^k\cdot \mathcal{R}_{x,d}^+ \ \det[x\cdot \boldsymbol{I}_d-\boldsymbol{A}^T\boldsymbol{A}]
=(-1)^t\cdot 
\Big(\prod_{i=1}^{t}\lambda_i\Big)\cdot 
\mathcal{R}_{x,d}^+\cdot \partial_x^k\   \prod_{i=1}^{t}(x-\beta_i).
\end{aligned}	
\end{equation*}
It follows that
\begin{equation}\label{eq:2025:xu89}
\mathrm{maxroot}\ p(x)\,\,=\,\,\mathrm{maxroot}\ \mathcal{R}_{x,d}^+\cdot \partial_x^k\cdot \mathcal{R}_{x,d}^+ \ \det[x\cdot \boldsymbol{I}_d-\boldsymbol{A}^T\boldsymbol{A}]\,\,=\,\,\frac{1}{\mathrm{minroot}\ \partial_x^k\   \prod_{i=1}^{t}(x-\beta_i)}	.
\end{equation}
Note that each $\beta_i\geq 1$ since we assume $\|\boldsymbol{A}\|_2^2=\lambda_1\leq 1$. By Lemma \ref{minroot-new} we obtain that 
\begin{equation}\label{eq:2025:xu90}
\mathrm{minroot}\ \partial_x^k\   \prod_{i=1}^{t}(x-\beta_i)
\geq \frac{k}{t}\cdot 
(1-\sqrt{\alpha_k}	)^2 \cdot \beta_{k+1}
=\frac{k}{t}\cdot 
(1-\sqrt{\alpha_k}	)^2 \cdot \ \frac{1}{\lambda_{k+1}}.
\end{equation}
Substituting \eqref{eq:2025:xu90} into \eqref{eq:2025:xu89}, we arrive at \eqref{thm3.1:eq1}.
\end{proof}

\subsection{Proof of Theorem \ref{th3}}\label{section-alg}
In this subsection we provide a deterministic polynomial-time algorithm for classical CUR matrix problem and prove  Theorem \ref{th3}.
Recall that we use $\lambda_{\max}^{\varepsilon}(p)\in \mathbb{R}$ to denote an $\varepsilon$-approximation to the largest root of a real-rooted polynomial $p(x)$, i.e.,
\begin{equation}\label{eps-approx}
| \lambda_{\max}^{\varepsilon}(p)-\mathrm{maxroot}\ p|\leq \varepsilon	.
\end{equation}
Throughout this subsection, we assume that each square submatrix of $\boldsymbol{A}$ with size at most $k$ is invertible. 
For convenience, for each ${S}\subset[n]$ and ${W}\subset[d]$ with $|S|=|{W}|\leq k$ we denote the degree-$d$ polynomial
\begin{equation*}
h_{S,W}(x):=\det[x\cdot \boldsymbol{I}_d- (\boldsymbol{A}-\boldsymbol{A}_{:,{W}}(\boldsymbol{A}_{{S},{W}})^{-1}\boldsymbol{A}_{{S},:})^T(\boldsymbol{A}-\boldsymbol{A}_{:,{W}}(\boldsymbol{A}_{{S},{W}})^{-1}\boldsymbol{A}_{{S},:})].	
\end{equation*}

\begin{algorithm}[t]
\caption{Deterministic polynomial-time algorithm for the classcial CUR matrix problem}\label{alg1}
\begin{algorithmic}[1]
\Require 
Target matrix $\boldsymbol{A}\in\mathbb{R}^{n\times d}$ that satisfies the assumption of Theorem \ref{th3}; sampling parameter $1\leq k\leq \mathrm{rank}(\boldsymbol{A})$; error parameter $\varepsilon>0$.
\Ensure A $k$-subset $\widehat{S}\subset[n]$ and a $k$-subset $\widehat{W}\subset[d]$ satisfying the error bound \eqref{eq:2025:xu32}.
\State 
Set $\widetilde{S}=\widetilde{W}=\emptyset$. 
Set $\boldsymbol{B}=\boldsymbol{A}$. Compute $\boldsymbol{B}^T\boldsymbol{B}$.

\For{$l=1,2,\ldots,k$}

\For{$i\in[n]$ and $j\in[d]$ such that $\boldsymbol{B}(i,j)\neq 0$}
\State Compute the matrix 
\[
\boldsymbol{M}_{i,j}:=(\boldsymbol{B}-\frac{1}{\boldsymbol{B}(i,j)} \boldsymbol{B}_{:,\{j\}}\boldsymbol{B}_{\{i\},:})^T(\boldsymbol{B}-\frac{1}{\boldsymbol{B}(i,j)} \boldsymbol{B}_{:,\{j\}}\boldsymbol{B}_{\{i\},:}) \in\mathbb{R}^{d\times d}.
\]
\State Compute the characteristic polynomial $h_{\widetilde{S}\cup\{i\},\widetilde{W}\cup\{j\}}(x)= \det[x\cdot \boldsymbol{I}_d-
\boldsymbol{M}_{i,j}]$.
\State Using the standard technique of binary search with a Sturm sequence, compute an $\varepsilon$-approximation 
\begin{equation*}
\lambda_{\max}^{\varepsilon}\Big(  \mathcal{R}_{x,d}^+\cdot (\partial_x\cdot x\cdot \partial_x)^{k-l} \cdot \mathcal{R}_{x,d}^+\ h_{\widetilde{S}\cup\{i\},\widetilde{W}\cup\{j\}}(x)\Big).	
\end{equation*}
\EndFor
\State Find $(i_{l},j_{l})=\mathop{\mathrm{argmin}}
\limits_{(i,j): \boldsymbol{B}(i,j)\neq 0}
\lambda_{\max}^{\varepsilon}\Big( 
\mathcal{R}_{x,d}^+\cdot (\partial_x\cdot x\cdot \partial_x)^{k-l} \cdot \mathcal{R}_{x,d}^+\ h_{\widetilde{S}\cup\{i\},\widetilde{W}\cup\{j\}}(x)
\Big)$.
\label{line 14}
\State Update $\widetilde{S}\gets \widetilde{S}\cup\{i_l\}$, $\widetilde{W}\gets \widetilde{W}\cup\{j_l\}$ and $\boldsymbol{B}\gets \boldsymbol{B}-\frac{1}{\boldsymbol{B}(i_{l},j_{l})} \boldsymbol{B}_{:,\{j_{l}\}}\boldsymbol{B}_{\{i_{l}\},:} \in\mathbb{R}^{n\times d}$.
\EndFor \\
\Return $\widehat{S}=\{i_1,i_2,\ldots,i_k\},\widehat{W}=\{j_1,j_2,\ldots,j_k\}$.
\end{algorithmic}
\end{algorithm}

Here we give a proof of Theorem \ref{th3}.

\begin{proof}[Proof of Theorem \ref{th3}]

We first prove \eqref{eq:2025:xu32}.
Note that the polynomial $f_{\widetilde{S},\widetilde{W}}(x)$ has degree-$d$ for each $\widetilde{S}\subset[n]$ with $|\widetilde{S}|\leq k$ and for each $\widetilde{W}\subset[d]$ with $|\widetilde{W}|\leq k$. 
Moreover, 
we have $\boldsymbol{A}(i,j)\neq 0$ for each $i\in[n]$ and for each $j\in[d]$.
By Lemma \ref{lemma2-f-SW}, for each $i\in[n]$ and for each $j\in[d]$ we have
\begin{equation}\label{meq54}
\mathrm{minroot}\ f_{\{i\},\{j\}}(x;\boldsymbol{A},\boldsymbol{A},\boldsymbol{A},\boldsymbol{A})
=-	\mathrm{maxroot}\ \mathcal{R}_{x,d}^+\cdot (\partial_x\cdot x\cdot \partial_x)^{k-1} \cdot \mathcal{R}_{x,d}^+\   h_{\{i\},\{j\}}(x).	
\end{equation}
Define
\begin{equation*}
\begin{aligned}
(i^*_1,j^*_1)
&=\mathop{\mathrm{argmax}}_{i\in[n],j\in[d] } 
\mathrm{minroot}\ f_{\{ {i}\},\{ {j}\}}(x;\boldsymbol{A},\boldsymbol{A},\boldsymbol{A},\boldsymbol{A})
\\&
=\mathop{\mathrm{argmin}}_{i\in[n],j\in[d]} 
\mathrm{maxroot}\ \mathcal{R}_{x,d}^+\cdot (\partial_x\cdot x\cdot \partial_x)^{k-1}  \cdot \mathcal{R}_{x,d}^+\   h_{\{i\},\{j\}}(x),
\end{aligned}
\end{equation*}
where the last equation follows from \eqref{meq54}.
By Lemma \ref{lemma-f-SW} and Lemma \ref{lemma2.67}, we have
\begin{equation}\label{meq53}
\mathrm{minroot}\ f_{\{ {i}^*_1\},\{ {j}^*_1\}}(x;\boldsymbol{A},\boldsymbol{A},\boldsymbol{A},\boldsymbol{A})\geq
\mathrm{minroot}\  f_{\emptyset,\emptyset}(x;\boldsymbol{A},\boldsymbol{A},\boldsymbol{A},\boldsymbol{A}).
\end{equation}
By our choice of $i_{1}$ and $j_{1}$ in Line 8 of Algorithm \ref{alg1}, we have
\begin{equation}\label{meq55}
\begin{aligned}
&\mathrm{maxroot}\ \mathcal{R}_{x,d}^+\cdot (\partial_x\cdot x\cdot \partial_x)^{k-1}  \cdot \mathcal{R}_{x,d}^+\   h_{\{i_1\},\{j_1\}}(x)\\
&\overset{(a)}  \leq \varepsilon+ \lambda_{\max}^{\varepsilon}\Big( \mathcal{R}_{x,d}^+\cdot (\partial_x\cdot x\cdot \partial_x)^{k-1}  \cdot \mathcal{R}_{x,d}^+\ h_{\{i_1\},\{j_1\}}(x)\Big)\\
& \leq \varepsilon+ \lambda_{\max}^{\varepsilon}\Big( \mathcal{R}_{x,d}^+\cdot (\partial_x\cdot x\cdot \partial_x)^{k-1}  \cdot \mathcal{R}_{x,d}^+\ h_{\{i_1^*\},\{j_1^*\}}(x)\Big)\\
&\overset{(b)}  \leq 2\varepsilon+ \mathrm{maxroot}\ \mathcal{R}_{x,d}^+\cdot (\partial_x\cdot x\cdot \partial_x)^{k-1}  \cdot \mathcal{R}_{x,d}^+\   h_{\{i_1^*\},\{j_1^*\}}(x),\\
\end{aligned}
\end{equation}
where ($a$) and ($b$) follow from \eqref{eps-approx}.
Combining \eqref{meq55} with \eqref{meq53} and \eqref{meq54}, we have
\begin{equation}\label{meq59}
\begin{aligned}
\mathrm{minroot}\ f_{\{i_1\},\{j_1\}}(x;\boldsymbol{A},\boldsymbol{A},\boldsymbol{A},\boldsymbol{A})
&\geq -2\varepsilon+ \mathrm{minroot}\ f_{\{i^*_1\},\{j^*_1\}}(x;\boldsymbol{A},\boldsymbol{A},\boldsymbol{A},\boldsymbol{A})\\
&\geq  	-2\varepsilon+\mathrm{minroot}\  f_{\emptyset,\emptyset}(x;\boldsymbol{A},\boldsymbol{A},\boldsymbol{A},\boldsymbol{A}).
\end{aligned}
\end{equation}
Define 
\begin{equation*}
\boldsymbol{B}=\boldsymbol{A}-\frac{1}{\boldsymbol{A}(i_{1},j_{1})} \boldsymbol{A}_{:,\{j_{1}\}}\boldsymbol{A}_{\{i_{1}\},:} \in\mathbb{R}^{n\times d}.	
\end{equation*}
Recall that each square submatrix of $\boldsymbol{A}$ with size at most $k$ is invertible.
By Lemma \ref{lemma2-f-SW}, for each $i\in[n]\backslash \{i_1\}$ and for each $j\in[d]\backslash \{j_1\}$ we have
\begin{equation}\label{meq58}
\mathrm{minroot}\ f_{\{i_1,i\},\{j_1,j\}}(x;\boldsymbol{A},\boldsymbol{A},\boldsymbol{A},\boldsymbol{A})
=-	\mathrm{maxroot}\ \mathcal{R}_{x,d}^+\cdot (\partial_x\cdot x\cdot \partial_x)^{k-2}  \cdot \mathcal{R}_{x,d}^+\   h_{\{i_1,i\},\{j_1,j\}}(x).	
\end{equation}
By Lemma \ref{Linear algebra-lemma-2}, we have $\boldsymbol{B}(i_{},j_{})\neq 0$ precisely when $i\in[n]\backslash \{i_1\}$ and $j\in[d]\backslash \{j_1\}$.
Moreover, for any $i\in[n]\backslash \{i_1\}$ and $j\in[d]\backslash \{j_1\}$ we have  
\begin{equation}\label{meq60}
\boldsymbol{A}-\boldsymbol{A}_{:,\{j_1,j\}}(\boldsymbol{A}_{\{i_1,i\},\{j_1,j\}})^{-1}\boldsymbol{A}_{\{i_1,i\},:}=
\boldsymbol{B}-\frac{1}{\boldsymbol{B}(i_{},j_{})}\boldsymbol{B}_{:,\{j\}} \boldsymbol{B}_{\{i\},:}.	
\end{equation}
Hence, we can rewrite each $h_{\{i,i_1\},\{j,j_1\}}(x)$ as
\begin{equation*}
h_{\{i_1,i\},\{j_1,j\}}(x)=\det[x\cdot \boldsymbol{I}_d-(\boldsymbol{B}-\frac{1}{\boldsymbol{B}(i_{},j_{})}\boldsymbol{B}_{:,\{j\}} \boldsymbol{B}_{\{i\},:})^T(\boldsymbol{B}-\frac{1}{\boldsymbol{B}(i_{},j_{})}\boldsymbol{B}_{:,\{j\}} \boldsymbol{B}_{\{i\},:})].	
\end{equation*} 
Define
\begin{equation*}
\begin{aligned}
(i^*_2,j^*_2)
&=\mathop{\mathrm{argmax}}_{(i,j):\boldsymbol{B}(i,j)\neq 0 } 
\mathrm{minroot}\ f_{\{ {i_1,i}\},\{ {j_1,j}\}}(x;\boldsymbol{A},\boldsymbol{A},\boldsymbol{A},\boldsymbol{A})\\
&=\mathop{\mathrm{argmin}}_{(i,j):\boldsymbol{B}(i,j)\neq 0 } 
\mathrm{maxroot}\ \mathcal{R}_{x,d}^+\cdot (\partial_x\cdot x\cdot \partial_x)^{k-2}  \cdot \mathcal{R}_{x,d}^+\ h_{\{i_1,i\},\{j_1,j\}}(x).
\end{aligned}
\end{equation*}
By Lemma \ref{lemma-f-SW} and Lemma \ref{lemma2.67}, we have
\begin{equation}\label{meq56}
\mathrm{minroot}\ f_{\{i_1, {i}^*_2\},\{j_1, {j}^*_2\}}(x;\boldsymbol{A},\boldsymbol{A},\boldsymbol{A},\boldsymbol{A})\geq
\mathrm{minroot}\  f_{\{i_1 \},\{j_1 \}}(x;\boldsymbol{A},\boldsymbol{A},\boldsymbol{A},\boldsymbol{A}).
\end{equation}
Using a similar calculation as in \eqref{meq55}, and based on our choice of $i_{2}$ and $j_{2}$ in Line 8 of Algorithm~\ref{alg1}, we have
\begin{equation}\label{meq57}
\begin{aligned}
&\mathrm{maxroot}\ \mathcal{R}_{x,d}^+\cdot (\partial_x\cdot x\cdot \partial_x)^{k-2}  \cdot \mathcal{R}_{x,d}^+\   h_{\{i_1,i_2\},\{j_1,j_2\}}(x)\\
&\leq 2\varepsilon+ \mathrm{maxroot}\ \mathcal{R}_{x,d}^+\cdot (\partial_x\cdot x\cdot \partial_x)^{k-2}  \cdot \mathcal{R}_{x,d}^+\    h_{\{i_1,i_2^*\},\{j_1,j_2^*\}}(x).\\
\end{aligned}
\end{equation}
Hence, we have
\begin{equation*}
\begin{aligned}
\mathrm{minroot}\ f_{\{i_1,i_2\},\{j_1,j_2\}}(x;\boldsymbol{A},\boldsymbol{A},\boldsymbol{A},\boldsymbol{A})
&\overset{(a)}\geq -2\varepsilon+ \mathrm{minroot}\ f_{\{i_1,i^*_2\},\{j_1,j^*_2\}}(x;\boldsymbol{A},\boldsymbol{A},\boldsymbol{A},\boldsymbol{A})\\
&\overset{(b)}\geq  	-2\varepsilon+\mathrm{minroot}\  f_{\{i_1\},\{j_1\}}(x;\boldsymbol{A},\boldsymbol{A},\boldsymbol{A},\boldsymbol{A})\\
&\overset{(c)}\geq -4\varepsilon+\mathrm{minroot}\  f_{\emptyset,\emptyset}(x;\boldsymbol{A},\boldsymbol{A},\boldsymbol{A},\boldsymbol{A}).	
\end{aligned}
\end{equation*}
Here, ($a$) follows from \eqref{meq58} and \eqref{meq57}, ($b$) follows from \eqref{meq56},  and ($c$) follows from \eqref{meq59}.
Then, we update  
\begin{equation*}
\boldsymbol{B}\gets \boldsymbol{B}-\frac{1}{\boldsymbol{B}(i_{2},j_{2})} \boldsymbol{B}_{:,\{j_{2}\}}\boldsymbol{B}_{\{i_{2}\},:}.
\end{equation*}
By \eqref{meq60}, we can rewrite $\boldsymbol{B}$ as $\boldsymbol{B}=
\boldsymbol{A}-\boldsymbol{A}_{:,\{j_1,j_2\}}(\boldsymbol{A}_{\{i_1,i_2\},\{j_1,j_2\}})^{-1}\boldsymbol{A}_{\{i_1,i_2\},:}$.  
Moreover, by Lemma \ref{Linear algebra-lemma-2}, we have $\boldsymbol{B}(i_{},j_{})\neq 0$ precisely when $i\in[n]\backslash \{i_1,i_2\}$ and $j\in[d]\backslash \{j_1,j_2\}$.
Since each square submatrix of $\boldsymbol{A}$ with size at most $k$ is invertible, we can repeat the above argument $k-2$ times.
Consequently, we obtain 
a subset $\widehat{S}=\{i_1,\ldots,i_k\}\subset[n]$ and a subset $\widehat{W}=\{j_1,\ldots,j_k\}\subset[d]$ such that
\begin{equation*}
\mathrm{minroot}\ f_{\widehat{S},\widehat{W}}(x;\boldsymbol{A},\boldsymbol{A},\boldsymbol{A},\boldsymbol{A})	
\geq -2k\varepsilon+ \mathrm{minroot}\ f_{\emptyset,\emptyset}(x;\boldsymbol{A},\boldsymbol{A},\boldsymbol{A},\boldsymbol{A}).
\end{equation*}
Note that
\begin{equation*}
\mathrm{minroot}\  f_{\widehat{S},\widehat{W}}(x;\boldsymbol{A},\boldsymbol{A},\boldsymbol{A},\boldsymbol{A})=-\Vert \boldsymbol{A}- \boldsymbol{A}_{:,\widehat{W}}( \boldsymbol{A}_{\widehat{S},\widehat{W}})^{-1}  \boldsymbol{A}_{\widehat{S},:}\Vert_{2}^2
\end{equation*}
and
\begin{equation*}
\mathrm{minroot}\ f_{\emptyset,\emptyset}(x;\boldsymbol{A},\boldsymbol{A},\boldsymbol{A},\boldsymbol{A})=-\mathrm{maxroot}\ P_k(-x;\boldsymbol{A},\boldsymbol{A},\boldsymbol{A},\boldsymbol{A}).	
\end{equation*}
Hence, we arrive at \eqref{eq:2025:xu32}.

Next, we will analyze the time complexity of Algorithm \ref{alg1}.
We need to compute the matrix $\boldsymbol{B}^{T}\boldsymbol{B}$ in Line 1 of Algorithm \ref{alg1} with a time complexity of $O(nd^2)$. 
We claim that the time complexity for computing each $\lambda_{\max}^{\varepsilon}(\mathcal{R}_{x,d}^+\cdot (\partial_x\cdot x\cdot \partial_x)^{k-l} \cdot \mathcal{R}_{x,d}^+\ h_{\widetilde{S}\cup\{i\},\widetilde{W}\cup\{j\}}(x))$ at each iteration is $O( dn+d^w\log(d\vee\frac{1}{\varepsilon}))$. Thus, the total running time of Algorithm \ref{alg1} is
\[
O(nd^2)+k\cdot dn\cdot O( dn+d^w\log(d\vee\frac{1}{\varepsilon}) )=O( kn^2d^2+knd^{w+1}\log(d\vee\frac{1}{\varepsilon})).
\]

It remains to prove that the time complexity for each $\lambda_{\max}^{\varepsilon}(\mathcal{R}_{x,d}^+\cdot (\partial_x\cdot x\cdot \partial_x)^{k-l} \cdot \mathcal{R}_{x,d}^+\ h_{\widetilde{S}\cup\{i\},\widetilde{W}\cup\{j\}}(x))$ is $O(dn+d^w\log(d\vee\frac{1}{\varepsilon}) )$.
At Line 4, we first compute 
$\boldsymbol{B}^T\boldsymbol{B}_{:,\{j\}}\in\mathbb{R}^{d}$, 
$\boldsymbol{B}^T\boldsymbol{B}_{:,\{j\}}\cdot \boldsymbol{B}_{\{i\},:}\in\mathbb{R}^{d\times d}$, 
$\|\boldsymbol{B}_{:,\{j\}}\|_2^2$, 
and $\boldsymbol{B}_{\{i\},:}^T\boldsymbol{B}_{\{i\},:}\in\mathbb{R}^{d\times d}$ in time $O(dn)$, $O(d^2)$, $O(n)$ and $O(d^2)$, respectively.
Since $\boldsymbol{B}^T\boldsymbol{B}$ is known through the last iteration, we can compute $\boldsymbol{M}_{i,j}$ through 
\begin{equation*}
\begin{aligned}
\boldsymbol{M}_{i,j}
&=\boldsymbol{B}^T\boldsymbol{B}-	\frac{\boldsymbol{B}^T\boldsymbol{B}_{:,\{j\}}\cdot \boldsymbol{B}_{\{i\},:}+(\boldsymbol{B}^T\boldsymbol{B}_{:,\{j\}}\cdot \boldsymbol{B}_{\{i\},:})^T}{\boldsymbol{B}(i,j)}+\frac{\|\boldsymbol{B}_{:,\{j\}}\|_2^2}{\boldsymbol{B}(i,j)^2}\cdot \boldsymbol{B}_{\{i\},:}^T\boldsymbol{B}_{\{i\},:}
\end{aligned}	
\end{equation*}
in time $O(d^2)$.
Then, it takes time $O(d^w\log d)$ to compute the degree-$d$ characteristic polynomial $h_{\widetilde{S}\cup\{i\},\widetilde{W}\cup\{j\}}(x)$ in Line 5 of Algorithm \ref{alg1} \cite{KG85}.
By employing  Lemma \ref{Linear algebra-lemma-6}, the time complexity for calculating the  the coefficients of the polynomial $\mathcal{R}_{x,d}^+\cdot (\partial_x\cdot x\cdot \partial_x)^{k-l} \cdot \mathcal{R}_{x,d}^+\ h_{\widetilde{S}\cup\{i\},\widetilde{W}\cup\{j\}}(x)$ is $O(1)$.
Note that $\mathcal{R}_{x,d}^+\cdot (\partial_x\cdot x\cdot \partial_x)^{k-l} \cdot \mathcal{R}_{x,d}^+\ h_{\widetilde{S}\cup\{i\},\widetilde{W}\cup\{j\}}(x)$ has at most $d$ nonzero roots, so the time complexity to  obtain an $\varepsilon$-approximation to its largest root is $O(d^2\log{\frac{1}{\varepsilon}} )$   (see \cite[Section 4.1]{inter3}).
Hence, the time complexity to obtain each $\lambda_{\max}^{\varepsilon}(\mathcal{R}_{x,d}^+\cdot (\partial_x\cdot x\cdot \partial_x)^{k-l} \cdot \mathcal{R}_{x,d}^+\ h_{\widetilde{S}\cup\{i\},\widetilde{W}\cup\{j\}}(x))$ is
\begin{equation*}
\begin{aligned}
&O(dn)+O(d^2)+O(n)+O(d^2)+O(d^2)+O(d^w\log d)+O(1)+O(d^2\log {\frac{1}{\varepsilon}})\\
&=O(dn+d^w\log(d\vee\frac{1}{\varepsilon})).	
\end{aligned}
\end{equation*}
This completes the proof.

\end{proof}

\section{The row subset selection problem: Proof of Theorem \ref{th1-newcase}}

In this section, we give a proof of Theorem \ref{th1-newcase} based on Theorem \ref{gcur-th2}, Proposition \ref{P-exp-base-gcur-x-newcase} and Lemma \ref{lemma4-new}.

\begin{proof}[Proof of Theorem \ref{th1-newcase}]
Let $p(x)=\det[x\cdot \boldsymbol{I}_d-\boldsymbol{A}^T(\boldsymbol{I}_n-\boldsymbol{C}\boldsymbol{C}^{\dagger})\boldsymbol{A}]$.
By Theorem \ref{gcur-th2}, there is a $k$-subset $\widehat{S}\subset[n]$ such that $\boldsymbol{C}_{\widehat{S},:}\in\mathbb{R}^{k\times k}$ is invertible and
\begin{equation}\label{2026-xu10}
\begin{aligned}
\Vert \boldsymbol{A}- \boldsymbol{C}_{}( \boldsymbol{C}_{\widehat{S},:})^{-1}  \boldsymbol{A}_{\widehat{S},:}\Vert_{2}^2
&\leq \mathrm{maxroot}\ P_{k}(-x;\boldsymbol{A},\boldsymbol{C},\boldsymbol{C},\boldsymbol{A})\\
&\overset{(a)}=\mathrm{maxroot}\ 	\mathcal{R}_{x,d}^+\cdot \partial_x^k\cdot x^k \cdot \mathcal{R}_{x,d}^+\ p(x)\\
&=\frac{1}{\mathrm{minroot}\ 	\partial_x^k\cdot x^k \cdot \mathcal{R}_{x,d}^+\ p(x)},
\end{aligned}
\end{equation} 
where ($a$) follows from Proposition \ref{P-exp-base-gcur-x-newcase}.
Note that $r=\mathrm{rank}(\boldsymbol{A}-\boldsymbol{C}\boldsymbol{C}^{\dagger}\boldsymbol{A})=\mathrm{deg}(\mathcal{R}_{x,d}^+\ p(x))$.
By Lemma \ref{lemma4-new} we have
\begin{equation}\label{2026-xu9}
\begin{aligned}
\mathrm{minroot}\  \partial_x^k\cdot x^k \cdot \mathcal{R}_{x,d}^+\ p(x)
&\geq \frac{1}{1+kr}\cdot \mathrm{minroot}\   \mathcal{R}_{x,d}^+ \ p(x)\\
&=\frac{1}{1+ kr}\cdot \frac{1}{\mathrm{maxroot}\  p(x)}\\
&=\frac{1}{1+ kr}\cdot \frac{1}{\|\boldsymbol{A}-\boldsymbol{C}\boldsymbol{C}^{\dagger}\boldsymbol{A}\|_2^2}.
\end{aligned}
\end{equation}
Substituting \eqref{2026-xu9} into \eqref{2026-xu10}, we arrive at our conclusion.
	
\end{proof}

\section{Acknowledgments}
{Jian-Feng Cai is supported by HKRGC GRF grants 16307325, 16306124, and 16307023.
Zhiqiang Xu is supported by the National Science Fund for Distinguished Young Scholars (12025108) and NSFC (12471361, 12021001, 12288201)}.
Zili Xu is supported by NSFC grant (12501121, 12571105) and the Natural Science Foundation of Shanghai grant  (25ZR1402131).
}

\Addresses

\newpage

\appendix
\section{Appendix}

\subsection{Proof of Lemma \ref{Linear algebra-lemma-2}}\label{proof-Linear algebra-lemma-2}

\begin{proof}[Proof of Lemma \ref{Linear algebra-lemma-2}]
(i) Using the Schur determinantal formula Lemma \ref{Linear algebra-lemma-schur}, we have
\begin{equation}\label{eq:2025:21}
\det\left[
\begin{matrix}
\boldsymbol{A}(i,j)& \boldsymbol{A}_{\{i\},W}\\
\boldsymbol{A}_{S,\{j\}}& \boldsymbol{A}_{S,W}	
\end{matrix}\right]
=\det[\boldsymbol{A}_{S,W}]\cdot (\boldsymbol{A}(i,j)-\boldsymbol{A}_{\{i\},W}(\boldsymbol{A}_{S,W})^{-1}\boldsymbol{A}_{S,\{j\}})= \det[\boldsymbol{A}_{S,W}]\cdot \boldsymbol{B}(i,j).	
\end{equation}	
For the case where $i\in S$ or $j\in  W$,  a simple observation is that the left hand side of \eqref{eq:2025:21} is zero. 
Hence, combining with   $\det[\boldsymbol{A}_{S,W}]\neq 0$, we have $\boldsymbol{B}(i,j)=0$  for $i\in S$ or $j\in  W$.
If  $i\in [n]\backslash S$ and $j\in [d]\backslash W$, then by \eqref{eq:2025:21} we arrive at
\begin{equation*}
\det[\boldsymbol{A}_{S\cup\{i\},W\cup\{j\}}]
=
\det\left[
\begin{matrix}
\boldsymbol{A}(i,j)& \boldsymbol{A}_{\{i\},W}\\
\boldsymbol{A}_{S,\{j\}}& \boldsymbol{A}_{S,W}	
\end{matrix}\right]=\det[\boldsymbol{A}_{S,W}]\cdot \boldsymbol{B}(i,j).
\end{equation*}

(ii) When $\boldsymbol{B}(i,j)\neq 0$, by (i) we have $\boldsymbol{A}_{S\cup\{i\},W\cup\{j\}}$ is invertible, i.e. $\det[\boldsymbol{A}_{S\cup\{i\},W\cup\{j\}}]\neq 0$, 
For simplicity, denote $|S|=|W|=k$.
Note that
\begin{equation*}
\boldsymbol{A}_{S\cup\{i\},W\cup\{j\}}
=
\begin{pmatrix}
\boldsymbol{A}_{S,W}& \boldsymbol{A}_{S,\{j\}}\\
\boldsymbol{A}_{\{i\},W}& \boldsymbol{A}(i,j)	
\end{pmatrix}=
\begin{pmatrix}
\boldsymbol{A}_{S,W}& \boldsymbol{0}  \\
\boldsymbol{A}_{\{i\},W} & 1	
\end{pmatrix}\cdot 
\begin{pmatrix}
\boldsymbol{I}_k& (\boldsymbol{A}_{S,W})^{-1}\boldsymbol{A}_{S,\{j\}}\\
\boldsymbol{0}  & \boldsymbol{B}(i,j)
\end{pmatrix}.	
\end{equation*}
Hence,
\begin{equation*}
\begin{aligned}
(\boldsymbol{A}_{S\cup\{i\},W\cup\{j\}})^{-1}
&=\begin{pmatrix}
\boldsymbol{I}_k& (\boldsymbol{A}_{S,W})^{-1}\boldsymbol{A}_{S,\{j\}}\\
\boldsymbol{0}  &\boldsymbol{B}(i,j)
\end{pmatrix}^{-1}
\begin{pmatrix}
\boldsymbol{A}_{S,W}& \boldsymbol{0}  \\
\boldsymbol{A}_{\{i\},W} & 1	
\end{pmatrix}^{-1}\\
&=\begin{pmatrix}
\boldsymbol{I}_k& -\frac{1}{\boldsymbol{B}(i,j)}\cdot (\boldsymbol{A}_{S,W})^{-1}\boldsymbol{A}_{S,\{j\}}\\
\boldsymbol{0}  &  \frac{1}{\boldsymbol{B}(i,j)}
\end{pmatrix}\begin{pmatrix}
(\boldsymbol{A}_{S,W})^{-1}& \boldsymbol{0}  \\
-\boldsymbol{A}_{\{i\},W}(\boldsymbol{A}_{S,W})^{-1}& 1	
\end{pmatrix}\\
&=	\begin{pmatrix}
(\boldsymbol{A}_{S,W})^{-1}+\frac{1}{\boldsymbol{B}(i,j)}\cdot (\boldsymbol{A}_{S,W})^{-1}\boldsymbol{A}_{S,\{j\}}\boldsymbol{A}_{\{i\},W}(\boldsymbol{A}_{S,W})^{-1}& -\frac{1}{\boldsymbol{B}(i,j)}\cdot (\boldsymbol{A}_{S,W})^{-1}\boldsymbol{A}_{S,\{j\}} \\
-\frac{1}{\boldsymbol{B}(i,j)}\boldsymbol{A}_{\{i\},W}(\boldsymbol{A}_{S,W})^{-1} & \frac{1}{\boldsymbol{B}(i,j)}	
\end{pmatrix}.
\end{aligned}
\end{equation*}	
Substituting the above expression into the left hand side of \eqref{meq61} yields the right hand side, thereby completing the proof.
\end{proof}

\subsection{Proof of Proposition \ref{P-exp-base-gcur1}}\label{proof-P-exp-base-gcur1}

We first prove the following lemmas.

\begin{lemma}\label{Linear algebra-lemma-0}
Assume that $\boldsymbol{A}, \boldsymbol{B}\in\mathbb{R}^{n\times d}$. Then we have
\begin{equation*}
\det\left[\begin{matrix}
x\cdot  \boldsymbol{I}_{d} & \boldsymbol{B}^T  \\
\boldsymbol{A}   & y\cdot \boldsymbol{I}_{n} \\
\end{matrix}\right]=y^{n-d}\cdot \det[xy\cdot \boldsymbol{I}_{d}-\boldsymbol{B}^{ T}\boldsymbol{A}].
\end{equation*}
\end{lemma}
\begin{proof}
Follows from the Schur determinantal formula in Lemma \ref{Linear algebra-lemma-schur}. 	
\end{proof}

\begin{lemma}\label{meq15-lemma1}
Let $\boldsymbol{M},\boldsymbol{N}\in\mathbb{R}^{d\times d}$ be symmetric matrices, where $\boldsymbol{N}$ has rank $t$.
Then
\begin{equation*}
f(x):=y^{t}\cdot \det[x\cdot \boldsymbol{I}_d+
\boldsymbol{M}-\frac{1}{y}\cdot \boldsymbol{N}]\ \bigg|_{y=0}
\end{equation*} 
is a degree-$(d-t)$ polynomial with the leading coefficient $(-1)^t \cdot \prod_{i=1}^t\lambda_i$.
Here, $\lambda_i$ is the $i$-th largest nonzero eigenvalue of $\boldsymbol{N}$.	
\end{lemma}

\begin{proof}
Let $\boldsymbol{N}=\boldsymbol{P}\boldsymbol{\Lambda}\boldsymbol{P}^T$ be the eigenvalue decomposition of $\boldsymbol{N}$,
where $\boldsymbol{P}\in\mathbb{R}^{d\times d}$ satisfies $\boldsymbol{P}^T\boldsymbol{P}=\boldsymbol{I}_{d}$, and $\boldsymbol{\Lambda}=\mathrm{diag}(\lambda_1,\ldots,\lambda_{t},0,\ldots,0)\in\mathbb{R}^{n\times n}$.
Denote  $\boldsymbol{H}=\boldsymbol{P}^T\boldsymbol{M}\boldsymbol{P}\in\mathbb{R}^{d\times d}$ and $h(x,y):=\det[x\cdot \boldsymbol{I}_d+
\boldsymbol{M}-y\cdot \boldsymbol{N}]$.
By \eqref{meq11} we have
\begin{equation*}
\begin{aligned}
h(x,y)
&=\det[x\cdot \boldsymbol{I}_d+
\boldsymbol{P}\boldsymbol{H}\boldsymbol{P}^T-y\cdot \boldsymbol{P}\boldsymbol{\Lambda}\boldsymbol{P}^T]
=\det[x\cdot \boldsymbol{I}_d+
\boldsymbol{H}-y\cdot \boldsymbol{\Lambda}]\\
&=\det[\mathrm{diag}(x-y\cdot \lambda_1,\ldots,x-y\cdot \lambda_t,x,\ldots,x)+\boldsymbol{H}]\\
&
=\sum_{S\subset[t]}\sum_{R\subset\{t+1,\ldots,d\}} 
\Big(\prod_{i\in S}(x-y\cdot \lambda_i)\Big)\cdot 
x^{|R|}\cdot 
\det[\boldsymbol{H}_{(S\cup R)^C,(S\cup R)^C}].\\
\end{aligned}	
\end{equation*}
Hence, we have
\begin{equation*}
\begin{aligned}
f(x)
&=y^t\cdot  h(x,1/y)\ \big|_{{y=0}}
=(-1)^t \cdot \Big(\prod_{i=1}^t\lambda_i\Big)\cdot
\sum_{R\subset\{t+1,\ldots,d\}} 
x^{|R|}\cdot 
\det[\boldsymbol{H}_{([t]\cup R)^C,([t]\cup R)^C}].\\
\end{aligned}
\end{equation*}
Note that the leading term in $f(x)$ is $x^{d-t}\cdot (-1)^t \cdot \prod_{i=1}^t\lambda_i$. 
Hence, we arrive at our conclusion.
\end{proof}

\begin{lemma}\label{meq15-lemma}
Let $\boldsymbol{M},\boldsymbol{N}\in\mathbb{R}^{n\times d}$, where $\boldsymbol{N}$ has rank $t$.
Then
\begin{equation*}
f(x):=y^{2t}\cdot \det[x\cdot \boldsymbol{I}_d+
(\boldsymbol{M}-\frac{1}{y}\cdot \boldsymbol{N})^T
(\boldsymbol{M}-\frac{1}{y}\cdot \boldsymbol{N})]\ \bigg|_{y=0}
\end{equation*} 
is a degree-$(d-t)$ polynomial with positive leading coefficient $\prod_{i=1}^t\lambda_i^2$.	
Here, $\lambda_i$ is the $i$-th largest nonzero singular value of $\boldsymbol{N}$.
\end{lemma}

\begin{proof}
By Lemma \ref{Linear algebra-lemma-0}, we have
\begin{equation}\label{meq14}
(-1)^d\cdot x^{n-d}\cdot f(-x^2)= y^{2t}\cdot \det[x\cdot \boldsymbol{I}_{d+n}+
\widehat{\boldsymbol{M}}-\frac{1}{y}\cdot \widehat{\boldsymbol{N}}]\ \bigg|_{y=0},
\end{equation}
where
\begin{equation*}
\widehat{\boldsymbol{M}}=	\begin{pmatrix}
\boldsymbol{0}& \boldsymbol{M}^T\\
\boldsymbol{M} 	 &  \boldsymbol{0}  
\end{pmatrix}
\quad\text{and}\quad
\widehat{\boldsymbol{N}}=	\begin{pmatrix}
\boldsymbol{0}& \boldsymbol{N}^T\\
\boldsymbol{N} 	 & \boldsymbol{0}   
\end{pmatrix}.
\end{equation*}
Note that $\mathrm{rank}(\widehat{\boldsymbol{N}})=2\cdot \mathrm{rank}(\boldsymbol{N})=2t$. 
By Lemma \ref{Linear algebra-lemma-0} we see that the nonzero eigenvalues of $\widehat{\boldsymbol{N}}$ are $\pm\lambda_1,\ldots,\pm\lambda_t$.
It follows from Lemma \ref{meq15-lemma1} that the right hand side of \eqref{meq14} is a degree-$(d+n-2t)$ polynomial with leading coefficient $(-1)^t\cdot \prod_{i=1}^t\lambda_i^2$.
Hence, $f(x)$ is a degree-$(d-t)$ polynomial with positive leading coefficient $\prod_{i=1}^t\lambda_i^2$. This completes the proof.
\end{proof}

Now we can present a proof of Proposition \ref{P-exp-base-gcur1}.

\begin{proof}[Proof of Proposition \ref{P-exp-base-gcur1}]
(i) Follows from equation \eqref{meq18}.

(ii) Note that the sign of the determinant changes when interchanging two rows of a square matrix, so we have
\begin{equation}
\label{eq:2025:xu41}
(-1)^{d}\cdot p_{S,W}(x;\boldsymbol{A},\boldsymbol{C},\boldsymbol{U},\boldsymbol{R})
= 
\det\left[\begin{matrix}
-x\cdot \boldsymbol{I}_d  & \boldsymbol{A}^T & \boldsymbol{0}_{d\times k} & (\boldsymbol{R}_{S,:})^T  \\
\boldsymbol{A}&   \boldsymbol{I}_n & \boldsymbol{C}_{:,W}  & \boldsymbol{0}_{n\times k}  \\
\boldsymbol{R}_{S,:} &  \boldsymbol{0}_{k\times n} & \boldsymbol{U}_{S,W} &  \boldsymbol{0}_{k\times k} \\
\boldsymbol{0}_{k\times d}  & (\boldsymbol{C}_{:,W})^T & \boldsymbol{0}_{k\times k} &   (\boldsymbol{U}_{S,W})^T
\end{matrix}\right].
\end{equation}
Since $\boldsymbol{U}_{S,W}$ is invertible, using the Schur determinantal formula in Lemma \ref{Linear algebra-lemma-schur} we further obtain
\begin{equation*}
\begin{aligned}
&(-1)^{d}\cdot p_{S,W}(x;\boldsymbol{A},\boldsymbol{C},\boldsymbol{U},\boldsymbol{R})\\
&=
\det[\boldsymbol{U}_{S,W}]^2\cdot
\det\left[
\begin{pmatrix}
-x\cdot \boldsymbol{I}_d & \boldsymbol{A}^T\\
\boldsymbol{A}	 &   \boldsymbol{I}_n
\end{pmatrix}
-
\begin{pmatrix}
\boldsymbol{0}_{d\times k} &  (\boldsymbol{R}_{S,:})^T\\
\boldsymbol{C}_{:,W} & \boldsymbol{0}_{n\times k}
\end{pmatrix}
\begin{pmatrix}
\boldsymbol{U}_{S,W} & \boldsymbol{0}_{k\times k}  \\
\boldsymbol{0}_{k\times k} & ( \boldsymbol{U}_{S,W})^{T} 
\end{pmatrix}^{-1}
\begin{pmatrix}
\boldsymbol{R}_{S,:} & \boldsymbol{0}_{k\times n} \\
\boldsymbol{0}_{k\times d} & (\boldsymbol{C}_{:,W})^T 
\end{pmatrix}
\right]\\
&=\det[\boldsymbol{U}_{S,W}]^2\cdot 
\det
\begin{bmatrix}
-x\cdot \boldsymbol{I}_d & (\boldsymbol{A}-\boldsymbol{C}_{:,W}( \boldsymbol{U}_{S,W})^{-1}  \boldsymbol{R}_{S,:}	)^T\\
\boldsymbol{A}-\boldsymbol{C}_{:,W}( \boldsymbol{U}_{S,W})^{-1}  \boldsymbol{R}_{S,:}	 &   \boldsymbol{I}_n
\end{bmatrix}\\
&=(-1)^d\cdot \det[\boldsymbol{U}_{S,W}]^2\cdot 
\det[x\cdot \boldsymbol{I}_d+
(\boldsymbol{A}- \boldsymbol{C}_{:,W}( \boldsymbol{U}_{S,W})^{-1}  \boldsymbol{R}_{S,:})^T
(\boldsymbol{A}- \boldsymbol{C}_{:,W}( \boldsymbol{U}_{S,W})^{-1}  \boldsymbol{R}_{S,:})],
\end{aligned}
\end{equation*}
where the last equation follows from Lemma \ref{Linear algebra-lemma-0}.
Hence, we arrive at the desired conclusion.

(iii) If $\mathrm{rank}
\left(
\begin{matrix}
\boldsymbol{U}_{S,W} &	\boldsymbol{R}_{S,:}
\end{matrix}\right)<k$, 
then the rows of $\left(
\begin{matrix}
\boldsymbol{0} & \boldsymbol{0} &\boldsymbol{U}_{S,W} &	\boldsymbol{R}_{S,:}
\end{matrix}\right)$ are linear dependent. 
Hence, $p_{S,W}(x;\boldsymbol{A},\boldsymbol{C},\boldsymbol{U},\boldsymbol{R})$ is identically zero.
Similarly, if $\mathrm{rank}
\left(
\begin{matrix}
\boldsymbol{C}_{:,W}\\
\boldsymbol{U}_{S,W}	
\end{matrix}\right)<k$, then $p_{S,W}(x;\boldsymbol{A},\boldsymbol{C},\boldsymbol{U},\boldsymbol{R})$ is also identically zero.

We now turn to the case where $\mathrm{rank}
\left(
\begin{matrix}
\boldsymbol{C}_{:,W}\\
\boldsymbol{U}_{S,W}	
\end{matrix}
\right)
=\mathrm{rank}
\left(
\begin{matrix}
\boldsymbol{U}_{S,W} &	\boldsymbol{R}_{S,:}
\end{matrix}
\right)=k$.
Denote $t=\mathrm{rank}(\boldsymbol{U}_{S,W})$.
If $t=k$, then the conclusion follows from (ii).
We next consider the case when $0\leq t\leq k-1$.
Let $\boldsymbol{U}_{S,W}=\boldsymbol{P}\boldsymbol{\Lambda}\boldsymbol{Q}^T$ be the full SVD of $\boldsymbol{U}_{S,W}$, 
where $\boldsymbol{P},\boldsymbol{Q}\in\mathbb{R}^{k\times k}$ satisfies $\boldsymbol{P}^T\boldsymbol{P}=\boldsymbol{Q}^T\boldsymbol{Q}=\boldsymbol{I}_k$, and $\boldsymbol{\Lambda}=\mathrm{diag}(\lambda_1,\ldots,\lambda_t,0,\ldots,0)\in\mathbb{R}^{k\times k}$.
Here, $\lambda_i$ is the $i$-th largest nonzero singular value of $\boldsymbol{U}_{S,W}$.
Define $\boldsymbol{\Lambda}_y=\mathrm{diag}(\lambda_1,\ldots,\lambda_t,y,\ldots,y)$.
By equation \eqref{eq:2025:xu41} we have
\begin{equation}\label{mmeq3}
\begin{aligned}
(-1)^{d}\cdot p_{S,W}(x;\boldsymbol{A},\boldsymbol{C},\boldsymbol{U},\boldsymbol{R})
&=\det\left[\begin{matrix}
-x\cdot \boldsymbol{I}_d  & \boldsymbol{A}^T & \boldsymbol{0}_{d\times k} & (\boldsymbol{R}_{S,:})^T  \\
\boldsymbol{A}&   \boldsymbol{I}_n & \boldsymbol{C}_{:,W}  & \boldsymbol{0}_{n\times k}  \\
\boldsymbol{R}_{S,:} &  \boldsymbol{0}_{k\times n} & \boldsymbol{P}\boldsymbol{\Lambda}_y\boldsymbol{Q}^T &  \boldsymbol{0}_{k\times k} \\
\boldsymbol{0}_{k\times d}  & (\boldsymbol{C}_{:,W})^T & \boldsymbol{0}_{k\times k} &   \boldsymbol{Q}\boldsymbol{\Lambda}_y\boldsymbol{P}^T
\end{matrix}\right]\ \bigg|_{y=0}\\
&=\det\left[\begin{matrix}
-x\cdot \boldsymbol{I}_d  & \boldsymbol{A}^T & \boldsymbol{0}_{d\times k} & (\boldsymbol{R}_{S,:})^T\boldsymbol{P}  \\
\boldsymbol{A}&   \boldsymbol{I}_n & \boldsymbol{C}_{:,W}\boldsymbol{Q}  & \boldsymbol{0}_{n\times k}  \\
\boldsymbol{P}^T\boldsymbol{R}_{S,:} &  \boldsymbol{0}_{k\times n} & \boldsymbol{\Lambda}_y &  \boldsymbol{0}_{k\times k} \\
\boldsymbol{0}_{k\times d}  & \boldsymbol{Q}^T(\boldsymbol{C}_{:,W})^T & \boldsymbol{0}_{k\times k} &  \boldsymbol{\Lambda}_y
\end{matrix}\right]\ \bigg|_{y=0}\\
&=\lambda_1^2\cdots\lambda_t^2\cdot y^{2(k-t)}\cdot h(x,y)\ \bigg|_{y=0}.
\end{aligned}	
\end{equation}
Here, the last equation follows from Lemma \ref{Linear algebra-lemma-schur}, and $h(x,y)$ is defined as 
\begin{equation}\label{mmeq2}
\begin{aligned}
h(x,y)
&:=\det\left[
\begin{pmatrix}
-x\cdot \boldsymbol{I}_d & \boldsymbol{A}^T\\
\boldsymbol{A}	 &   \boldsymbol{I}_n
\end{pmatrix}
-
\begin{pmatrix}
\boldsymbol{0}_{d\times k} &  (\boldsymbol{R}_{S,:})^T\boldsymbol{P}\\
\boldsymbol{C}_{:,W}\boldsymbol{Q} & \boldsymbol{0}_{n\times k}
\end{pmatrix}
\begin{pmatrix}
\boldsymbol{\Lambda}_{y}^{-1} & \boldsymbol{0}_{k\times k}  \\
\boldsymbol{0}_{k\times k} & \boldsymbol{\Lambda}_{y}^{-1} 
\end{pmatrix}
\begin{pmatrix}
\boldsymbol{P}^T\boldsymbol{R}_{S,:} & \boldsymbol{0}_{k\times n} \\
\boldsymbol{0}_{k\times d} & \boldsymbol{Q}^T(\boldsymbol{C}_{:,W})^T 
\end{pmatrix}
\right]\\	
&=(-1)^d\cdot 
\det[x\cdot \boldsymbol{I}_d+
(\boldsymbol{A}- \boldsymbol{C}_{:,W}\boldsymbol{Q}\boldsymbol{\Lambda}_{y}^{-1} \boldsymbol{P}^T \boldsymbol{R}_{S,:})^T
(\boldsymbol{A}- \boldsymbol{C}_{:,W}\boldsymbol{Q}\boldsymbol{\Lambda}_{y}^{-1} \boldsymbol{P}^T \boldsymbol{R}_{S,:})].\\
\end{aligned}
\end{equation}
Write $\boldsymbol{P}=[\boldsymbol{P}_1, \boldsymbol{P}_2]$ and $\boldsymbol{Q}=[\boldsymbol{Q}_1, \boldsymbol{Q}_2]$, where $\boldsymbol{P}_1,\boldsymbol{Q}_1\in\mathbb{R}^{k\times t}$ and $\boldsymbol{P}_2,\boldsymbol{Q}_2\in\mathbb{R}^{k\times (k-t)}$.
Note that
\begin{equation}\label{mmeq1}
\boldsymbol{Q}\boldsymbol{\Lambda}_{y}^{-1} \boldsymbol{P}^T=
\boldsymbol{Q}_1\cdot \mathrm{diag}(\lambda_1^{-1},\ldots,\lambda_t^{-1})\cdot  \boldsymbol{P}_1^T +	\frac{1}{y}\cdot \boldsymbol{Q}_2  \boldsymbol{P}_2^T=(\boldsymbol{U}_{S,W})^{\dagger}+\frac{1}{y}\cdot \boldsymbol{Q}_2  \boldsymbol{P}_2^T.
\end{equation}
Hence, substituting \eqref{mmeq2} and \eqref{mmeq1} into \eqref{mmeq3} we have
\begin{equation*}
p_{S,W}(x;\boldsymbol{A},\boldsymbol{C},\boldsymbol{U},\boldsymbol{R})
=\lambda_1^2\cdots\lambda_t^2\cdot y^{2(k-t)}\cdot \det[x\cdot \boldsymbol{I}_d+
(\boldsymbol{M}-\frac{1}{y}\cdot \boldsymbol{N})^T
(\boldsymbol{M}-\frac{1}{y}\cdot \boldsymbol{N})]\ \bigg|_{y=0},
\end{equation*}
where 
$\boldsymbol{M}=\boldsymbol{A}- \boldsymbol{C}_{:,W}(\boldsymbol{U}_{S,W})^{\dagger} \boldsymbol{R}_{S,:}\in\mathbb{R}^{n\times d}$ and
$\boldsymbol{N}=\boldsymbol{C}_{:,W}\boldsymbol{Q}_2   \boldsymbol{P}_2^T\boldsymbol{R}_{S,:}\in\mathbb{R}^{n\times d}$.
Then it follows from Lemma \ref{meq15-lemma} that
$p_{S,W}(x;\boldsymbol{A},\boldsymbol{C},\boldsymbol{U},\boldsymbol{R})$ is a degree-$(d-\mathrm{rank}(\boldsymbol{N}))$ polynomial with positive leading coefficient.

To prove the lemma, it remains to prove that $\mathrm{rank}(\boldsymbol{N})=k-t$. 
Note that $\boldsymbol{N}=\boldsymbol{C}_{:,W}\boldsymbol{Q}_2   \boldsymbol{P}_2^T\boldsymbol{R}_{S,:}\in\mathbb{R}^{n\times d}$, where $\boldsymbol{C}_{:,W}\boldsymbol{Q}_2\in\mathbb{R}^{n\times (k-t)}$ and $\boldsymbol{P}_2^T\boldsymbol{R}_{S,:}\in\mathbb{R}^{(k-t)\times d}$.
By Sylvester rank inequality \cite[Page 13]{HJ12}, we have
\begin{equation}\label{meq17}
\mathrm{rank}(\boldsymbol{C}_{:,W}\boldsymbol{Q}_2)+\mathrm{rank}(\boldsymbol{P}_2^T\boldsymbol{R}_{S,:})-(k-t)\leq \mathrm{rank}(\boldsymbol{N})\leq k-t.
\end{equation}
Since 
\begin{equation*}
k=\mathrm{rank}
\left(
\begin{matrix}
\boldsymbol{C}_{:,W}\\
\boldsymbol{U}_{S,W}	
\end{matrix}
\right)
=\mathrm{rank}
\left(
\begin{pmatrix}
\boldsymbol{I}_{n}&\\
&\boldsymbol{P^T}	
\end{pmatrix}\cdot 
\left(
\begin{matrix}
\boldsymbol{C}_{:,W}\\
\boldsymbol{U}_{S,W}	
\end{matrix}
\right)\cdot 	\boldsymbol{Q}\right)	
=\mathrm{rank}
\left(
\begin{matrix}
\boldsymbol{C}_{:,W}\boldsymbol{Q}\\
\boldsymbol{\Lambda}	
\end{matrix}
\right),
\end{equation*}
we have $\mathrm{rank}(\boldsymbol{C}_{:,W}\boldsymbol{Q}_2)=k-t$.
Similarly, we obtain $\mathrm{rank}(\boldsymbol{P}_2^T\boldsymbol{R}_{S,:})=k-t$. 
Combining this with \eqref{meq17}, we arrive at $\mathrm{rank}(\boldsymbol{N})= k-t$. This completes the proof.

\end{proof}

\subsection{Proof of Proposition \ref{P-exp-base-gcur-x}}\label{proof-P-exp-base-gcur-x}

We first introduce the following lemma.

\begin{lemma}{\rm \cite[Page 27]{HJ12}}\label{lemmas: det}
Let $d$ be a positive integer. Let $\boldsymbol{A},\boldsymbol{B},\boldsymbol{C}$ and $\boldsymbol{D}\in\mathbb{R}^{d\times d}$. If $\boldsymbol{C}\boldsymbol{D}=\boldsymbol{D}\boldsymbol{C}$, then 
\begin{equation*}
\det\left[\begin{matrix}
\boldsymbol{A}_{} & \boldsymbol{B}  \\
\boldsymbol{C}   &  \boldsymbol{D}_{} \\
\end{matrix}\right]= \det[\boldsymbol{A}\boldsymbol{D}-\boldsymbol{B}\boldsymbol{C}].
\end{equation*}

\end{lemma}

Now we can give a proof of Proposition \ref{P-exp-base-gcur-x}.

\begin{proof}[Proof of Proposition \ref{P-exp-base-gcur-x}]
A direct calculation shows that
\begin{equation}\label{eq:2025:xu43}
\begin{aligned}
\det\left[\begin{matrix}
\boldsymbol{I}_n  & \boldsymbol{0} & \boldsymbol{A} & \boldsymbol{A}  \\
\boldsymbol{0}& z\cdot \boldsymbol{I}_{n} & \boldsymbol{A}  & \boldsymbol{A}  \\
\boldsymbol{A}^{T}  &  \boldsymbol{A}^T & y\cdot \boldsymbol{I}_{d} &  \boldsymbol{0} \\
\boldsymbol{A}^{T}  & \boldsymbol{A}^{T} & \boldsymbol{0} & x\cdot  \boldsymbol{I}_d
\end{matrix}\right]
&=\det\left[\begin{matrix}
y\cdot \boldsymbol{I}_{d} &  \boldsymbol{0}&\boldsymbol{A}^{T}  &  \boldsymbol{A}^T  \\
\boldsymbol{0} & x\cdot  \boldsymbol{I}_d & \boldsymbol{A}^{T}  & \boldsymbol{A}^{T} \\
\boldsymbol{A} & \boldsymbol{A} & \boldsymbol{I}_n  & \boldsymbol{0}  \\
\boldsymbol{A}  & \boldsymbol{A} & \boldsymbol{0}& z\cdot \boldsymbol{I}_{n}  \\
\end{matrix}\right] \\
&\overset{(a)}=z^{n}
\det\left[
\begin{pmatrix}
y\cdot \boldsymbol{I}_{d} & \boldsymbol{0}   \\
\boldsymbol{0}& x\cdot \boldsymbol{I}_d   \\
\end{pmatrix}
-
\begin{pmatrix}
\boldsymbol{A}^{T}  &  \boldsymbol{A}^T  \\
\boldsymbol{A}^{T}  & \boldsymbol{A}^{T}   \\
\end{pmatrix}
\begin{pmatrix}
\boldsymbol{I}_{n} & \boldsymbol{0}   \\
\boldsymbol{0}& \frac{1}{z}\cdot \boldsymbol{I}_n  \\
\end{pmatrix}
\begin{pmatrix}
\boldsymbol{A} & \boldsymbol{A}  \\
\boldsymbol{A}  & \boldsymbol{A}  \\
\end{pmatrix}
\right]\\
&=z^{n}
\det\left[\begin{matrix}
y\cdot \boldsymbol{I}_{d} -(1+\frac{1}{z})\cdot \boldsymbol{A}^{T}\boldsymbol{A}  & -(1+\frac{1}{z})\cdot \boldsymbol{A}^{T}\boldsymbol{A}    \\
-(1+\frac{1}{z})\cdot \boldsymbol{A}^{T}\boldsymbol{A} & x\cdot \boldsymbol{I}_d -(1+\frac{1}{z})\cdot \boldsymbol{A}^{T}\boldsymbol{A}   \\
\end{matrix}
\right]\\
&\overset{(b)}=z^{n}
\det[(y\cdot \boldsymbol{I}_{d} -(1+\frac{1}{z}) \boldsymbol{A}^{T}\boldsymbol{A})(x\cdot \boldsymbol{I}_d -(1+\frac{1}{z}) \boldsymbol{A}^{T}\boldsymbol{A})- (1+\frac{1}{z})^2 (\boldsymbol{A}^{T}\boldsymbol{A})^2]\\
&=z^{n}y^d
\det[x\cdot \boldsymbol{I}_{d} - (1+\frac{x}{y}) (1+\frac{1}{z})\cdot \boldsymbol{A}^{T}\boldsymbol{A}],\\
\end{aligned}
\end{equation}	
where $(a)$ follows from Lemma \ref{Linear algebra-lemma-schur} and $(b)$ follows from Lemma \ref{lemmas: det}.
Note that for any multivariate polynomial $f(x,y,z)$ that has degree at most $n$ in $z$ we have
\begin{equation}\label{eq:2025:xu42}
\frac{1}{(n-k)!}\partial_{z}^{n-k}\ f(x,y,z)\ \bigg|_{z=0}
=\frac{1}{k!}\partial_{z}^{k}\Big( z^{n}f(x,y,1/z)\Big)\  \bigg|_{z=0}.	
\end{equation}
Hence, substituting \eqref{eq:2025:xu43} into \eqref{eq:P-exp-base-qpsw}, 
   and replacing $x$ with $-x$ therein,
 we obtain
\begin{equation}\label{eq:2025:xu43-12}
\begin{aligned}
P_{k}(-x;\boldsymbol{A},\boldsymbol{A},\boldsymbol{A},\boldsymbol{A})
&= \frac{(-1)^{d-k}}{ (d-k)!\cdot(n-k)!}\cdot \partial_y^{d-k} \partial_z^{n-k}\
z^{n}y^d
\det[x\cdot \boldsymbol{I}_{d} - (1+\frac{x}{y}) (1+\frac{1}{z})\cdot \boldsymbol{A}^{T}\boldsymbol{A}]  
\ \Big\vert_{y=z=0}\\
&\overset{(a)}=	 \frac{(-1)^{d-k}}{ (k!)^2}\cdot \partial_y^{k} \partial_z^{k}\  
\det[x\cdot \boldsymbol{I}_{d} -(1+x y) (1+z)\cdot \boldsymbol{A}^{T}\boldsymbol{A}]  \ \Big\vert_{y=z=0}\\
&\overset{(b)}=	 \frac{(-1)^{d-k}x^k}{ (k!)^2}\cdot \partial_y^{k} \partial_z^{k}\  
\det[x\cdot \boldsymbol{I}_{d} - yz\cdot \boldsymbol{A}^{T}\boldsymbol{A}]  \ \Big\vert_{y=z=1},\\
\end{aligned}
\end{equation}
where ($a$) follows from \eqref{eq:2025:xu42} and ($b$) follows from the chain rule.
Let us write $\det[x\cdot \boldsymbol{I}_{d} -\boldsymbol{A}^{T}\boldsymbol{A}]=\sum_{i=0}^{d}c_ix^i	$.
Then we have
\begin{equation}\label{eq:2025:xu43-11}
\det[x\cdot \boldsymbol{I}_{d} - yz\cdot \boldsymbol{A}^{T}\boldsymbol{A}]=
y^dz^d\det[\frac{x}{yz}\cdot \boldsymbol{I}_{d} -  \boldsymbol{A}^{T}\boldsymbol{A}]=
\sum_{i=0}^{d}c_ix^iy^{d-i}z^{d-i}.	
\end{equation}
Substituting \eqref{eq:2025:xu43-11} into \eqref{eq:2025:xu43-12}, we obtain
\begin{equation*}
\begin{aligned}
P_{k}(-x;\boldsymbol{A},\boldsymbol{A},\boldsymbol{A},\boldsymbol{A})
&= \frac{(-1)^{d-k}x^k}{ (k!)^2} \partial_y^{k} \partial_z^{k}\  
\sum_{i=0}^{d}c_ix^iy^{d-i}z^{d-i}  \ \Big\vert_{y=z=1}\\
&= \frac{(-1)^{d-k}x^k}{ (k!)^2}  
\sum_{i=0}^{d-k} \Big(\frac{(d-i)!}{(d-i-k)!}\Big)^2 c_i x^i\\
&=\frac{(-1)^{d-k}}{(k!)^2}\cdot  \mathcal{R}_{x,d}^+\cdot  (\partial_x\cdot x\cdot \partial_x)^k \cdot \mathcal{R}_{x,d}^+\   \det[x\cdot \boldsymbol{I}_d-\boldsymbol{A}^T\boldsymbol{A}],
\end{aligned}
\end{equation*}
where the last equation follows from Lemma \ref{Linear algebra-lemma-6}. Thus, we arrive at \eqref{eq:P-exp-base-p1}.

\end{proof}

\subsection{Proof of Proposition \ref{P-exp-base-gcur-x-newcase}}\label{proof-P-exp-base-gcur-x-newcase}

\begin{proof}[Proof of Proposition \ref{P-exp-base-gcur-x-newcase}]
Note that by Proposition \ref{P-exp-base-gcur2}, we have
\begin{equation}\label{2026-xu1}
P_{k}(-x;\boldsymbol{A},\boldsymbol{C},\boldsymbol{C},\boldsymbol{A})
=\frac{(-1)^{d-k}}{(n-k)!}\cdot  \partial_z^{n-k}\ 
h(x,z) 
\ \Big\vert_{z=0}
=\frac{(-1)^{d-k}}{k!}\cdot  \partial_z^{k}\ 
\Big( z^n\cdot h(x,1/z)\Big) 
\ \Big\vert_{z=0},	
\end{equation}
where the last equation follows from \eqref{eq:2025:xu42} and
\begin{equation*}
h(x,z):=\det\left[\begin{matrix}
\boldsymbol{I}_n  & \boldsymbol{0} & \boldsymbol{C} & \boldsymbol{A}  \\
\boldsymbol{0}& z\cdot \boldsymbol{I}_{n} & \boldsymbol{C}  & \boldsymbol{A}  \\
\boldsymbol{C}^{T}  &  \boldsymbol{C}^T &  \boldsymbol{0}_{} &  \boldsymbol{0} \\
\boldsymbol{A}^{T}  & \boldsymbol{A}^{T} & \boldsymbol{0} & x\cdot  \boldsymbol{I}_d
\end{matrix}\right].	
\end{equation*}
We first simplify the formula of $h(x,z)$.
Let $\boldsymbol{C}=\boldsymbol{P}\boldsymbol{\Lambda}\boldsymbol{Q}^T$ be the full SVD of $\boldsymbol{C}$, 
where $\boldsymbol{P}\in\mathbb{R}^{n\times n},\boldsymbol{Q}\in\mathbb{R}^{k\times k}$ satisfies $\boldsymbol{P}^T\boldsymbol{P}=\boldsymbol{I}_n$, $\boldsymbol{Q}^T\boldsymbol{Q}=\boldsymbol{I}_k$, and $\boldsymbol{\Lambda}=(\boldsymbol{\Lambda}_1,\boldsymbol{0}_{k\times (n-k)})^T\in\mathbb{R}^{n\times k}$ with $\boldsymbol{\Lambda}_1:=\mathrm{diag}(\lambda_1,\ldots,\lambda_k)\in\mathbb{R}^{k\times k}$.
Here, $\lambda_i$ is the $i$-th largest nonzero singular value of $\boldsymbol{C}_{}$.
Write $\boldsymbol{P}=(\boldsymbol{P}_1, \boldsymbol{P}_2)$, where $\boldsymbol{P}_1\in\mathbb{R}^{n\times k}$ and $\boldsymbol{P}_2\in\mathbb{R}^{n\times (n-k)}$.
A direct calculation shows that
\begin{equation}\label{eq:2025:xu43}
\begin{aligned}
h(x,z)
&=\det\left[\begin{matrix}
\boldsymbol{I}_n  & \boldsymbol{0} & \boldsymbol{P}\boldsymbol{\Lambda}\boldsymbol{Q}^T & \boldsymbol{A}  \\
\boldsymbol{0}& z\cdot \boldsymbol{I}_{n} & \boldsymbol{P}\boldsymbol{\Lambda}\boldsymbol{Q}^T  & \boldsymbol{A}  \\
\boldsymbol{Q}\boldsymbol{\Lambda}^T\boldsymbol{P}^T  &  \boldsymbol{Q}\boldsymbol{\Lambda}^T\boldsymbol{P}^T &  \boldsymbol{0}_{} &  \boldsymbol{0} \\
\boldsymbol{A}^{T}  & \boldsymbol{A}^{T} & \boldsymbol{0} & x\cdot  \boldsymbol{I}_d
\end{matrix}\right]
\\&
=\det\left[\begin{matrix}
\boldsymbol{I}_n  & \boldsymbol{0} & \boldsymbol{\Lambda} & \boldsymbol{P}^T\boldsymbol{A}  \\
\boldsymbol{0}& z\cdot \boldsymbol{I}_{n} & \boldsymbol{\Lambda}  & \boldsymbol{P}^T\boldsymbol{A}  \\
\boldsymbol{\Lambda}^T  &  \boldsymbol{\Lambda}^T &  \boldsymbol{0}_{} &  \boldsymbol{0} \\
\boldsymbol{A}^{T}\boldsymbol{P}  & \boldsymbol{A}^{T}\boldsymbol{P} & \boldsymbol{0} & x\cdot  \boldsymbol{I}_d
\end{matrix}\right]
\overset{(a)}=(-1)^k \det\left[\begin{matrix}
\boldsymbol{M}_1  & \boldsymbol{M}_2  \\
\boldsymbol{M}_3& \boldsymbol{M}_{4} \\
\end{matrix}\right],\\
\end{aligned}
\end{equation}	
where
\begin{equation*}
\boldsymbol{M}_1:=\begin{pmatrix}
\boldsymbol{0}  & \boldsymbol{0} & \boldsymbol{P}_1^T\boldsymbol{A} & \boldsymbol{0}  \\
\boldsymbol{0}& \boldsymbol{I}_{n-k} & \boldsymbol{P}_2^T\boldsymbol{A}  & \boldsymbol{0}  \\
\boldsymbol{A}^T\boldsymbol{P}_1  &  \boldsymbol{A}^T\boldsymbol{P}_2 &  x\cdot \boldsymbol{I}_{d} &  \boldsymbol{A}^T\boldsymbol{P}_2 \\
\boldsymbol{0}  & \boldsymbol{0} & \boldsymbol{P}_2^T\boldsymbol{A} & z\cdot  \boldsymbol{I}_{n-k}
\end{pmatrix},
\quad 
\boldsymbol{M}_2:=\begin{pmatrix}
\boldsymbol{\Lambda}_1  & z\cdot \boldsymbol{I}_k \\
\boldsymbol{0}&  \boldsymbol{0}  \\
\boldsymbol{0} & \boldsymbol{A}^T\boldsymbol{P}_1   \\
\boldsymbol{0}  & \boldsymbol{0}
\end{pmatrix},	
\end{equation*}
\begin{equation*}
\boldsymbol{M}_3:=\begin{pmatrix}
\boldsymbol{\Lambda}_1  & \boldsymbol{0} & \boldsymbol{0}  & \boldsymbol{0}  \\
\boldsymbol{I}_k& \boldsymbol{0} & \boldsymbol{P}_1^T\boldsymbol{A}  & \boldsymbol{0}  \\
\end{pmatrix}
\quad\text{and}\quad
\boldsymbol{M}_4:=\begin{pmatrix}
\boldsymbol{0}  & \boldsymbol{\Lambda}_1 \\
\boldsymbol{\Lambda}_1 &  \boldsymbol{0}  \\
\end{pmatrix}.	
\end{equation*}
Here, equality ($a$) follows from the fact that the sign of the determinant changes when interchanging two rows of a square matrix. 
Then we have
\begin{equation}\label{2026-xu2}
\begin{aligned}
h(x,z)&\overset{(b)}=(-1)^k\cdot \det[\boldsymbol{M}_4]\cdot \det[\boldsymbol{M}_1-\boldsymbol{M}_2\boldsymbol{M}_4^{-1}\boldsymbol{M}_3]\\
&=(\det \boldsymbol{\Lambda}_1)^2\cdot 
\det \left[\begin{matrix}
-(z+1)\cdot \boldsymbol{I}_k  & \boldsymbol{0} & \boldsymbol{0} & \boldsymbol{0}  \\
\boldsymbol{0}& \boldsymbol{I}_{n-k} & \boldsymbol{P}_2^T\boldsymbol{A}  & \boldsymbol{0}  \\
\boldsymbol{0}  &  \boldsymbol{A}^T\boldsymbol{P}_2 &  x\cdot \boldsymbol{I}_{d} &  \boldsymbol{A}^T\boldsymbol{P}_2 \\
\boldsymbol{0}  & \boldsymbol{0} & \boldsymbol{P}_2^T\boldsymbol{A} & z\cdot  \boldsymbol{I}_{n-k}
\end{matrix}\right]
\\&
=(\det \boldsymbol{\Lambda}_1)^2\cdot 
\det \left[\begin{matrix}
 x\cdot \boldsymbol{I}_{d} &  \boldsymbol{A}^T\boldsymbol{P}_2 &\boldsymbol{0}  &  \boldsymbol{A}^T\boldsymbol{P}_2  \\
\boldsymbol{P}_2^T\boldsymbol{A} & z\cdot  \boldsymbol{I}_{n-k}&\boldsymbol{0}  & \boldsymbol{0} \\
 \boldsymbol{0} & \boldsymbol{0}&-(z+1)\cdot \boldsymbol{I}_k  & \boldsymbol{0}   \\
\boldsymbol{P}_2^T\boldsymbol{A}  & \boldsymbol{0} &\boldsymbol{0}& \boldsymbol{I}_{n-k} 
\end{matrix}\right]\\
&\overset{(c)}=(\det \boldsymbol{\Lambda}_1)^2\cdot (-1)^k\cdot z^{n-k}\cdot  (z+1)^k\cdot 	\det [x\cdot \boldsymbol{I}_d-(1+\frac{1}{z})\cdot\boldsymbol{A}^T\boldsymbol{P}_2\boldsymbol{P}_2^T\boldsymbol{A}],\\
&\overset{(d)}=\det\boldsymbol{C}^T\boldsymbol{C}\cdot (-1)^k\cdot z^{n-k}\cdot  (z+1)^k\cdot 	\det [x\cdot \boldsymbol{I}_d-(1+\frac{1}{z})\cdot\boldsymbol{A}^T(\boldsymbol{I}_n-\boldsymbol{C}\boldsymbol{C}^{\dagger})\boldsymbol{A}],\\
\end{aligned}	
\end{equation}
where ($b$), ($c$) follows from the Schur determinantal formula in Lemma \ref{Linear algebra-lemma-schur}, and ($d$) follows from
\begin{equation*}
\boldsymbol{A}^T\boldsymbol{P}_2\boldsymbol{P}_2^T\boldsymbol{A}=\boldsymbol{A}^T(\boldsymbol{I}_n-\boldsymbol{P}_1\boldsymbol{P}_1^T)\boldsymbol{A}=\boldsymbol{A}^T(\boldsymbol{I}_n-\boldsymbol{C}\boldsymbol{C}^{\dagger})\boldsymbol{A}
\quad\text{and}\quad
(\det \boldsymbol{\Lambda}_1)^2=\det\boldsymbol{C}^T\boldsymbol{C}.	
\end{equation*}
Substituting \eqref{2026-xu2} into \eqref{2026-xu1},
we obtain
\begin{equation}\label{2026-xu4}
\begin{aligned}
P_{k}(-x;\boldsymbol{A},\boldsymbol{C},\boldsymbol{C},\boldsymbol{A})
&=\frac{(-1)^{d}\det\boldsymbol{C}^T\boldsymbol{C}}{k!}\cdot  \partial_z^{k}\ 
\bigg( (z+1)^k\cdot 	\det [x\cdot \boldsymbol{I}_d-(1+z)\cdot\boldsymbol{A}^T(\boldsymbol{I}_n-\boldsymbol{C}\boldsymbol{C}^{\dagger})\boldsymbol{A}] \bigg) 
\ \Bigg\vert_{z=0}\\
&=\frac{(-1)^{d}\det\boldsymbol{C}^T\boldsymbol{C}}{k!}\cdot  \partial_z^{k}\ 
\bigg( z^k\cdot 	\det [x\cdot \boldsymbol{I}_d-z\cdot \boldsymbol{A}^T(\boldsymbol{I}_n-\boldsymbol{C}\boldsymbol{C}^{\dagger})\boldsymbol{A}] \bigg) 
\ \Bigg\vert_{z=1},	
\end{aligned}	
\end{equation}
where the last equation follows from the chain rule.
Write $\det[x\cdot \boldsymbol{I}_{d} -\boldsymbol{A}^T(\boldsymbol{I}_n-\boldsymbol{C}\boldsymbol{C}^{\dagger})\boldsymbol{A}]=\sum_{i=0}^{d}c_ix^i	$.
Then we have
\begin{equation}\label{2026-xu3}
\det [x\cdot \boldsymbol{I}_d-z\cdot \boldsymbol{A}^T(\boldsymbol{I}_n-\boldsymbol{C}\boldsymbol{C}^{\dagger})\boldsymbol{A}]=
z^d\det[\frac{x}{z}\cdot \boldsymbol{I}_{d} -  \boldsymbol{A}^T(\boldsymbol{I}_n-\boldsymbol{C}\boldsymbol{C}^{\dagger})\boldsymbol{A}]=
\sum_{i=0}^{d}c_ix^iz^{d-i}.	
\end{equation}
Substituting \eqref{2026-xu3} into \eqref{2026-xu4}, we obtain
\begin{equation}\label{2026-xu5}
\begin{aligned}
P_{k}(-x;\boldsymbol{A},\boldsymbol{C},\boldsymbol{C},\boldsymbol{A})
&= \frac{(-1)^{d}\det\boldsymbol{C}^T\boldsymbol{C}}{k!}\cdot  \partial_z^{k}\ 
\bigg(z^k  \sum_{i=0}^{d}c_ix^iz^{d-i} \bigg) 
\ \Bigg\vert_{z=1}\\
&= \frac{(-1)^{d}\det\boldsymbol{C}^T\boldsymbol{C}}{k!}   \sum_{i=0}^{d}\frac{c_i(d-i+k)!}{(d-i)!}x^i.\\
\end{aligned}
\end{equation}
On the other hand, a direct calculation shows that
\begin{equation}\label{2026-xu6}
\begin{aligned}
&\mathcal{R}_{x,d}^+ \cdot   \partial_x^k\cdot x^k \cdot \mathcal{R}_{x,d}^+\   \det[x\cdot \boldsymbol{I}_d-\boldsymbol{A}^T(\boldsymbol{I}_n-\boldsymbol{C}\boldsymbol{C}^{\dagger})\boldsymbol{A}]
\\&
=\mathcal{R}_{x,d}^+ \cdot \partial_x^k\cdot x^k \cdot \mathcal{R}_{x,d}^+\  \sum_{i=0}^{d}c_ix^i 
=\mathcal{R}_{x,d}^+ \cdot \partial_x^k \cdot \sum_{i=0}^{d}c_ix^{d-i+k}\\
&=\mathcal{R}_{x,d}^+\cdot\sum_{i=0}^{d}\frac{c_i(d-i+k)! }{(d-i)!}x^{d-i}
=\sum_{i=0}^{d}\frac{c_i(d-i+k)! }{(d-i)!}x^{i}.
\end{aligned}
\end{equation}
Substituting \eqref{2026-xu6} into \eqref{2026-xu5}, we  arrive at \eqref{eq:P-exp-base-p1-newcase}.

\end{proof}


\subsection{Proof of Lemma \ref{lemma2-f-SW}}\label{proof-lemma2-f-SW}

\begin{proof}[Proof of Lemma \ref{lemma2-f-SW}]
(i) We first prove \eqref{meq30}.
By Proposition \ref{P-exp-base-gcur1} (i), we have
\begin{equation}\label{meq40}
\begin{aligned}
&f_{\widetilde{S} ,\widetilde{W}}(-x;\boldsymbol{A},\boldsymbol{C},\boldsymbol{U},\boldsymbol{R})\\
&=(-1)^{d-k}\cdot\bigg( 
\sum_{\substack{S\in\binom{[n_R]}{k},\widetilde{S} \subset S}} \partial_{\boldsymbol{Z}^{S^C}}\bigg)
\cdot \bigg(
\sum_{\substack{W\in\binom{[d_C]}{k},\widetilde{W}\subset W\\}}
\partial_{\boldsymbol{Y}^{W^C}}\bigg) \ 
Q(x,\boldsymbol{Y},\boldsymbol{Z};\boldsymbol{A},\boldsymbol{C},\boldsymbol{U},\boldsymbol{R})  \ \Bigg\vert_{\substack{y_i=0,\forall i\in[d_C]\\z_j=0,\forall j\in[n_R]}}.\\
\end{aligned}	
\end{equation}
Note that for any multiaffine polynomial $f(x_1,\ldots,x_m)$ and any subset $S'\subset[m]$ of size at most $k$, we have
\begin{equation*}
\begin{aligned}
\sum_{\substack{S\in\binom{[m]}{k},S'\subset S\\ }}	
\partial_{\boldsymbol{X}^{S^C}} \ f(x_1,\ldots,x_m)\ \bigg|_{x_i=0,\forall i}
&=(-1)^k\cdot  \sum_{\substack{S\in\binom{[m]}{k}, S'\subset S}}	
\partial_{\boldsymbol{X}^{S}}\ \widetilde{f}(x_1,\ldots,x_m)\ \bigg|_{x_i=0,\forall i}\\
&=(-1)^k\cdot \Big(\sum_{\substack{S\subset[m], |S|=k-|S'|}}	
\partial_{\boldsymbol{X}^{S}}\Big) \cdot \partial_{\boldsymbol{X}^{S'}}\cdot  \widetilde{f}(x_1,\ldots,x_m)\ \bigg|_{x_i=0,\forall i},\\
\end{aligned}
\end{equation*}
where $\widetilde{f}(x_1,\ldots,x_m):=(\prod_{i=1}^{m}\mathcal{R}_{x_i,1}^-)f=x_1\cdots x_m\cdot  f(-1/x_1,\ldots,-1/x_m)$.
Therefore,
\begin{equation*}
\begin{aligned}
f_{\widetilde{S} ,\widetilde{W}}(-x;\boldsymbol{A},\boldsymbol{C},\boldsymbol{U},\boldsymbol{R})
&=(-1)^{d-k}\cdot 
\Big(\sum_{S\subset[n_R], |S|=k-|\widetilde{S}|}	\partial_{\boldsymbol{Z}^{S}}\Big) 
\Big(\sum_{W\subset[d_C], |W|=k-|\widetilde{W}|}	\partial_{\boldsymbol{Y}^{W}}\Big) 
\\
&\ \ \partial_{\boldsymbol{Y}^{\widetilde{W}}}
\partial_{\boldsymbol{Z}^{\widetilde{S}}} 
\Big(\prod_{i=1}^{d_C}\mathcal{R}_{y_i,1}^-\Big)
\Big(\prod_{i=1}^{n_R}\mathcal{R}_{z_i,1}^-\Big)
Q(x,\boldsymbol{Y},\boldsymbol{Z};\boldsymbol{A},\boldsymbol{C},\boldsymbol{U},\boldsymbol{R})  \ \Bigg\vert_{\substack{y_i=0,\forall i\in[d_C]\\z_j=0,\forall j\in[n_R]}}\\
&=\frac{(-1)^{d-k}}{(k-|\widetilde{W}|)!(k-|\widetilde{S}|)!}
\partial_y^{k-|\widetilde{W}|}
\partial_z^{k-|\widetilde{S}|}
g(x,y,z)
\ \Bigg\vert_{y=z=0},\\ 
\end{aligned}	
\end{equation*}
where the last equation follows from \eqref{meq31}. We arrive at \eqref{meq30}.

We next prove that $f_{\widetilde{S},\widetilde{W}}(x;\boldsymbol{A},\boldsymbol{C},\boldsymbol{U},\boldsymbol{R})$ is either identically zero or real-rooted.
By Lemma \ref{real stable} and Lemma \ref{real stable2}, we obtain that $g(x,y,z)$ defined in \eqref{meq32} is identically zero or real stable in $x,y$ and $z$.
Since differential operator $\partial_{y}$ and $\partial_{z}$ preserve real stability, the univariate polynomial $f_{\widetilde{S},\widetilde{W}}(-x;\boldsymbol{A},\boldsymbol{C},\boldsymbol{U},\boldsymbol{R})$ is identically zero or real-rooted. 
Hence, $f_{\widetilde{S},\widetilde{W}}(x;\boldsymbol{A},\boldsymbol{C},\boldsymbol{U},\boldsymbol{R})$ is either identically zero or real-rooted. 
This completes the proof.  

(ii)  
Note that for any multiaffine polynomial $f(x_1,\ldots,x_m)$ and any subset $S'\subset[m]$ of size at most $k$, we have
\begin{equation*}
\begin{aligned}
\sum_{\substack{S\subset[m]\\ |S|=k,S'\subset S}}	
\partial_{\boldsymbol{X}^{S^C}}\ f(x_1,\ldots,x_m)\ \bigg|_{x_i=0,\forall i}
&=\sum_{\substack{S\subset[m]\backslash S'\\ |S|=m-k }}	
\partial_{\boldsymbol{X}^{S}}\ {f}(x_1,\ldots,x_m)\ \bigg|_{x_i=0,\forall i}\\
&=\frac{1}{(m-k)!}\cdot  
\partial_{x}^{m-k}\bigg( {f}(x_1,\ldots,x_m)\ \bigg|_{\substack{x_i=0,\forall i\in S' \\ x_i=x,\forall i\notin S'}}\bigg)\ \Bigg\vert_{x=0}.\\
\end{aligned}
\end{equation*}
Combining with \eqref{meq40} we have
\begin{equation}\label{meq50}
\begin{aligned}
f_{\widetilde{S},\widetilde{W}}(-x;\boldsymbol{A},\boldsymbol{A},\boldsymbol{A},\boldsymbol{A})
=\frac{(-1)^{d-k}}{(n-k)!(d-k)!}\cdot  
\partial_{z}^{n-k}\partial_{y}^{d-k}
\ h(x,y,z)
\ \Bigg\vert_{y=z=0},\\ 
\end{aligned}	
\end{equation}
where 
\begin{equation*}
\begin{aligned}
h(x,y,z)
&:=	
\det\left[\begin{matrix}
\boldsymbol{I}_n  & \boldsymbol{0} & \boldsymbol{A} & \boldsymbol{A}  \\
\boldsymbol{0}& {\boldsymbol{Z}}_{} & \boldsymbol{A}  & \boldsymbol{A}  \\
\boldsymbol{A}^{T}  &  \boldsymbol{A}^T & {\boldsymbol{Y}} &  \boldsymbol{0} \\
\boldsymbol{A}^{T}  & \boldsymbol{A}^{T} & \boldsymbol{0} & x\cdot  \boldsymbol{I}_d
\end{matrix}\right] 
\ \Bigg\vert_{\substack{y_i=0,\forall i\in\widetilde{W},y_i=y,\forall i\notin\widetilde{W}\\z_i=0,\forall i\in\widetilde{S},z_i=z,\forall i\notin\widetilde{S}}}.
\end{aligned}
\end{equation*}
Here, $\boldsymbol{Z}=\mathrm{diag}(z_1,\ldots,z_n)$ and $\boldsymbol{Y}=\mathrm{diag}(y_1,\ldots,y_d)$.
Note that the sign of the determinant changes when interchanging two rows of a square matrix, so we have
\begin{equation}
\label{eq:2025:xu41}
\begin{aligned}
h(x,y,z)
&=\det\left[\begin{matrix}
\boldsymbol{M}_1 & \boldsymbol{M}_2 \\
\boldsymbol{M}_2^T &  \boldsymbol{M}_3 \\
\end{matrix}\right],
\end{aligned}
\end{equation}
where 
\begin{equation*}
\boldsymbol{M}_2:=
\begin{pmatrix}
\boldsymbol{A}_{:,\widetilde{W}}&   \boldsymbol{0} \\
\boldsymbol{A}_{\widetilde{S}^C,\widetilde{W}} &  \boldsymbol{0}  \\
\boldsymbol{0} &  (\boldsymbol{A}_{\widetilde{S},\widetilde{W}^C})^T\\
\boldsymbol{0} & (\boldsymbol{A}_{\widetilde{S},:})^T   \\
\end{pmatrix}
\quad 
\boldsymbol{M}_3:=
\begin{pmatrix}
\boldsymbol{0} & (\boldsymbol{A}_{\widetilde{S},\widetilde{W}})^T \\
\boldsymbol{A}_{\widetilde{S},\widetilde{W}}& \boldsymbol{0}\\
\end{pmatrix},	
\end{equation*}
and
\begin{equation*}
\boldsymbol{M}_1:=\begin{pmatrix}
\boldsymbol{I}_n & \boldsymbol{0}   & \boldsymbol{A}_{:,\widetilde{W}^C} & \boldsymbol{A}   \\
\boldsymbol{0}  &  z\cdot \boldsymbol{I}_{n-l} & \boldsymbol{A}_{\widetilde{S}^C,\widetilde{W}^C} &  \boldsymbol{A}_{\widetilde{S}^C,:}    \\
(\boldsymbol{A}_{:,\widetilde{W}^C})^T & (\boldsymbol{A}_{\widetilde{S}^C,\widetilde{W}^C})^T & y\cdot \boldsymbol{I}_{d-l} & \boldsymbol{0}    \\
\boldsymbol{A}^T & (\boldsymbol{A}_{\widetilde{S}^C,:})^T & \boldsymbol{0}  & x\cdot \boldsymbol{I}_d    \\
\end{pmatrix}.
\end{equation*}
Since $\boldsymbol{A}_{\widetilde{S},\widetilde{W}}$ is invertible and $\det \boldsymbol{M}_3=(-1)^l\cdot \det[\boldsymbol{A}_{\widetilde{S},\widetilde{W}}]^2$, using Lemma \ref{Linear algebra-lemma-schur} we obtain
\begin{equation*}
\begin{aligned}
h(x,y,z)&=\det \boldsymbol{M}\cdot \det [\boldsymbol{M}_1-\boldsymbol{M}_2\boldsymbol{M}_3^{-1}\boldsymbol{M}_2^T]\\
&=(-1)^l\cdot \det[\boldsymbol{A}_{\widetilde{S},\widetilde{W}}]^2\cdot 
\det\left[\begin{matrix}
\boldsymbol{I}_n & \boldsymbol{0}   & \boldsymbol{B}_{:,\widetilde{W}^C} &\boldsymbol{B} \\
\boldsymbol{0}  &  z\cdot \boldsymbol{I}_{n-l} & \boldsymbol{B}_{\widetilde{S}^C,\widetilde{W}^C}&\boldsymbol{B}_{\widetilde{S}^C,:}  \\
(\boldsymbol{B}_{:,\widetilde{W}^C})^T & (\boldsymbol{B}_{\widetilde{S}^C,\widetilde{W}^C})^T & y\cdot \boldsymbol{I}_{d-l} &\boldsymbol{0}    \\
\boldsymbol{B}^T & (\boldsymbol{B}_{\widetilde{S}^C,:})^T & \boldsymbol{0}&x\cdot \boldsymbol{I}_d     \\
\end{matrix}\right]\\
&=(-1)^l\cdot \det[\boldsymbol{A}_{\widetilde{S},\widetilde{W}}]^2\cdot x^l\cdot  
\det\left[\begin{matrix}
\boldsymbol{I}_{n-l} & \boldsymbol{0}   & \boldsymbol{N} &\boldsymbol{N} \\
\boldsymbol{0}  &  z\cdot \boldsymbol{I}_{n-l} & \boldsymbol{N}&\boldsymbol{N}  \\
\boldsymbol{N}^T & \boldsymbol{N}^T & y\cdot \boldsymbol{I}_{d-l} &\boldsymbol{0}    \\
 \boldsymbol{N}^T & \boldsymbol{N}^T & \boldsymbol{0}&x\cdot \boldsymbol{I}_{d-l}     \\
\end{matrix}\right],\\
\end{aligned}
\end{equation*}
where $\boldsymbol{N} :=\boldsymbol{B}_{\widetilde{S}^C,\widetilde{W}^C}\in \mathbb{R}^{(n-l)\times (d-l)}$.
Here, in the last equation we use the fact that $\boldsymbol{B}(i,j)=0$ if $i\in \widetilde{S}$ or $j\in \widetilde{W}$, which follows from Lemma \ref{Linear algebra-lemma-2} (i).
Substituting above equation into \eqref{meq50}, we obtain
\begin{equation*}
\begin{aligned}
&f_{\widetilde{S},\widetilde{W}}(-x;\boldsymbol{A},\boldsymbol{A},\boldsymbol{A},\boldsymbol{A})\\
&=\frac{(-1)^{d-k+l}\cdot \det[\boldsymbol{A}_{\widetilde{S},\widetilde{W}}]^2\cdot  x^l}{(n-k)!(d-k)!}\cdot  
\partial_{z}^{n-k}\partial_{y}^{d-k}
\ \det\left[\begin{matrix}
\boldsymbol{I}_{n-l} & \boldsymbol{0}   & \boldsymbol{N} &\boldsymbol{N} \\
\boldsymbol{0}  &  z\cdot \boldsymbol{I}_{n-l} & \boldsymbol{N}&\boldsymbol{N}  \\
\boldsymbol{N}^T & \boldsymbol{N}^T & y\cdot \boldsymbol{I}_{d-l} &\boldsymbol{0}    \\
 \boldsymbol{N}^T & \boldsymbol{N}^T & \boldsymbol{0}&x\cdot \boldsymbol{I}_{d-l}     \\
\end{matrix}\right]
\ \Bigg\vert_{y=z=0},\\ 
&\overset{(a)}=(-1)^{l}\cdot \det[\boldsymbol{A}_{\widetilde{S},\widetilde{W}}]^2\cdot  x^l\cdot  
P_{k-l}(x;\boldsymbol{N},\boldsymbol{N},\boldsymbol{N},\boldsymbol{N})\\
&\overset{(b)}=\frac{(-1)^{d-k+l}\cdot \det[\boldsymbol{A}_{\widetilde{S},\widetilde{W}}]^2\cdot  x^l}{(k-l)!(k-l)!}\cdot  
\mathcal{R}_{x,d-l}^+\cdot (\partial_x\cdot x \cdot \partial_x)^{k-l} \cdot \mathcal{R}_{x,d-l}^+\   \det[x\cdot \boldsymbol{I}_{d-l}-\boldsymbol{N}^T\boldsymbol{N}],\\
\end{aligned}	
\end{equation*}
where ($a$) follows from Proposition \ref{P-exp-base-gcur2} and ($b$) follows from Proposition \ref{P-exp-base-gcur-x}.
Recall that $\boldsymbol{B}(i,j)=0$ if $i\in \widetilde{S}$ or $j\in \widetilde{W}$, so we have
\begin{equation*}
x^{l}\cdot \det[x\cdot \boldsymbol{I}_{d-l}-\boldsymbol{N}^T\boldsymbol{N}]
=\det[x\cdot \boldsymbol{I}_{d}-\boldsymbol{B}^T\boldsymbol{B}].	
\end{equation*}
Hence, we obtain that
\begin{equation*}
\begin{aligned}
f_{\widetilde{S},\widetilde{W}}(-x;\boldsymbol{A},\boldsymbol{A},\boldsymbol{A},\boldsymbol{A})=\frac{(-1)^{d-k+l}\cdot \det[\boldsymbol{A}_{\widetilde{S},\widetilde{W}}]^2}{(k-l)!(k-l)!}\cdot  
\mathcal{R}_{x,d}^+\cdot (\partial_x \cdot x \cdot \partial_x)^{k-l} \cdot \mathcal{R}_{x,d}^+\   \det[x\cdot \boldsymbol{I}_{d}-\boldsymbol{B}^T\boldsymbol{B}].\\
\end{aligned}	
\end{equation*}
We arrive at \eqref{meq35}.
Comparing the roots of the polynomials on the both sides of \eqref{meq35}, we obtain \eqref{meq350}. This completes the proof.

\end{proof}

\subsection{Proof of Lemma \ref{minroot-new}}\label{proof-minroot-new}

To prove Lemma \ref{minroot-new}, we use the barrier function method introduced in \cite{inter2,inter3}. We first introduce the following definition. 

\begin{definition}
For a real-rooted degree-$t$ polynomial $f(x)=\prod_{i=1}^{t}(x-\beta_i)$, we define the lower barrier function
\begin{equation*}
\Phi_f(x):=\sum_{i=1}^{t}\frac{1}{\beta_i-x}.
\end{equation*}	
\end{definition}

\begin{lemma}{\rm (see \cite[Lemma 4.3]{inter3} and \cite[Lemma 2.11]{XX21})}\label{barrier-lemma}
Suppose that $f(x)$ is a real-rooted polynomial and $\delta> 0$. Suppose that $b < \mathrm{minroot}\ f(x)$ and $\Phi_f(b)\leq\frac{1}{\delta}$.
Then $b+\delta\leq \mathrm{minroot}\ \partial_xf(x)$ and
\begin{equation*}
\Phi_{\partial_x f}(b+\delta)\leq 	\Phi_f(b).
\end{equation*}

\end{lemma}

Now we can give a proof of Lemma \ref{minroot-new}.

\begin{proof}
[Proof of Lemma \ref{minroot-new}]
Let $b<\beta_1$ to be determined later.	Let $\delta=\frac{1}{\Phi_f(b)}$. 
By Lemma \ref{barrier-lemma} we have $b+\delta\leq \mathrm{minroot}\ \partial_xf(x)$ and
\begin{equation*}
\Phi_{\partial_x f}(b+\delta)\leq 	\Phi_f(b)=\frac{1}{\delta}.
\end{equation*} 
Applying Lemma \ref{barrier-lemma} again, we have $b+2\delta\leq \mathrm{minroot}\ \partial_x^2f(x)$ and
\begin{equation*}
\Phi_{\partial_x^2 f}(b+2\delta)\leq \Phi_{\partial_x f}(b+\delta)\leq  	\Phi_f(b)=\frac{1}{\delta}.
\end{equation*}
Repeating this argument $k-2$ times, we obtain $b+k\delta\leq \mathrm{minroot}\ \partial_x^kf(x)$. 
Therefore, we have
\begin{equation}\label{eq:2025:xu86}
\mathrm{minroot}\ \partial_x^kf(x)\geq \max_{b<\beta_1}\ b+\frac{k}{\Phi_f(b)}.	
\end{equation}

It remains to derive a lower bound on the right hand side of \eqref{eq:2025:xu86}.
Since $1\leq \beta_1\leq\cdots\leq\beta_t$, we have
\begin{equation*}
\begin{aligned}
\Phi_f(b)
=\sum_{i=1}^{t}\frac{1}{\beta_i-b}
\leq  \sum_{i=1}^{k}\frac{1}{\beta_i- b}+\sum_{i=k+1}^t \frac{1}{\beta_{k+1}-b} 
=\frac{t-k}{\beta_{k+1}-b}+\sum_{i=1}^{k}g(\beta_i),
\end{aligned}
\end{equation*}
where $g(x):=\frac{1}{x-b}$.
Note that $g(x)$ is convex on $(b,+\infty)$. Therefore,  for each $1\leq i\leq k$, we have
\begin{equation*}
g({\beta_i})=g(s_i\cdot \beta_1+(1-s_i)\cdot {\beta_{k+1}})
\leq s_i\cdot g(\beta_1)+(1-s_i)\cdot g(\beta_{k+1}),
\end{equation*}
where $s_i:=\frac{\beta_{k+1}-\beta_i}{\beta_{k+1}-\beta_1}\in[0,1]$.
Then,
\begin{equation*}
\Phi_f(b) 
\leq  \frac{t-k}{\beta_{k+1}-b} +\sum_{i=1}^{k}\left( s_i\cdot g(\beta_1)+(1-s_i)\cdot g(\beta_{k+1})\right)
=k\cdot\Big(\frac{\alpha_{k}}{\beta_1-b}+\frac{\gamma}{\beta_{k+1}-b}\Big),
\end{equation*}
where $\alpha_{k}=\frac{\beta_{k+1}-\frac{1}{k}\sum_{i=1}^{k}\beta_i}{\beta_{k+1}-\beta_1}\in(0,1)$ and $\gamma=\frac{t}{k}-\alpha_{k}>0$.
Combining with \eqref{eq:2025:xu86}, we have
\begin{equation}\label{eq:2025:xu87}
\mathrm{minroot}\ \partial_x^kf(x)\geq \max_{b<\beta_1}\ h(b).
\end{equation}
Here,
\begin{equation*}
\begin{aligned}
h(b)
&:=b+\frac{1}{\frac{\alpha_{k}}{\beta_1-b}+\frac{\gamma}{\beta_{k+1}-b} }
=\beta_1+\frac{k^2}{t^2}\cdot \Big(c_0-(\frac{t-k}{k}\cdot \theta +\frac{c_1}{\theta })\Big),
\end{aligned}
\end{equation*}
where $\theta =\alpha_{k}\beta_{k+1}+\gamma\beta_1-\frac{t}{k}\cdot b$, $c_0=(\beta_{k+1}-\beta_1)((\frac{t}{k}-1)\alpha_{k}+\gamma)$ and
$c_1=(\beta_{k+1}-\beta_1)^2\alpha_{k}\gamma$.
Note that $c_1\geq 0$. Thus, by Cauchy inequality we have
\begin{equation*}
h(b)
\overset{(a)}\leq \beta_1+\frac{k^2}{t^2}\cdot \Big(c_0-2\sqrt{\frac{t-k}{k}\cdot c_1}\Big)
=\beta_1+\frac{k}{t}\cdot (\beta_{k+1}-\beta_1)\cdot 
\bigg(\sqrt{1-\frac{k}{t}\cdot \alpha_{k}}-\sqrt{\alpha_{k}-\frac{k}{t}\cdot \alpha_{k}}	\bigg)^2,
\end{equation*}
where the equality in ($a$) holds if and only if $\theta=\sqrt{\frac{kc_1}{t-k}}$, i.e.,
\begin{equation*}
b=\beta_1-\frac{(1-\alpha_{k})(\beta_{k+1}-\beta_1)}{\frac{t}{k}-1+\sqrt{(\frac{t}{k}-1)\frac{\gamma}{\alpha_{k}}}}	
<\beta_1.
\end{equation*}
Combining with \eqref{eq:2025:xu87}, we obtain
\begin{equation*}
\begin{aligned}
\mathrm{minroot}\ \partial_x^kf(x)
\geq \max_{b<\beta_1}\ h(b)
&=\beta_1+\frac{k}{t}\cdot (\beta_{k+1}-\beta_1)\cdot 
\bigg(\sqrt{1-\frac{k}{t}\cdot \alpha_{k}}-\sqrt{\alpha_{k}-\frac{k}{t}\cdot \alpha_{k}}	\bigg)^2\\
&= (1-c_k)\cdot \beta_1+c_k\cdot \beta_{k+1}.
\end{aligned}
\end{equation*}
which implies \eqref{eq:2025:93}. 

We now proceed to prove \eqref{eq:2025:92}. Note that
\begin{equation*}
1
>\sqrt{1-\frac{k}{t}\cdot \alpha_{k}}-\sqrt{\alpha_{k}-\frac{k}{t}\cdot \alpha_{k}}
=\frac{1-\alpha_{k}}
{\sqrt{1-\frac{k}{t}\cdot \alpha_{k}}+\sqrt{\alpha_{k}-\frac{k}{t}\cdot \alpha_{k}}}	
\geq \frac{1-\alpha_{k}}{1+\sqrt{\alpha_{k}}}=1-\sqrt{\alpha_{k}}.
\end{equation*}
Thus, we can further derive
\begin{equation*}
\begin{aligned}
\mathrm{minroot}\ \partial_x^kf(x)
&\geq  \beta_1+\frac{k}{t}\cdot (\beta_{k+1}-\beta_1)\cdot 
(1-\sqrt{\alpha_{k}}	)^2\geq \frac{k}{t}\cdot 
(1-\sqrt{\alpha_{k}}	)^2 \cdot \beta_{k+1}.\\
\end{aligned}
\end{equation*}
This completes the proof.

\end{proof}

\end{document}